\documentclass[11pt]{amsart}
\usepackage{amsmath,amssymb,amsxtra,amsthm,amscd,cite}
\usepackage{bbm,epic,eepic,graphics}
\usepackage[margin=3cm]{geometry} 
\usepackage[hypertex]{hyperref}
\usepackage[all]{xy}

\numberwithin{equation}{section}

\newtheorem{theorem}{Theorem}[section]
\newtheorem{lemma}[theorem]{Lemma} 
\newtheorem{proposition}[theorem]{Proposition} 
\newtheorem{corollary}[theorem]{Corollary} 

\theoremstyle{definition}
\newtheorem{definition}[theorem]{Definition} 
\newtheorem{notation}[theorem]{Notation} 
\newtheorem{remark}[theorem]{Remark} 
\newtheorem{example}[theorem]{Example}

\newcommand{\C}{\mathbb{C}}
\newcommand{\Z}{\mathbb{Z}} 
\newcommand{\N}{\mathbb{N}}
\newcommand{\Q}{\mathbb{Q}} 
\newcommand{\R}{\mathbb{R}}
\newcommand{\G}{\mathbb{G}}
\newcommand{\T}{\mathbb{T}} 
\newcommand{\OO}{\mathbb{O}}
\newcommand{\PP}{\mathbb{P}} 
\newcommand{\LL}{\mathbb{L}} 
\newcommand{\K}{\mathbb{K}} 
\newcommand{\X}{\mathbb{X}}

\newcommand{\cO}{\mathcal{O}} 
\newcommand{\cU}{\mathcal{U}} 
\newcommand{\cE}{\mathcal{E}} 
\newcommand{\tcE}{\widetilde{\cE}}
\newcommand{\hcE}{\widehat{\cE}} 
\newcommand{\cL}{\mathcal{L}} 
\newcommand{\cH}{\mathcal{H}} 
\newcommand{\cD}{\mathcal{D}} 
\newcommand{\hcD}{{\widehat{\cD}}}
\newcommand{\cM}{\mathcal{M}}
\newcommand{\hcM}{\widehat{\cM}} 
\newcommand{\cP}{\mathcal{P}}
\newcommand{\cR}{\mathcal{R}} 
\newcommand{\cY}{\mathcal{Y}} 
\newcommand{\hcY}{\widehat{\cY}}

\newcommand{\hxi}{{\hat{\xi}}}
\newcommand{\heta}{\hat{\eta}}
\newcommand{\hA}{{\widehat{A}}}
\newcommand{\ha}{{\hat{a}}}

\newcommand{\tC}{\widetilde{C}} 
\newcommand{\tbeta}{\tilde{\beta}} 
\newcommand{\txi}{\tilde{\xi}} 
\newcommand{\tv}{\tilde{v}} 

\newcommand{\Qbar}{\overline{Q}} 
\newcommand{\ybar}{\overline{y}} 
\newcommand{\varthetabar}{\overline{\vartheta}} 
\newcommand{\partialbar}{\overline{\partial}}

\newcommand{\bN}{\mathbf{N}}
\newcommand{\hbN}{\widehat{\bN}} 
\newcommand{\bM}{\mathbf{M}}
\newcommand{\bk}{\mathbf{k}} 
\newcommand{\tbk}{\tilde{\bk}}
\newcommand{\bl}{{\boldsymbol{\ell}}}
\newcommand{\bt}{\mathbf{t}}
\newcommand{\be}{\mathbf{e}}
\newcommand{\bSigma}{\mathbf{\Sigma}} 
\newcommand{\Laa}{\boldsymbol{\Lambda}} 
\newcommand{\bR}{\mathbf{R}} 
\newcommand{\bU}{\mathbf{U}}

\newcommand{\fry}{\mathfrak{y}}
\newcommand{\frI}{\mathfrak{I}} 
\newcommand{\frX}{\mathfrak{X}} 
\newcommand{\frT}{\mathfrak{T}} 
\newcommand{\frakm}{\mathfrak{m}} 
\newcommand{\tfrakm}{\tilde{\frakm}}

\newcommand{\st}{{\rm st}}
\newcommand{\cl}{{\rm cl}}
\newcommand{\qu}{{\rm qu}}

\newcommand{\Spec}{\operatorname{Spec}} 
\newcommand{\Spf}{\operatorname{Spf}}
\newcommand{\ev}{\operatorname{ev}} 
\newcommand{\pt}{\operatorname{pt}} 
\newcommand{\Frac}{\operatorname{Frac}} 
\newcommand{\Lie}{\operatorname{Lie}} 
\newcommand{\Cone}{\operatorname{Cone}} 
\newcommand{\Hom}{\operatorname{Hom}} 
 
\newcommand{\End}{\operatorname{End}} 
\newcommand{\rank}{\operatorname{rank}}
\newcommand{\Sym}{\operatorname{Sym}} 
\newcommand{\Ker}{\operatorname{Ker}} 
\newcommand{\Cok}{\operatorname{Cok}} 
\newcommand{\Image}{\operatorname{Im}} 
\newcommand{\Pic}{\operatorname{Pic}} 
\newcommand{\Aut}{\operatorname{Aut}}
\newcommand{\Res}{\operatorname{Res}} 
\newcommand{\age}{\operatorname{age}}  
\newcommand{\diag}{\operatorname{diag}} 
\newcommand{\Bx}{\operatorname{Box}} 
\newcommand{\CR}{\operatorname{CR}} 
\newcommand{\inv}{\operatorname{inv}} 
\newcommand{\GM}{\operatorname{GM}}
\newcommand{\Loc}{\operatorname{Loc}} 
\newcommand{\hbigoplus}{\mathop{\widehat\bigoplus}} 
\newcommand{\Jac}{\operatorname{Jac}} 
\newcommand{\Crit}{\operatorname{Crit}} 
\newcommand{\Asym}{\operatorname{Asym}} 
\newcommand{\Gal}{\operatorname{Gal}} 
\newcommand{\ad}{\operatorname{ad}}

\DeclareMathOperator{\NE}{NE}
\DeclareMathOperator{\OEf}{OE}

\newcommand{\Gr}{\operatorname{Gr}}
\newcommand{\tGr}{\widetilde{\Gr}}
\newcommand{\gr}{\operatorname{gr}} 
\newcommand{\Mat}{\operatorname{Mat}} 
\newcommand{\GL}{\operatorname{GL}}
\newcommand{\LGL}{\operatorname{LGL}}

\newcommand{\iu}{\sqrt{-1}} 
\newcommand{\unit}{\boldsymbol{1}} 
\newcommand{\bunit}{\mathbbm{1}} 
\newcommand{\bunitbar}{\overline{\bunit}} 
\newcommand{\hbunit}{\hat{\bunit}}

\def\corr#1{\left\langle#1 \right\rangle} 
\def\parfrac#1#2{\frac{\partial #1}{\partial #2}} 
\def\floor#1{\lfloor #1 \rfloor}
\def\ceil#1{\lceil #1 \rceil}
\def\fract#1{\langle #1 \rangle} 

\begin{document} 

\title{Hodge-Theoretic Mirror Symmetry for Toric Stacks}

\author{Tom Coates}
\email{t.coates@imperial.ac.uk}
\address{Department of Mathematics\\ Imperial College London\\ 
180 Queen's Gate\\ London SW7 2AZ\\ United Kingdom}

\author{Alessio Corti}
\email{a.corti@imperial.ac.uk}
\address{Department of Mathematics\\ Imperial College London\\ 
180 Queen's Gate\\ London SW7 2AZ\\ United Kingdom}

\author{Hiroshi Iritani}
\email{iritani@math.kyoto-u.ac.jp}
\address{Department of Mathematics\\ Graduate School of Science\\ 
Kyoto University\\ Kitashirakawa-Oiwake-cho\\ Sakyo-ku\\ 
Kyoto\\ 606-8502\\ Japan}

\author{Hsian-Hua Tseng}
\email{hhtseng@math.ohio-state.edu}
\address{Department of Mathematics\\ Ohio State University\\ 
100 Math Tower, 231 West 18th Ave. \\ Columbus \\ OH 43210\\ USA}

\begin{abstract} 
Using the mirror theorem \cite{CCIT:mirrorthm}, 
we give a Landau--Ginzburg mirror description for the big equivariant 
quantum cohomology of toric Deligne--Mumford stacks. 
More precisely, we prove that the big equivariant quantum $D$-module 
of a toric Deligne--Mumford stack 
is isomorphic to the Saito structure associated to 
the mirror Landau--Ginzburg potential. 
We give a GKZ-style presentation of the 
quantum $D$-module, and a combinatorial description of quantum cohomology as a quantum Stanley--Reisner ring. We establish the convergence of the 
mirror isomorphism and of quantum cohomology in the 
big and equivariant setting. 
\end{abstract} 

\maketitle

\let\oldtocsection=\tocsection
\let\oldtocsubsection=\tocsubsection

\renewcommand{\tocsection}[2]{\hspace{0em}\oldtocsection{#1}{#2}}
\renewcommand{\tocsubsection}[2]{\hspace{1em}\oldtocsubsection{#1}{#2}}
{\tableofcontents } 

\section{Introduction} 
This paper is the last in a series of papers \cite{CCIT:mirrorthm, 
CCIT:applications} that study the genus-zero Gromov--Witten theory 
of toric Deligne--Mumford stacks.  Let $\frX$ be a toric Deligne--Mumford stack, or toric stack for short, that satisfies a mild semi-projectivity hypothesis (spelled out below). 
In \cite{CCIT:mirrorthm} we proved a mirror theorem that says 
that a certain hypergeometric function, called the $I$-function, lies on the 
Givental cone for $\frX$. This determines all
genus-zero Gromov--Witten invariants of~$\frX$. 
The present paper builds on this mirror theorem to establish Hodge-theoretic 
mirror symmetry for toric stacks in a very general setting -- without assuming that $\frX$ is compact, or imposing any positivity condition on $c_1(\frX)$.
We prove that the big and equivariant quantum cohomology $D$-module 
of $\frX$ can be described as the Saito structure of the 
Landau--Ginzburg model mirror to $\frX$. 

It has been proposed by Givental \cite{Givental:ICM} (see also \cite{Hori-Vafa}) 
that the mirror of a toric manifold $X$ is a Landau--Ginzburg model, 
or more precisely, a
Laurent polynomial function $F=F(x_1,\dots,x_n)$ with Newton polytope 
equal to the fan polytope of $X$. 
In particular, Givental~\cite{Givental:toric_mirrorthm} showed that, for weak Fano toric manifolds $X$,
oscillatory integrals $\int e^{F/z} \frac{dx_1\cdots dx_n}{x_1\cdots x_n}$ 
give solutions of the small quantum cohomology $D$-module of $X$.
His result also implies that the quantum cohomology ring of $X$ 
is isomorphic to the Jacobian ring of $F$, via an isomorphism which matches  
 the Poincar\'e pairing with the residue pairing. 
Givental-style mirror symmetry has been extended to 
big quantum cohomology by Barannikov, Douai--Sabbah, and Mann 
\cite{Barannikov:projective, Douai-Sabbah:II, Mann:wP}; 
this compares the Frobenius manifold structure 
\cite{Dubrovin:2DTFT}
defined by the big quantum cohomology of $X$ 
with K.~Saito's flat structure \cite{SaitoK:primitiveform, SaitoK:higherresidue} 
associated to a miniversal unfolding of $F$. 

Let us briefly review our main construction. 
Let $\frX$ be a toric Deligne--Mumford stack with 
semi-projective 
coarse moduli space.   (This means that the coarse moduli space is projective over affine and contains a torus-fixed point.)
We introduce an unfolding $F(x;y)$ of Givental's Landau--Ginzburg potential 
by choosing a finite subset $G$ in the fan lattice $\bN$: 
\[
F(x;y) = \sum_{i=1}^m y_i Q^{\lambda(b_i)} x^{b_i} + 
\sum_{\bk\in G} y_\bk Q^{\lambda(\bk)} x^{\bk} 
\]
where $b_1,\dots,b_m$ are generators of one-dimensional cones 
of the stacky fan of $\frX$, $Q$ is the Novikov variable, 
 $\lambda(b_i)$ and $\lambda(\bk)\in H_2(X,\Q)$ are certain curve classes, $x \in \Hom(\bN,\C^\times)$ is a torus co-ordinate, and 
$y_i$,~$y_\bk$ are deformation parameters. See \S \ref{subsec:unfolding} for details.
Generalizing the construction in \cite{Iritani:shift_mirror} to stacks, 
we introduce a \emph{formal} and \emph{logarithmic}  
Landau--Ginzburg model (see \S\ref{subsec:GM}) 
\[
\xymatrix@C=40pt{
\hcY \ar[r]^{F(x;y)} \ar[d] & \C \\
\hcM 
}
\]
where $\hcY \to \hcM$ is a degenerating family of affine toric 
varieties over the base $\hcM= \Spf \C[\![\Laa_+]\!] \times \Spf \C[\![y]\!]$ 
with $\C[\![\Laa_+]\!]$ the Novikov ring 
(i.e.~the completed semigroup ring of the monoid 
$\Laa_+\subset H_2(X,\Q)$ of effective curves). 
The spaces $\hcY$ and $\hcM$ have natural log structures 
defined by their toric boundaries. 
We then consider the logarithmic twisted de Rham complex
\[
\left( \Omega^\bullet_{\hcY/\hcM}\{z\}, z d + d F \wedge \right)
\]
and define the Gauss--Manin system $\GM(F)$ to be the top cohomology 
of this complex. 
In the equivariant case, we consider the potential 
$F_\chi = F - \sum_{i=1}^n \chi_i \log x_i$ in place of $F$,
where $\chi_i$ are torus-equivariant parameters. 
The equivariant Gauss--Manin system $\GM(F_\chi)$ is equipped with 
the Gauss--Manin connection $\nabla$, the grading operator $\Gr^{\rm B}$ 
and the higher 
residue pairing $P \colon \GM(F_\chi) \times \GM(F_\chi) 
\to S_\T[z][\![\Laa_+]\!][\![y]\!]$. We call the quadruple 
\begin{equation}
\label{eq:Saito} 
\left( \GM(F_\chi), \nabla, \Gr^{\rm B}, P \right) 
\end{equation} 
the \emph{Saito structure} associated with the Landau--Ginzburg model. 
\begin{theorem}[see Theorems \ref{thm:mirror_isom}, 
\ref{thm:pairings_match} for the details]
\label{thm:mirror_isom_introd} 
The Saito structure \eqref{eq:Saito} is isomorphic to 
the big and equivariant quantum connection of 
the toric stack $\frX$ together with the Poincar\'e pairing, 
under the identification of the base spaces given by a mirror map. 
\end{theorem} 

One of the important aspects in our construction is that the 
Landau--Ginzburg model is partially compactified across the large radius limit point $Q=0$; 
the Gauss--Manin connection then has logarithmic singularities at $Q=0$. 
The choice of a partial compactification is subtle 
when $\frX$ is a toric stack, rather than a toric manifold, because the family 
$\hcY \to \hcM$ then carries an additional \emph{Galois symmetry} 
of $\Pic^\st(\frX) := 
\Pic(\frX)/\Pic(X)$, where $X$ is the coarse moduli space of $\frX$. 
The Galois symmetry has stabilizers along the compactifying divisor, 
and the quotient family $\hcY/\Pic^\st(\frX) \to \hcM/\Pic^\st(\frX)$ 
gives a partial compactification of the traditional mirror family\footnote
{The correct mirror family should be thought of as a formal stack 
$[\hcY/\Pic^\st(\frX)] \to [\hcM/\Pic^\st(\frX)]$.}. 
Our construction gives a generalization of the work 
of de Gregorio--Mann \cite{deGregorio-Mann} who studied 
the Jacobian ring at the limit $Q=0$ for mirrors of 
weighted projective spaces 
(see also \cite{DKK:symplectomorphism} for a related 
partial compactification). 
The new ingredient for us is the 
\emph{refined fan sequence} \eqref{eq:refined_fanseq} 
for stacky fans. 

Our results yield a combinatorial description 
for the quantum $D$-module of toric stacks 
which is closely related to the better-behaved 
GKZ system of Borisov--Horja \cite{Borisov-Horja:bbGKZ}. 
We introduce a \emph{fan $D$-module} for a stacky fan 
(Definition \ref{def:fanD-mod}) 
and show that the quantum $D$-module of the corresponding toric stack $\frX$
is isomorphic to the $(Q,y)$-adic completion of the fan $D$-module
 (see Theorem \ref{thm:quantum_D-mod_presentation}). 
By taking the semiclassical limit $z\to 0$ of the $D$-module, 
we obtain a quantum Stanley--Reisner description 
of the big and equivariant quantum cohomology of $\frX$ as follows, generalizing 
the previous works \cite{Givental:fixedpoint_toric, Batyrev:qcoh_toric, 
Givental:toric_mirrorthm, McDuff-Tolman, Iritani:coLef, FOOO:toricI, 
BCS, Iritani:integral, Gonzalez-Woodward:tmmp}: 
\begin{theorem}[see Theorem \ref{thm:qcoh_presentation} for the details] 
The big and equivariant quantum cohomology of $\frX$ is isomorphic to the 
Jacobian ring of $F_\chi$ under the identification of parameters given by 
the mirror map. The latter ring is isomorphic to the space 
$\hbigoplus_{\bk \in \bN \cap |\Sigma|} \C[\![\Laa_+]\!][\![y]\!] \bunit_\bk$ 
equipped with the following product and 
 $H_\T^*(\pt,\C)$-module structure: 
\[
\bunit_{\bk} \star \bunit_{\bl} = Q^{d(\bk,\bl)} \bunit_{\bk+\bl},  
\qquad 
\chi = \sum_{i=1}^m (\chi \cdot b_i) y_i \bunit_{b_i} 
+ \sum_{\bl\in G} (\chi \cdot \bl) y_\bl \bunit_{\bl}
\] 
with $\chi \in H_\T^2(\pt) =\bN^\star\otimes \C$.
See \eqref{eq:d(,)} for the definition of $d(\bk,\bl)\in H_2(X,\Q)$. 
\end{theorem} 

In the last part of the paper, we discuss the convergence of the 
mirror map and the mirror isomorphism in Theorem \ref{thm:mirror_isom_introd}. 
Beyond the weak Fano case or small quantum cohomology, 
it was known \cite{Iritani:coLef} that the mirror isomorphism is 
not fully analytic: it is only defined over formal power series in 
$z$ in general. 
We prove a partial analyticity result for the mirror isomorphism 
which in turn shows that the big and equivariant 
quantum cohomology itself is convergent and analytic. 

\begin{theorem}[see Theorem \ref{thm:convergence}, 
Corollary \ref{cor:convergence} for the details] 
The mirror map and the isomorphism in Theorem \ref{thm:mirror_isom_introd}
satisfy the following:
\begin{itemize}
\item[(1)] The mirror map is analytic in $(Q,y,\chi)$. 

\item[(2)] With respect to a basis of the Gauss--Manin system $\GM(F_\chi)$ 
formed by polynomial differential forms, 
the mirror isomorphism is a formal power series in $z$ 
with coefficients in analytic functions in $(Q,y,\chi)$. 
\end{itemize} 
It follows that the structure constants of the big and equivariant quantum 
cohomology of semi-projective toric stacks are analytic functions in 
their arguments $(\tau,\chi)$. 
\end{theorem} 

The proof uses mirror symmetry in an essential way. 
We combine the fact that the formal asymptotic expansions 
of oscillatory integrals are Gevrey series of order 1 
with a gauge fixing result from \cite[Proposition 4.8]{Iritani:coLef}.

\begin{remark} 
Hodge theoretic mirror symmetry for toric varieties or stacks 
has been studied by many people.  We explain how our results fit with this earlier work.

(1) In singularity theory, our Gauss--Manin system has been studied 
for isolated hypersurface singularities under the name of 
\emph{Brieskorn lattice}. K.~Saito \cite{SaitoK:primitiveform} 
and M.~Saito \cite{SaitoM:Brieskorn} constructed flat (Frobenius manifold) 
structures on the base of miniversal deformations 
of isolated singularities. 
This was generalized to global singularities by 
Sabbah, Barannikov, and  Douai--Sabbah \cite{Sabbah:tame, Barannikov:projective, Douai-Sabbah:I, Douai-Sabbah:II}, 
and applications to mirror symmetry are discussed there.
See also Mann \cite{Mann:wP}. 

(2) More recently, Reichelt--Sevenheck \cite{Reichelt-Sevenheck:logFrob, 
Reichelt-Sevenheck:nonaffine} constructed nc-Hodge structures,
which roughly speaking correspond to the Saito structures here, 
for mirrors of weak Fano toric manifolds and discussed their relation to the GKZ system. 
They also used log structures on the mirror to define the twisted 
de Rham complex; more precisely, they put log structures along the toric boundary of 
each fiber of the mirror family $\cY \to \cM$, but not along $Q=0$. 
They described a logarithmic extension of the mirror $D$-module 
across $Q=0$ using a GKZ-style presentation. 
See also T.~Mochizuki \cite{Mochizuki_T:twistor_GKZ}. 

(3) Mirror symmetry for non-weak-Fano toric manifolds and its convergence 
were analysed by Iritani \cite{Iritani:coLef}. 
Fukaya--Oh--Ohta--Ono \cite{FOOO:toricI} gave a Jacobian description of the 
quantum cohomology of general toric manifolds using 
Lagrangian Floer theory (see Chan--Lau--Leung--Tseng \cite{CLLT:Seidel} 
for an explicit computation in the weak Fano case). 
Gross \cite{Gross:tropical_P2} constructed mirrors 
of the big quantum cohomology of $\PP^2$ by counting tropical discs. 

(4) A mirror theorem for weighted projective spaces 
was proved by Coates--Corti--Lee--Tseng \cite{CCLT:wp} and 
was generalized to toric stacks in our previous work \cite{CCIT:mirrorthm}; see Cheong--Ciocan-Fontanine--Kim \cite{Cheong-CF-Kim} 
for a more general result. 
Based on these works, Landau--Ginzburg mirror symmetry for 
small quantum $D$-modules 
was described by Iritani \cite{Iritani:integral} for weak Fano toric stacks, 
and by Douai--Mann \cite{Douai-Mann:wp} for weighted projective spaces. 
Iritani \cite{Iritani:integral} also described the natural integral 
structure on the mirror in terms of the Gamma class (see also \cite{KKP}); 
a missing piece in the present work is the identification of the integral 
(or rational, or real) structure 
for mirrors of general toric stacks. 

(5) Gonz\'alez--Woodward \cite{Gonzalez-Woodward:tmmp} used gauged 
Gromov--Witten theory and quantum Kirwan maps 
to give a Jacobian description for quantum cohomology of toric stacks. 

(6) After we finished a draft of this paper, we learned that 
Mann--Reichelt \cite{Mann-Reichelt} studied closely related 
logarithmic degenerations of mirrors of weak-Fano toric orbifolds along $Q=0$. 
They used an extended version of the refined fan sequence 
in the case where $\bN$ has no torsion 
(see \cite[equation 2.17]{Mann-Reichelt})  
and obtained a logarithmic extension of the mirror $D$-module 
via a GKZ-style presentation 
(see \cite[Defintion 4.9, Theorem 6.6]{Mann-Reichelt} 
and Remark \ref{rem:GKZ}). 
\end{remark}

\begin{remark}
It should be possible to construct the mirror map and the 
mirror isomorphism for toric stacks 
via the Seidel representation \cite{Seidel:pi1} and 
shift operators \cite{BMO:Springer}, 
as \cite{Gonzalez-Iritani:Selecta,Iritani:shift_mirror} did for toric manifolds. 
\end{remark} 

\medskip
\noindent 
{\bf Acknowledgements.}  
We thank Thomas Reichelt for his interest in our paper and 
helpful comments.  
T.C.~was supported in part by a Royal Society University Research Fellowship,
ERC Starting Investigator Grant number~240123, and the Leverhulme Trust.  
H.I.~was supported in part by 
EPSRC grant EP/E022162/1 and 
JSPS Kakenhi Grant Number 16K05127, 16H06337, 
25400069, 26610008, 23224002, 25400104, 
22740042. 
H.-H.T.~was supported in part by 
NSF grant DMS-1506551 and a Simons Foundation Collaboration Grant.

\section{Toric stacks} 
\label{sec:toric_stacks}
In this section, we establish notation for 
toric Deligne--Mumford stacks (toric stacks for short) in the sense of 
Borisov, Chen and Smith \cite{BCS}. For the basics on 
toric stacks or varieties, we refer the reader to 
\cite{BCS,FMN,Iwanari1, Iwanari2, CLS}. 
\subsection{Definition} 
\label{subsec:definition} 
A \emph{stacky fan} \cite{BCS} is a triple 
$\bSigma = (\bN, \Sigma, \beta)$ consisting of 
\begin{itemize} 
\item a finitely generated abelian group $\bN$ of rank $n$; 

\item a rational simplicial fan $\Sigma$  
in the vector space $\bN_\R = \bN\otimes_\Z \R$; 

\item a homomorphism $\beta \colon \Z^m \to \bN$ 
such that $\{\R_{\ge 0}b_1, \dots, \R_{\ge 0}b_m\}$ 
is the set of one-dimensional cones of $\Sigma$, 
where $b_i = \beta(e_i)$ is the image of the 
$i$th basis vector $e_i\in \Z^m$. 
\end{itemize} 
Abusing notation, we shall identify a cone $\sigma$ of $\Sigma$ 
with the subset $\{i : b_i \in \sigma\}$ 
of $\{1,\dots,m\}$. For instance, we write 
$I\in \Sigma$ for $I\subset \{1,\dots,m\}$ 
if the cone spanned by $\{b_i : i\in I\}$ 
belongs to $\Sigma$, and we write $i\in \sigma$ 
for a cone $\sigma \in \Sigma$ if $b_i \in \sigma$. 
Define 
\[
\cU_\Sigma := \C^m \setminus 
\bigcup_{\{1,\dots,m\} \setminus I \notin \Sigma} 
 \C^I,  
\]
where $\C^I = \{(Z_1,\dots,Z_m) \in \C^m : 
\text{$Z_i = 0$ for $i\notin I$}\}$.  
Define the group $\G$ by 
\[
\G := H^{-1}(\Cone(\beta) \otimes^\LL \C^\times).  
\]
This is isomorphic to the product of the algebraic torus $(\C^\times)^{m-n}$ and 
a finite group. The group $\G$ acts on $\C^m$ via 
the connecting homomorphism $H^{-1}(\Cone(\beta)\otimes^\LL \C^\times) 
\to H^0(\Z^m\otimes \C^\times) = (\C^\times)^m$. 
A \emph{toric Deligne--Mumford 
stack} $\frX$  associated to 
the stacky fan $\bSigma$ is defined \cite{BCS} to be the quotient stack 
\begin{equation*}
\frX := \left[\cU_\Sigma/\G   
\right ].  
\end{equation*} 
We assume that 
\begin{itemize}
\item $\Sigma$ contains a cone of maximal dimension $n=\dim \bN_\R$;  
\item the support $|\Sigma|$ of the fan $\Sigma$ is convex;  
\item the fan admits a strictly convex piecewise linear 
function $f\colon |\Sigma| \to \R$ which is linear on 
each cone. 
\end{itemize}
These assumptions are equivalent to the condition 
that the coarse moduli space $X$ of $\frX$ 
is semi-projective \cite{CLS}, that is, $X$ is projective over an affine 
variety and has a torus fixed point.  
We set 
\[
\LL := \Ker(\beta), \qquad 
\bM := \Hom(\bN,\Z). 
\]
By definition, $\LL$ is the lattice of relations 
among $b_1,\dots,b_m$.  
The \emph{fan sequence} is the exact sequence:  
\begin{equation}
\label{eq:fanseq} 
\begin{CD}
0@>>> \LL @>>> \Z^m @>{\beta}>> \bN 
\end{CD} 
\end{equation} 
and the \emph{divisor sequence} is its Gale dual \cite{BCS}:  
\begin{equation} 
\label{eq:divseq} 
\begin{CD}
0 @>>> \bM @>{\beta^\star}>> (\Z^m)^\star @>{D}>> \LL^\vee, 
\end{CD} 
\end{equation} 
where $\LL^\vee := H^1(\Cone(\beta)^\star) 
\cong \Hom(\G,\C^\times)$ and the map
$D \colon (\Z^m)^\star \to \LL^\vee$ is induced by the 
natural map 
$(\Z^m)^\star \to (\Cone(\beta)[-1])^\star 
= \Cone(\beta)^\star[1]$. 
The ordinary dual $\LL^\star = \Hom(\LL,\Z)$ 
can be identified with the torsion-free part $\LL^\vee/(\LL^\vee)_{\rm tor}$ 
of $\LL^\vee$ and the torsion part of $\LL^\vee$ 
is given by $(\LL^\vee)_{\rm tor} = \Hom(\Cok(\beta),\C^\times)$. 
Note that $\Cok(\beta) \cong \pi_0(\G)$ is isomorphic 
to the orbifold fundamental group of $\frX$. 
The torsion part $\bN_{\rm tor}$ of $\bN$ is isomorphic to the generic stabilizer  
$\Ker(\G \to (\C^\times)^m)$ of $\frX$. 
We write $u_i = e_i^\star\in (\Z^m)^\star$ for the $i$th basis vector 
and $D_i := D(u_i)$ for the image of $u_i$ by $D$. 

\begin{notation}  
By the subscripts $\Q$, $\R$, $\C$,
we mean the tensor products 
with $\Q$, $\R$, $\C$ (over $\Z$), 
e.g. $\bN_\R = \bN\otimes \R$, $\LL_\Q = \LL\otimes \Q$. 
For an element $\bk\in \bN$, 
we denote by $\overline{\bk}$ 
the image of $\bk$ in $\bN_\Q$ (or in $\bN_\R$). 
By abuse of notation, we write  
$\bN \cap |\Sigma| := \{\bk\in \bN : 
\overline{\bk} \in |\Sigma|\}$. 
\end{notation} 

\subsection{Torus action and divisor sequence} 
\label{subsec:torus_action} 
The $(\C^\times)^m$-action on $\cU_\Sigma \subset \C^m$ 
naturally induces the action of the Picard stack 
$\frT = [(\C^\times)^m/\G]$ on $\frX$ \cite{FMN}. 
A line bundle on $\frX$ corresponds 
to a $\G$-equivariant line bundle on $\cU_\Sigma$, 
which is determined by a character of $\G$. 
Similarly, a $\frT$-equivariant line bundle on $\frX$ 
corresponds to a $(\C^\times)^m$-equivariant line 
bundle on $\cU_\Sigma$, which is determined by 
a character of $(\C^\times)^m$. 
Therefore we have the following natural identifications: 
\begin{align*}
\Pic(\frX) & 
\cong \Hom(\G,\C^\times) = \LL^\vee, \\ 
\Pic^\frT(\frX) & 
\cong \Hom((\C^\times)^m, \C^\times) = (\Z^m)^\star.  
\end{align*}  
The natural map $\Pic^\frT(\frX) \to \Pic(\frX)$ can be 
identified with the map $D \colon (\Z^m)^\star \to \LL^\vee$ 
in the divisor sequence \eqref{eq:divseq}. 
Let $\T:= (\C^\times)^m/ \Image(\G \to (\C^\times)^m) 
\cong \bN\otimes \C^\times$ 
be the coarse moduli space\footnote 
{There exists an exact sequence of Picard stacks: 
$1 \to BN_{\rm tor} \to  
\frT \to \T \to 1$ and this sequence splits: 
$\frT \cong BN_{\rm tor} \times \T$; see \cite[Proposition 2.5]{FMN}. 
Thus the $\T$-action on $X$ lifts to 
a $\T$-action on $\frX$. } of $\frT$. 
The torus $\T$ acts 
on the coarse moduli space $X$ of $\frX$. 
Taking first Chern classes of (equivariant) line bundles, we obtain 
the following canonical identifications over $\Q$: 
\begin{align}
\label{eq:description_H2}  
\begin{split}
H^2(X,\Q) & \cong \LL^\star_\Q \\
H^2_\T(X,\Q) &  \cong (\Q^m)^\star \\
H^2_\T({\rm pt},\Q) & \cong \bM_\Q 
\end{split} 
\end{align} 
such that the divisor sequence \eqref{eq:divseq} over $\Q$ 
is identified with  
\begin{equation*}
\label{eq:eqcohseq}
\begin{CD}
0 @>>> H^2_\T({\rm pt},\Q) 
@>>> H^2_\T(X,\Q)  
@>>> H^2(X,\Q) @>>> 0. 
\end{CD}
\end{equation*} 
Via the identification \eqref{eq:description_H2}, 
we regard $u_i = e_i^\star \in (\Q^m)^\star$, 
$D_i \in \LL_\Q^\star$  
as (equivariant or non-equivariant) cohomology classes. 
These are the (equivariant or non-equivariant)  
Poincar\'{e} duals of the toric divisor $[\{Z_i= 0\}/\G] \subset [\cU_\Sigma/\G]$, 
where $Z_i$ is the $i$th co-ordinate on $\C^m$. 

\subsection{Inertia stack, Box and orbifold cohomology} 
\label{subsec:inertia} 
Recall that the \emph{inertia stack} $I\frX$ is defined to 
be the fiber product $\frX \times_{\Delta, \frX \times \frX, \Delta} \frX$ 
of the diagonal morphisms $\Delta \colon \frX \to \frX \times \frX$. 
A point of $I\frX$ is given by 
a pair $(x,g)$ of a point $x\in \frX$ and 
an automorphism $g \in \Aut(x)$.  There is a map $\inv \colon I \frX \to I\frX$ that sends
$(x,g)$ to $(x,g^{-1})$.
Borisov, Chen and Smith \cite[Lemma~4.6]{BCS} showed that 
connected components of the inertia stack $I\frX$ 
are indexed by the set $\Bx$: 
\begin{equation*}
\Bx := \bigcup_{\sigma \in \Sigma} \Bx(\sigma), 
\quad 
\Bx(\sigma) := \left\{ 
v \in \bN\; \Bigg|\; \overline{v} \in \sigma, \; 
\overline{v} = \sum_{i\in \sigma} c_i \overline{b}_i, \; 
c_i\in [0,1) \right\}. 
\end{equation*} 
Let $\frX_v$ denote the component of $I\frX$ corresponding 
to $v\in \Bx$; then we have $I \frX = \bigsqcup_{v\in\Bx} \frX_v$. 
The component $\frX_v$ is isomorphic to the closed toric substack 
of $\frX$ associated with the minimal cone $\sigma(v)\in \Sigma$ 
containing $\overline{v}\in \bN_\R$. 
See the proof of Lemma~\ref{lem:age} below for the description 
of the stabilizer $g_v\in \G$ along $\frX_v$. 

\begin{notation} 
\label{nota:Psi}
We introduce a function $\Psi \colon \bN \cap |\Sigma| 
\to (\Q_{\ge 0})^m$ by 
\begin{equation*}
\Psi(\bk) = (\Psi_i(\bk))_{1\le i\le m}, 
\qquad \Psi_i(\bk) := 
\begin{cases} 
c_i  & i \in \sigma ; \\
0    & i \in \{1,\dots,m\}\setminus \sigma, 
\end{cases}   
\end{equation*}  
where $\bk \in \bN \cap |\Sigma|$, $\sigma\in \Sigma$ 
is the minimal cone containing $\overline{\bk}$ and we write 
$\overline{\bk} = \sum_{i\in \sigma} c_i \overline{b}_i$. 
The age function $|\cdot| \colon \bN \cap |\Sigma| 
\to \Q_{\ge 0}$ is defined to be 
$|\bk| = \sum_{i=1}^m \Psi_i(\bk)$. 
\end{notation} 

Let $X_v$ denote the coarse moduli space of $\frX_v$. 
The \emph{orbifold cohomology group} 
\cite{Chen-Ruan:new_coh} of $\frX$ is defined 
to be 
\begin{equation*}
\label{eq:orbcoh} 
H_{\CR}^*(\frX) := \bigoplus_{v\in \Bx} 
H^{* - 2 |v|}(X_v, \C).  
\end{equation*} 
The \emph{$\T$-equivariant orbifold cohomology group} is 
defined similarly: 
\[
H^*_{\CR, \T}(\frX) := 
\bigoplus_{v\in \Bx} H^{*-2|v|}_{\T}(X_v,\C). 
\] 
Chen and Ruan \cite{Chen-Ruan:new_coh} introduced a super-commutative 
product structure on orbifold cohomology, called the \emph{Chen--Ruan 
cup product}. For toric stacks, the Chen--Ruan product is commutative.
The orbifold cohomology ring of the toric stack $\frX$ has been 
computed by Borisov--Chen--Smith \cite{BCS} in the complete case, 
Jiang--Tseng \cite{Jiang-Tseng} in the semi-projective case, 
and by Liu \cite{CCLiu:localization} in the equivariant case.  
The $\T$-equivariant orbifold cohomology ring is: 
\[
H^*_{\CR,\T}(\frX) \cong \bigoplus_{\bk\in \bN\cap |\Sigma|} 
\C \phi_\bk
\]
where the product structure is given by: 
\begin{equation} 
\label{eq:equivariant_CR_product}
\phi_{\bk_1} \cdot \phi_{\bk_2} = 
\begin{cases} 
\phi_{\bk_1 + \bk_2}  & \text{if $\overline{\bk}_1,
\overline{\bk}_2$ lie in the same cone 
of $\Sigma$}; \\ 
0 &\text{otherwise} 
\end{cases} 
\end{equation} 
and the $R_\T := H^*_\T(\pt,\C) = \Sym^*(\bM_\C)$-module structure 
is given by $
\chi \mapsto \sum_{i=1}^m (\chi\cdot b_i) \phi_{b_i}$ 
for $\chi \in \bM_\C$. 
The non-equivariant orbifold cohomology ring is the quotient 
of $H_{\CR,\T}^*(\frX)$ by the ideal generated by 
equivariant parameters $\chi\in \bM_\C$, i.e.~
\[
H^*_{\CR}(\frX) \cong \frac{\bigoplus_{\bk\in \bN\cap |\Sigma|} 
\C \phi_{\bk}}{
\left \langle \sum_{i=1}^m (\chi\cdot b_i) \phi_{b_i} : \chi\in \bM_\C 
\right \rangle}. 
\]
For a box element $v \in \Bx$, $\phi_v$ represents 
the identity class $\unit_v\in H^0_\T(X_v)$ 
supported on the component $X_v$. 
The element $\phi_{b_i}$, $1\le i\le m$, represents the class $u_i$ 
(or $D_i$ in the non-equivariant case) 
of a toric divisor; see \S \ref{subsec:torus_action}. 
In particular, we have 
\[
\phi_{\bk} = \left( \prod_{i=1}^m u_i^{\floor{\Psi_i(\bk)}} \right) \unit_v 
\]
for $v = \bk - \sum_{i=1}^m \floor{\Psi_i(\bk)} b_i \in \Bx$. 

\begin{remark} 
Since the odd cohomology of $X_v$ vanishes, the Serre spectral sequence 
for $X_v \times_\T E\T \to B\T$ degenerates over $\Q$ at the $E_2$ term, and 
we find that $H^*_{\CR,\T}(\frX)$ is a free module over $R_\T 
= H^*_\T(\pt)$ of rank $\dim H^*_{\CR}(\frX)$. 
\end{remark} 

\subsection{Refined fan sequence} 
\label{subsec:refined_fanseq}
We introduce an overlattice $\Laa$ of $\LL \cong 
\Hom(\Pic(\frX),\Z)$ which contains all curve classes in $\frX$. 
This overlattice $\Laa$ fits 
into a refined version of the fan sequence \eqref{eq:fanseq}. 
Let $\OO$ be the subgroup of $\Q^m \oplus\bN$ given by 
\[
\OO := \sum_{\bk \in \bN \cap |\Sigma|} \Z (\Psi(\bk),\bk). 
\]
Note that $\OO$ is a subgroup of  
$\{(\lambda,\bk) \in \Q^m \oplus \bN : \beta(\lambda) = 
\overline{\bk}\}$. 
We define: 
\[
\Laa := \{(\lambda,0) \in \OO \} \subset \Q^m.  
\]
The lattice $\Laa$ is a subgroup of $\LL_\Q\subset \Q^m$, and contains $\LL$. 
For $\bk_1,\bk_2 \in \bN\cap |\Sigma|$, 
define an element $d(\bk_1,\bk_2) \in \Q^m$ by 
\begin{equation}
\label{eq:d(,)} 
d(\bk_1,\bk_2) := \Psi(\bk_1) + \Psi(\bk_2) - \Psi(\bk_1+\bk_2).  
\end{equation} 
It is easy to see that $d(\bk_1,\bk_2) \in \Laa$, and moreover 
that $\Laa$ is generated by these classes:  
\[
\Laa = \sum_{\bk_1,\bk_2\in \bN\cap |\Sigma|} \Z d(\bk_1,\bk_2).  
\]
We obtain the following extension of the fan sequence \eqref{eq:fanseq}: 
\begin{equation}
\label{eq:refined_fanseq} 
\begin{aligned} 
\xymatrix{ 
0 \ar[r]  & \LL \ar[r]\ar@{^{(}->}[d] & \Z^m \ar[r] 
\ar@{^{(}->}[d]  
& \bN 
\ar@{=}[d] \\ 
0 \ar[r] & \Laa \ar[r] & \OO \ar[r] & \bN \ar[r] & 0 
} 
\end{aligned}
\end{equation} 
where the map $\OO \to \bN$ is projection to the second factor and the map $\Z^m \to \OO$ is given by 
sending $e_i\in \Z^m$ to $(e_i,b_i)\in \OO$. 
We call the exact sequence in the second row 
the \emph{refined fan sequence}. 
Note that the torsion part $\OO_{\rm tor}$ of $\OO$ 
is isomorphic to $\bN_{\rm tor}$ under projection, and hence 
the refined fan sequence splits. 

\begin{example} \label{ex:P112_refined_fan_sequence}
The fan sequence and the 
refined fan sequence for the toric stack $\frX = \PP(1,1,2)$ are:
\begin{align*} 
\xymatrix{
0 \ar[r] & \Z \ar@{_{(}->}[d]
\ar[r]^{ {}^t ( 1\  1\  2)}  & 
\Z^3 \ar[rr]^{(b_1\ b_2\ b_3)} 
\ar@{_{(}->}[d]
&  &  \Z^2 \ar@{=}[d]  \ar[r] & 0 \\  
0 \ar[r] & \frac{1}{2} \Z \ar[r] 
&  
\Z^3 + \Z {\tiny \begin{pmatrix} 1/2 \\ 1/2 \\ 0 \end{pmatrix}}
\ar[rr]
& & \Z^2 \ar[r] & 0 
} 
\end{align*} 
where $\bN = \Z^2$, $\LL = \Z$, $\Laa = \frac{1}{2} \Z$, 
$b_1 ={}^t(1,0)$, $b_2 = {}^t (-1,2)$, 
$b_3 = {}^t (0,-1)$ and 
$\OO = \Z^3 + \Z\, {}^t (1/2,1/2,0)$. 
\end{example}

\begin{example} \label{ex:Bmu2_refined_fan_sequence}
The refined fan sequence for the toric stack $\frX = B \mu_2$ is
\[
\xymatrix{
  0 \ar[r] & 0 \ar[r] & \Z/2\Z \ar[r]^{\cong} & \Z/2\Z \ar[r] & 0
}
\] 
where $\bN = \Z/2\Z$, $\OO = \Z/2\Z$ and $\Laa = 0$. 
\end{example}

\begin{remark} 
The set of degrees in $H_2(X,\Q)\cong \LL_\Q$ 
of stable maps to $\frX$ is generated by 
representable toric morphisms $f\colon \PP^1_{r_1,r_2} \to \frX$,
where $\PP^1_{r_1,r_2}$ denotes the one-dimensional 
toric stack with coarse moduli space equal to $\PP^1$, 
isotropy groups $\mu_{r_1}$, 
$\mu_{r_2}$ at $0$, $\infty$ respectively, 
and no other isotropy groups. These toric morphisms are classified in 
\cite[\S 3.5]{CCIT:mirrorthm}.
It is easy to see 
that the degrees $l(c,\sigma,j)$ of such toric morphisms 
given in \cite[Definition 12, Remark 13]{CCIT:mirrorthm} 
can be written as $d(\bk_1,\bk_2)$ for some $\bk_1, \bk_2 
\in \bN\cap |\Sigma|$ 
lying in two maximal cones meeting along a codimension one face. 
Therefore $\Laa$ contains all homology classes of stable maps. 
We will prove in Lemma \ref{lem:age_Laa}(2) below that $\Laa$ is 
the dual lattice of the Picard group of the coarse moduli space. 
On the other hand, $\OO$ should correspond to the group generated by 
classes in $H_2(X,L;\Q)\cong \Q^m$ of orbi-discs with boundaries in a Lagrangian 
torus orbit $L\subset X$. The notation $\OO$ is intended to 
mean degrees of ``open''  curves. 
\end{remark} 

\subsection{Mori cone and associated monoids}
\label{subsec:Mori_cone}
For a cone $\sigma \in \Sigma$, we define the 
following cones: 
\begin{align} 
\label{eq:cone_C} 
\begin{split} 
\tC_\sigma & := \left\{\lambda \in \R^m: \beta(\lambda) \in \sigma, 
\text{$\lambda_i \ge 0$ for $i\notin \sigma$} \right\}, \\ 
C_\sigma & := \LL_\R \cap \tC_\sigma 
= \left\{ \lambda \in \LL_\R : 
\text{$D_i \cdot \lambda \ge 0$  
for $i\notin \sigma$} \right\}. 
\end{split}  
\end{align} 
Note that $D_i\cdot \lambda$ is the $i$th component of $\lambda 
\in \LL_\R$ regarded as an element of $\R^m$. 
We also define 
\begin{align*} 
\OEf(\frX) & := \sum_{\sigma\in \Sigma} \tC_\sigma \subset \R^m, \\ 
\NE(\frX) & := \sum_{\sigma \in \Sigma} C_\sigma \subset \LL_\R. 
\end{align*} 
Under the identification $\LL_\R \cong H_2(X,\R)$ 
from \eqref{eq:description_H2}, $\NE(\frX)$ 
corresponds to the \emph{Mori cone}, i.e.~
the cone generated by effective curves.  
The cone $\OEf(\frX)$ should be its open analogue. 
We define: 
\begin{align*} 
\Laa_+ & := \Laa \cap \NE(\frX), \\ 
\OO_+ & := \left\{(\lambda,\bk) \in \OO : \lambda \in \OEf(\frX)\right\}.  
\end{align*} 

\begin{lemma} 
\label{lem:OO_+}
Projection to the second factor defines a map $\OO_+ \to \bN\cap |\Sigma|$. 
The fiber of this map at $\bk \in \bN\cap |\Sigma|$ equals $(\Psi(\bk),\bk) + 
\Laa_+$. In particular, the fiber at $0\in \bN \cap |\Sigma|$, which is 
$\Laa \cap \OEf(\frX)$, equals $\Laa_+$. 
\end{lemma} 
\begin{proof} 
It is clear from the definition that the projection to the second factor of an element in 
$\OO_+$ lies in $\bN \cap |\Sigma|$. Also it is clear that 
$(\Psi(\bk),\bk) + \Laa_+$ is contained in the fiber at $\bk\in \bN\cap |\Sigma|$. 
Let $(\lambda,\bk)$ lie in $\OO_+$. 
We have $(\lambda, \bk) = (\lambda -\Psi(\bk),0) + (\Psi(\bk),\bk)$. 
We want to show that $\lambda -\Psi(\bk) \in \Laa_+$. 
It suffices to show that $\lambda -\Psi(\bk) \in \NE(\frX)$, 
since it is clear that $\lambda -\Psi(\bk)$ lies in $\Laa$.  
By assumption there exist $\lambda_\sigma 
\in \tC_\sigma$ for $\sigma\in \Sigma$ such that 
$\lambda = \sum_{\sigma\in \Sigma} \lambda_\sigma$.  
We have that $\overline{\bk}_\sigma = \beta(\lambda_\sigma) 
\in \sigma$ and $\overline{\bk} = \beta(\lambda)$. Then 
\[
\lambda -\Psi(\bk) = \sum_{\sigma\in \Sigma} 
(\lambda_\sigma - \Psi(\overline{\bk}_\sigma) ) 
+ \sum_{\sigma \in \Sigma} \Psi(\overline{\bk}_\sigma) 
- \Psi(\bk).  
\] 
Here $\lambda_\sigma -\Psi(\overline{\bk}_\sigma)$ lies 
in $C_\sigma$ and 
$\sum_{\sigma \in \Sigma} \Psi(\overline{\bk}_\sigma) 
- \Psi(\bk)$ lies in $C_\tau$ for a cone $\tau$ containing 
$\overline{\bk}$. It follows that $\lambda -\Psi(\bk)$ 
lies in $\sum_{\sigma \in \Sigma} C_\sigma = \NE(\frX)$. 
\end{proof} 

\begin{remark} 
The group ring $\C[\OO_+]$ of $\OO_+$ can be viewed as an 
equivariant and orbifold generalization 
of Batyrev's quantum ring \cite{Givental:fixedpoint_toric, Batyrev:qcoh_toric}. 
See Theorem \ref{thm:qcoh_presentation} 
below for its relation to quantum cohomology. 
\end{remark}

\section{Toric Gromov--Witten theory} 
In this section, we review Gromov--Witten 
invariants, quantum cohomology, the quantum connection,
and the Givental cone. 
Most of the arguments apply to semi-projective smooth 
Deligne--Mumford stacks equipped with $\T$-action;
we restrict ourselves, however,
to the toric Deligne--Mumford stacks $\frX$ from \S \ref{sec:toric_stacks}. 
We also recall the main result of
our previous paper \cite{CCIT:mirrorthm}. 

\subsection{Gromov--Witten invariants} 
Gromov--Witten theory for symplectic orbifolds or smooth Deligne--Mumford stacks 
has been developed by Chen--Ruan \cite{Chen-Ruan:orbGW} and 
Abramovich--Graber--Vistoli \cite{AGV:GW}. 
We refer the reader to \cite{Chen-Ruan:orbGW,AGV:GW, 
Tseng:QRR} for a detailed discussion.
For $l\ge 0$ and $d\in H_2(X,\Z)$, 
we denote by $\frX_{0,l,d}$ the moduli stack of genus-zero 
twisted stable maps to $\frX$ of degree $d$ 
(this is denoted by $\mathcal{K}_{0,l}(\frX,d)$ in 
\cite{AGV:GW}). 
Note that $\frX_{0,l,d}$ is empty if $l \le 2$ and $d=0$. 
There are evaluation maps
$\ev_i \colon \frX_{0,l,d} \to \overline{I}\frX$, $i=1,\dots,l$, 
to the rigidified cyclotomic inertia stack $\overline{I}\frX$ 
\cite[\S 3.4]{AGV:GW}. 
For cohomology classes $\alpha_1,\dots,\alpha_l \in H^*_{\CR,\T}(\frX)$ 
and non-negative integers $k_1,\dots,k_l$, 
\emph{equivariant Gromov--Witten invariants} 
are defined to be the $\T$-equivariant integrals 
\[
\corr{\alpha_1 \psi^{k_1}, \dots, \alpha_l \psi^{k_l}}_{0,l,d}
= \int_{[\frX_{0,l,d}]^{\rm vir}}^\T \prod_{i=1}^l 
\psi_i^{k_i} \ev_i^\star(\alpha_i) 
\] 
where $[\frX_{0,l,d}]^{\rm vir}$ is the (equivariant) virtual fundamental 
class \cite{AGV:GW,Tseng:QRR} 
and $\psi_i$ is the first Chern class of the $i$th universal cotangent 
line bundle over $\frX_{0,l,d}$. 
We note that: 
\begin{itemize} 
\item since the underlying 
complex analytic spaces of the rigidified inertia stack 
and the inertia stack are the same, we can pull back the 
cohomology classes $\alpha_i$ via $\ev_i$;  
\item when $\frX$ is non-compact, we can define the right-hand 
side by the Atiyah--Bott-style virtual localization formula \cite{Graber-Pandharipande, 
CCLiu:localization}.  In this case the integral takes values in the fraction field $S_\T := \Frac(R_\T)$ 
of $R_\T = H^*_\T(\pt,\C)$. 
\end{itemize} 
A special case of Gromov--Witten invariants yields the \emph{orbifold 
Poincar\'e pairing}: it is defined by  
\[
(\alpha,\beta) := \corr{1, \alpha, \beta}_{0,3,0} 
= \int_{I\frX} \alpha \cup \inv^\star \beta 
\]
where the map $\inv$ was defined in \S\ref{subsec:inertia}. The pairing takes values in $R_\T$ or 
in $S_\T$, depending on whether $\frX$ is compact 
or non-compact.

\subsection{Quantum cohomology and the quantum connection} 
\label{subsec:qcoh} 
Recall the lattice $\Laa\subset H_2(X,\Q)$ 
and the monoid $\Laa_+$ of curve classes from 
\S\S \ref{subsec:refined_fanseq}--\ref{subsec:Mori_cone}. 
The quantum cohomology of $\frX$ is defined over 
the Novikov ring $\C[\![\Laa_+]\!]$, which is the 
completion of the group ring $\C[\Laa_+]$. 
We write $Q^d\in \C[\![\Laa_+]\!]$ 
for the element corresponding to $d \in \Laa_+$. 
The \emph{big equivariant quantum product} $\star$ is a formal family 
of commutative ring structures parametrized by 
$\tau \in H^*_{\CR,\T}(\frX)$, defined by 
\begin{equation}
\label{eq:qprod}
(\alpha \star \beta,\gamma)  = \sum_{l=0} ^\infty
\frac{1}{l!} 
\corr{\alpha,\beta,\gamma,\tau,\dots,\tau}_{0,l+3,d} Q^d 
\end{equation}
where $\alpha,\beta,\gamma \in H^*_{\CR,\T}(\frX)$. 
Since the orbifold Poincar\'e pairing is non-degenerate, 
this uniquely defines $\alpha\star\beta$. 
Choose a homogeneous basis $\{T_i\}$ of $H^*_{\CR,\T}(\frX)$ 
over $R_\T$ and write $\tau = \sum_{i} \tau^i T_i$. 
We regard $\{\tau^i\}$ as co-ordinates on $H^*_{\CR,\T}(\frX)$. 
For a ring $K$, we write $K[\![\tau]\!] = K[\![\{\tau^i\}]\!]$ 
for the ring of formal power series in $\{\tau^i\}$. 
The big equivariant quantum product defines a commutative ring structure on 
\[
H^*_{\CR,\T}(\frX) \otimes_{R_\T} R_\T[\![\Laa_+]\!][\![\tau]\!]. 
\]
Note that we do not need the fraction field $S_\T$ even when 
$\frX$ is non-compact. This is because the evaluation map 
$\ev_i \colon \frX_{0,l,d} \to \overline{I}\frX$ is always proper, 
and we can define the quantum product in terms of the 
push-forward along the $\ev_3$. The properness of 
$\ev_i$ is ensured by the fact that $\frX$ is semi-projective. 
In particular, the non-equivariant quantum 
product is defined on 
\[
H^*_{\CR}(\frX) \otimes \C[\![\Laa_+]\!][\![\tau]\!] 
\] 
as the non-equivariant limit. 
We also note that the Chen--Ruan cup product is the limit 
of the quantum product as $Q\to 0$ and $\tau \to 0$. 

The quantum connection is a pencil of flat connections with pencil parameter $z$. 
For $\xi \in \LL^\star_\C \cong H^2(X,\C)$ and a ring $K$, we write 
$\xi Q\parfrac{}{Q}$ for the derivation of $K[\![\Laa_+]\!]$ 
such that $\xi Q\parfrac{}{Q}\cdot Q^d = 
(\xi \cdot d) Q^d$. We also fix a splitting $\LL^\star_\C 
\to (\C^m)^\star \cong H^2_\T(X,\C)$ (over $\C$) of the composition of the 
divisor sequence \eqref{eq:divseq} with projection 
to the free part $\LL^\vee \to \LL^\star$, and write 
$\hxi\in H^2_\T(X,\C)$ for the lift of $\xi\in \LL^\star_\C$ with 
respect to the splitting. 
The quantum connection is defined by 
\begin{align}
\label{eq:q_conn} 
\begin{split}  
\nabla_{\xi Q\parfrac{}{Q}} &= \xi Q\parfrac{}{Q} 
+ \frac{1}{z} (\hxi \star) \\
\nabla_{\parfrac{}{\tau^i}} & = \parfrac{}{\tau^i} 
+ \frac{1}{z} (T_i \star) 
\end{split} 
\end{align} 
for $\xi \in H^2(X,\C)$. These operators define maps 
\[
H^*_{\CR,\T}(\frX)\otimes_{R_\T}R_\T [z][\![\Laa_+]\!][\![\tau]\!] 
\to z^{-1} H^*_{\CR,\T}(\frX)\otimes_{R_\T} 
R_\T [z][\![\Laa_+]\!][\![\tau]\!]. 
\]
The quantum connection is flat, 
i.e.~$[\nabla_{\vec{v}},\nabla_{\vec{w}}] = \nabla_{[\vec{v},\vec{w}]}$.  
Quantum cohomology has a grading structure. 
Define the \emph{Euler vector field} by 
\begin{equation} 
\label{eq:Euler_A}
\cE^{\rm A} = c_1(\frX) Q\parfrac{}{Q} + \sum_{i} 
\left(1 -\tfrac{1}{2} \deg T_i \right) \tau^i \parfrac{}{\tau^i} 
+ \sum_{i=1}^n \chi_i \parfrac{}{\chi_i} 
\end{equation} 
where $c_1(\frX) = D_1 +\cdots + D_m \in \LL^\star$ is the 
first Chern class of $T\frX$, 
$\deg T_i$ means the age-shifted degree of $T_i$, 
and $(\chi_1,\dots,\chi_n)$ is a basis of $\bM_\Q = H^2_\T(\pt,\Q)$ 
(so that $R_\T = \C[\chi_1,\dots,\chi_n]$). 
$\cE^{\rm A}$~is a derivation of $R_\T[\![\Laa_+]\!][\![\tau]\!]$. 
We define $\Gr_0\in \End_\C(H^*_{\CR,\T}(\frX))$ by 
\begin{equation} 
\label{eq:Gr_0} 
\Gr_0(\alpha) = \frac{1}{2} (\deg \alpha) \alpha  
\end{equation} 
for a homogeneous element $\alpha \in H^*_{\CR,\T}(\frX)$. 
We note that $\Gr_0$ is not linear over $H^*_\T(\pt,\C) = R_\T$. 
The grading on $H^*_{\CR,\T}(\frX) \otimes_{R_\T} 
R_\T[z][\![\Laa_+]\!][\![\tau]\!]$ is defined by the 
 operator
\[
\Gr^{\rm A} := z\parfrac{}{z} + \cE^{\rm A} + \Gr_0 
\]
where $z\parfrac{}{z} + \cE^{\rm A}$ acts on the 
coefficient ring $R_\T[z][\![\Laa_+]\!][\![y]\!]$ 
and $\Gr_0$ acts on $H^*_{\CR,\T}(\frX)$.  That is, we have 
\begin{equation} 
\label{eq:action_vf+Gr0}
\Gr^{\rm A}(c \alpha) = 
\left(\left(z\parfrac{}{z} + \cE^{\rm A}\right)c \right) 
\alpha 
+ c \Gr_0 \alpha
\end{equation} 
for $c\in R_\T[z][\![\Laa_+]\!][\![y]\!]$  
and $\alpha \in H^*_{\CR,\T}(\frX)$. 
The grading structure is compatible with the quantum 
connection in the sense that 
\[
\left[\nabla_{\xi Q\parfrac{}{Q}}, \Gr^{\rm A}\right] =0, 
\quad 
\left[\nabla_{\parfrac{}{\tau^i}}, \Gr^{\rm A}\right] = 
(1-\tfrac{1}{2} \deg T_i) \nabla_{\parfrac{}{\tau^i}}. 
\]

There is a canonical fundamental solution for the 
quantum connection. 
Define $M(\tau,z) \in \End_{R_\T}(H^*_{\CR,\T}(\frX)) 
\otimes R_\T(\!(z^{-1})\!)[\![\Laa_+]\!][\![\tau]\!]$ by 
\[
M(\tau,z) \alpha = \alpha + \sum_{d\in \Laa_+} 
\sum_{l= 0}^\infty \sum_{i} 
\frac{Q^d}{l!} 
\corr{\alpha,\tau,\dots,\tau,\frac{T_i}{z-\psi}}_{0,l+2,d} T^i 
\]
where $\{T^i\}$ is the basis of 
$H_{\CR,\T}^*(\frX) \otimes_{R_\T} S_\T$ 
dual to $\{T_i\}$ with respect to the orbifold Poincar\'e 
pairing, so that $(T_i,T^j) = \delta_i^j$. 

\begin{proposition} 
\label{pro:fundsol} 
The fundamental solution $M(\tau,z)$ satisfies the following 
differential equations: 
\begin{align*} 
M(\tau,z) \nabla_{\parfrac{}{\tau^i}} \alpha &= 
\parfrac{}{\tau^i} M(\tau,z) \alpha\\ 
M(\tau,z) \nabla_{\xi  Q\parfrac{}{Q}} \alpha & 
= \left( \xi Q\parfrac{}{Q} + z^{-1} \hxi \right) M(\tau,z) \alpha \\ 
M(\tau,z) (\Gr^{\rm A} \alpha)  & = \Gr^{\rm A} (M(\tau,z) \alpha) 
\end{align*} 
for $\alpha \in H^*_{\CR,\T}(\frX) \otimes_{R_\T} R_\T[z][\![\Laa_+]\!]
[\![\tau]\!]$. Moreover, $M(\tau,z)$ preserves the Poincar\'e pairing in the 
sense that 
$(M(\tau,-z) \alpha,M(\tau,z) \beta) = (\alpha,\beta)$ 
for all $\alpha$,~$\beta\in H^*_{\CR,\T}(\frX)$.   
\end{proposition} 
\begin{proof} 
These properties are well-known; see \cite[Corollary 6.7]{Givental:equivariant}, 
\cite[Proposition 2]{Pandharipande:afterGivental}, 
\cite[\S 1]{Givental:elliptic}, 
\cite[Proposition 2.4]{CIJ}. The first equation follows from the topological recursion 
relations as explained in \cite{Pandharipande:afterGivental}, the 
second equation follows by combining the first one with the divisor equation 
(see e.g.~\cite[\S 2.6]{CIJ}), and the last one follows from the 
degree axiom for Gromov--Witten invariants. That  
$M(\tau,z)$ preserves the Poincar\'e pairing is shown in 
\cite[\S 1]{Givental:elliptic}, \cite[Proposition 2.4]{CIJ}.  
\end{proof} 
We define the $J$-function to be 
\begin{align} 
\label{eq:J-function}
\begin{split} 
J(\tau,z) & = z M(\tau,z) \unit \\ 
& = z \unit + \tau + \sum_{
\substack{d\in \Laa_+, l\ge 0 \\ (d,l) \neq (0,0)}} 
\frac{Q^d}{l!} 
\corr{\tau,\dots,\tau,\frac{T^i}{z-\psi}}_{0,l+1,d}
T_i 
\end{split} 
\end{align} 
where $\unit$ is the identity class supported on the non-twisted sector 
$\frX \subset I \frX$. 

\begin{remark} 
In the non-equivariant theory, we can introduce the connection 
in the $z$-direction by the formula: 
\begin{equation} 
\label{eq:conn_z}
\nabla_{z\parfrac{}{z}} = \Gr - \nabla_{\cE^{\rm A}} - \frac{n}{2}.  
\end{equation} 
\end{remark} 

\subsection{Galois symmetry} 
\label{subsec:Galois_A} 
We introduce the Galois symmetry of the equivariant 
quantum connection. 
This is an adaptation of \cite[Proposition 2.3]{Iritani:integral} 
to our setting. 
The \emph{age} of a line bundle $L\to\frX$ 
along the twisted sector $\frX_v\subset I\frX$, $v\in \Bx$, 
is defined to be the rational number $f = \age_v(L)\in [0,1)$ such that 
the stabilizer $g_v$ along $\frX_v$ acts on fibers of 
$L|_{\frX_v}$ by $\exp(2\pi\iu f)$. 
Recall from \S\ref{subsec:torus_action} that $\Pic(\frX)\cong \LL^\vee$.  
For $\xi\in \LL^\vee$, we write $L_\xi$ for the line bundle 
corresponding to $\xi$. We define a linear map 
$g_0(\xi) \colon H^*_{\CR,\T}(\frX) \to H^*_{\CR,\T}(\frX)$ 
by 
\[
g_0(\xi) \left( \bigoplus_{v\in \Bx} \tau_v \right) =  \bigoplus_{v\in \Bx} 
e^{2\pi\iu \age_v(L_\xi)} \tau_v
\] 
with $\tau_v \in H^*_\T(X_v,\C)$. 
This map preserves the orbifold Poincar\'e pairing. 
Let $g(\xi)^*\colon \C[\![\Laa_+]\!][\![\tau]\!]
\to \C[\![\Laa_+]\!][\![\tau]\!]$ denote 
the action on the variables $(Q,\tau)$ 
given as the pull-back of the cohomology parameter $\tau$ by $g_0(\xi)$ and 
\[
g(\xi)^* Q^d = e^{-2\pi\iu\xi \cdot d} Q^d 
\]
with $d\in \Laa_+$. This defines a morphism  
$g(\xi) \colon \Spf \C[\![\Laa_+]\!][\![\tau]\!]\to 
\Spf \C[\![\Laa_+]\!][\![\tau]\!]$ which induces a map 
$g_0(\xi)$ on cohomology at $Q=0$. 
We call the maps $g_0(\xi)$, $g(\xi)^*$ the \emph{Galois action} of 
$\xi\in \LL^\vee$. 
Note that the Galois action descends to the action of 
the ``stacky'' Picard group (see Lemma \ref{lem:age_Laa}(2))
\[
\Pic^\st(\frX) := \Pic(\frX)/\Pic(X).  
\]
The following proposition can be proved using an argument 
similar to \cite[Proposition 2.3]{Iritani:integral}. 

\begin{proposition} 
\label{pro:Galois_A}
The quantum connection is equivariant under the Galois action in the 
sense that the map $g_0(\xi) \colon H^*_{\CR,\T}(\frX) 
\to H^*_{\CR,\T}(\frX)$ intertwines the 
quantum connection $\nabla$ with $g(\xi)^*\nabla$ 
for $\xi\in\LL^\vee$:
\begin{align*} 
g_0(\xi) \circ \nabla_{\eta Q\parfrac{}{Q}} \circ g_0(\xi)^{-1} 
& = (g(\xi)^* \nabla)_{\eta Q\parfrac{}{Q}} 
= \eta Q\parfrac{}{Q} + \frac{1}{z} g(\xi)^*(\heta\star) \\
g_0(\xi) \circ \nabla_{\parfrac{}{\tau^i}} \circ g_0(\xi)^{-1} 
& = (g(\xi)^*\nabla)_{\parfrac{}{\tau^i}} 
=  \parfrac{}{\tau^i} + \frac{1}{z} g(\xi)^*((g_0(\xi) T_i)\star). 
\end{align*} 
where $\eta\in \LL_\C^\star$. 
Moreover the fundamental solution $M(\tau,z)$ in 
Proposition \ref{pro:fundsol} satisfies 
$g(\xi)^* M(\tau,z) = g_0(\xi) M(\tau,z) g_0(\xi)^{-1}$. 
\end{proposition}

\subsection{Givental cone} 
\label{subsec:Givental_cone}
Givental's symplectic vector space \cite{Givental:quadratic, 
Givental:symplectic, CCIT:mirrorthm} 
in equivariant Gromov--Witten theory is 
\[
\cH = H^*_{\CR,\T}(\frX) \otimes_{R_\T} S_\T(\!(z^{-1})\!)[\![\Laa_+]\!] 
\] 
equipped with the symplectic form: 
\[
\Omega(f,g) = - \Res_{z=\infty} (f(-z),g(z)) dz.  
\]
The space $\cH$ has a standard polarization 
$\cH = \cH_+ \oplus \cH_-$, where 
\begin{align*} 
\cH_+ & = H^*_{\CR,\T}(\frX) \otimes S_\T[z][\![\Laa_+]\!], \\
\cH_- & = z^{-1} H^*_{\CR,\T}(\frX) \otimes S_\T[\![z^{-1}]\!][\![\Laa_+]\!] 
\end{align*} 
are $\Omega$-isotropic subspaces. We can identify $\cH$ 
with the contangent bundle $T^\star \cH_+$ via this polarization. 
The equivariant Givental cone $\cL_\frX \subset \cH$ 
is defined as the graph of the differential 
of the genus-zero descendant potential, defined in the formal 
neighbourhood of $-z \unit \in \cH_+$ 
\cite{Givental:symplectic, CCIT:mirrorthm}. 
Equivalently, we can describe it as the set of points in $\cH$ 
of the form: 
\begin{equation} 
\label{eq:point_on_the_cone}
{-z} \unit + \bt(z) + \sum_{l=0}^\infty \sum_{d\in \Laa_+} 
\sum_i  \frac{Q^d}{l!} 
\corr{\bt(\psi),\dots,\bt(\psi), \frac{T_i}{-z-\psi}}_{0,l+1,d} 
T^i 
\end{equation} 
where $\{T_i\}$,~$\{T^i\}$ are the mutually dual bases 
of $H_{\CR,\T}^*(\frX) \otimes_{R_\T} S_\T$ 
with respect to the orbifold Poincar\'e 
pairing, i.e.~$(T_i,T^j) = \delta_i^j$, and $\bt(z) \in \cH_+$. 
For example, the $J$-function $J(\tau,-z)$ 
\eqref{eq:J-function} is a point on $\cL_\frX$. 
A more precise definition of the notion of points on $\cL_\frX$ 
is as follows. 
Let $t=(t_1,\dots,t_N)$ be an arbitrary set of formal variables. 
An \emph{$S_\T[\![\Laa_+]\!][\![t]\!]$-valued point} on 
$\cL_\frX$ is a point of the form \eqref{eq:point_on_the_cone} 
with 
\[
\bt(z) \in \cH_+[\![t]\!], \quad 
\bt(z)|_{Q=0, t = 0} = 0.  
\]
The Givental cone is a cone -- i.e.~it is invariant under 
dilation in $\cH$, -- and has very special geometric 
properties, which are sometimes referred to as being ``over-ruled''. 
We refer the reader to \cite{Givental:symplectic} or 
\cite[Appendix B]{CCIT:computing} for details. 
In this paper, we need the following fact. 

\begin{proposition}[{\!\!\cite{Givental:symplectic}, 
\cite[Proposition B.4]{CCIT:computing}}]
\label{pro:tangentsp_cone}
The tangent space of $\cL_\frX$ at an 
$S_\T[\![\Laa_+]\!][\![t]\!]$-valued point  
is spanned over $S_\T[z][\![\Laa_+]\!][\![t]\!]$ by the 
derivatives of the $J$-function: 
\[
\partial_{\tau^i} J(\tau,-z) = M(\tau,-z) T_i 
\]
for some $\tau \in H_{\CR,\T}^*(\frX) \otimes_{R_\T}
S_\T[\![\Laa_+]\!][\![t]\!]$ 
with $\tau|_{Q=t =0}=0$. 
\end{proposition}

\subsection{Mirror theorem} 
We next review the mirror theorem from \cite{CCIT:mirrorthm}. 
We fix a finite subset 
$G \subset \bN \cap |\Sigma|$ in this section. 

\begin{definition} 
\label{def:K}
Let $\K^G_0$ denote the set of 
$\lambda =(\lambda_i, \lambda_\bk : 1\le i\le m, \bk\in G) 
\in \Q^m \times \Z^G$ 
such that $\sum_{i=1}^m \lambda_i \overline{b}_i + 
\sum_{\bk\in G} \lambda_\bk \overline{\bk} =0$ 
and that $\{ 1\le i\le m : \lambda_i \notin \Z\}\in \Sigma$. 
For $\lambda\in \K^G_0$, we define 
\[
v(\lambda) := \sum_{i=1}^m \ceil{\lambda_i} b_i + \sum_{\bk\in G} 
\lambda_\bk \bk. 
\]
Since $\overline{v(\lambda)} = \sum_{i=1}^m \fract{-\lambda_i} b_i$, 
$v(\lambda)$ belongs to $\Bx$. We also set 
\[
d(\lambda) := (\lambda_1,\dots,\lambda_m) 
+ \sum_{\bk\in G} \lambda_\bk \Psi(\bk) 
\in \LL_\Q\subset \Q^m 
\]
where $\Psi$ is given in Notation \ref{nota:Psi}. 
\end{definition} 
\begin{lemma} 
\label{lem:d(lambda)}
For $\lambda \in \K^G_0$, 
$d(\lambda)\in \Laa$ 
(see \S \ref{subsec:refined_fanseq} for $\Laa$). 
\end{lemma} 
\begin{proof} 
Using $\Psi(\overline{v(\lambda)}) = 
(\fract{-\lambda_1},\dots,\fract{-\lambda_m})$, we deduce 
\[
d(\lambda) := (\ceil{\lambda_1},\dots,\ceil{\lambda_m}) + 
\sum_{\bk\in G} \lambda_\bk \Psi(\bk) - \Psi(v(\lambda)). 
\]
We also have $\sum_{i=1}^m \ceil{\lambda_i} b_i + 
\sum_{\bk\in G} \lambda_\bk \bk - v(\lambda) =0$. 
Consequently, we have
\[
(d(\lambda),0) = \sum_{i=1}^m \ceil{\lambda_i}(\Psi(b_i),b_i) 
+ \sum_{\bk\in G} \lambda_{\bk} (\Psi(\bk),\bk) - (\Psi(v(\lambda)), 
v(\lambda)) 
\] 
in $\Q^m \times \bN$. The right-hand side belongs to $\OO$, and thus 
$d(\lambda)\in \Laa$. 
\end{proof} 

\begin{definition}[\!\!\cite{CCIT:mirrorthm}] 
\label{def:I} 
Let $G \subset \bN\cap |\Sigma|$ be a finite subset. 
The \emph{$G$-extended $I$-function} is the cohomology-valued 
power series
\[
I(Q,\fry,t,z) = ze^{\sum_{i=1}^m t_i u_i/z}
\sum_{\lambda \in \K^G_0} Q^{d(\lambda)} 
e^{t\cdot d(\lambda)}\fry^\lambda
\left(\prod_{i \in \{1,\dots,m\} \cup G} 
\frac{\prod_{c\le 0, \fract{c} = \fract{\lambda_i}} u_i + cz}
{\prod_{c\le \lambda_i, \fract{c} = \fract{\lambda_i}} u_i+cz} 
\right) 
\unit_{v(\lambda)}
\] 
where 
\begin{itemize} 
\item $t=(t_1,\dots,t_m)$ and $\fry =(\fry_\bk : \bk\in G)$ are parameters; 
\item $t \cdot d(\lambda) := \sum_{i=1}^m t_i (D_i \cdot d(\lambda))$ 
and $\fry^\lambda := \prod_{\bk\in G} \fry_\bk^{\lambda_\bk}$; 
\item for $1\le i\le m$, $u_i$ is the equivariant Poincar\'e dual 
of a toric divisor in \S \ref{subsec:torus_action}, and 
for $i\in G$ we set $u_i:=0$;   
\item $\unit_{v(\lambda)}$ is the identity class supported 
on the twisted sector $\frX_{v(\lambda)}$. 
\end{itemize}
The $G$-extended $I$-function belongs to 
$H^*_{\CR,\T}(\frX) \otimes_{R_\T} R_\T(\!(z^{-1})\!)[\![\Laa_+]\!]
[\![t,\fry]\!]$  -- see Lemma \ref{lem:Loc_well-defined} for a proof of this in a more 
general setting. 
\end{definition} 

\begin{theorem}[\!\!{\cite[Theorem 31]{CCIT:mirrorthm}}]
\label{thm:mirrorthm} 
The $G$-extended $I$-function $I(Q,\fry,t,z)$ is an  
$S_\T[\![\Laa_+]\!][\![\fry,t]\!]$-valued point on the 
equivariant Givental cone $\cL_\frX$ of the toric 
Deligne--Mumford stack $\frX$.  
\end{theorem} 

\section{The Gauss--Manin System and the Mirror Isomorphism} 
In this section, we introduce a (partially compactified) 
Landau--Ginzburg model that corresponds under mirror symmetry to the toric stack $\frX$,
and show that the quantum connection for $\frX$ is isomorphic to the Gauss--Manin system 
associated with the Landau--Ginzburg potential. 
The construction closely follows the one in \cite{Iritani:shift_mirror} 
for toric manifolds. 

\subsection{Landau--Ginzburg model} 
\label{subsec:LG}
Recall the refined fan sequence \eqref{eq:refined_fanseq}. 
Applying the exact functor $\Hom(-,\C^\times)$ to it, 
we obtain 
\[
\begin{CD} 
1 @>>> \Hom(\bN,\C^\times) @>>> \Hom(\OO,\C^\times) 
@>>> \Hom(\Laa,\C^\times) @>>> 1. 
\end{CD} 
\]
Note that $\Hom(\bN,\C^\times)$ is a disjoint union of $|\bN_{\rm tor}|$ 
copies of the algebraic torus $(\C^\times)^n$. 
The \emph{uncompactified Landau--Ginzburg model} is given by 
the smooth family of algebraic varieties $\Hom(\OO,\C^\times) \to 
\Hom(\Laa,\C^\times)$ 
equipped with the Landau--Ginzburg potential 
$f \colon \Hom(\OO,\C^\times) \to \C$:  
\[
f = w_1 + \cdots + w_m 
\]
where $w_i\in \C[\OO]$ is a function on $\Hom(\OO,\C^\times)$ 
given by the evaluation at $(e_i,b_i) \in \OO$. 

Next we introduce a partial compactification of the above construction 
using the cones and monoids from \S \ref{subsec:Mori_cone}. 
The \emph{partially compactified Landau--Ginzburg model} is given by 
the flat family of algebraic varieties: 
\[
\Spec \C[\OO_+]\to \Spec \C[\Laa_+] 
\] 
induced by the inclusion $\Laa_+ \to \OO_+$ of monoids, 
equipped with the potential function $f \colon \Spec \C[\OO_+] 
\to \C$ as above 
(since $(e_i,b_i) \in \OO_+$, $w_i$ extends to 
a function on $\Spec \C[\OO_+]$). 
When we refer to the Landau--Ginzburg model, we will 
mean the partially compactified one unless otherwise stated. 

We introduce a co-ordinate system on the Landau--Ginzburg model. 
We write $w^{(\lambda,\bk)} \in \C[\OO_+]$ for the element 
corresponding to $(\lambda,\bk) \in \OO_+$. 
Define functions $w_i$, $w_\bk$ for $\bk \in \bN\cap |\Sigma|$, and $Q^\lambda
\in \C[\OO_+]$ for  $\lambda \in \Laa_+$ by 
\begin{equation*} 
w_i := w^{(e_i,b_i)}, \quad 
w_\bk := w^{(\Psi(\bk),\bk)}, \quad 
Q^\lambda := w^{(\lambda,0)}.  
\end{equation*} 
See Notation \ref{nota:Psi} for $\Psi$. 
We have that $w_i =w_{b_i}$ and 
$w_\bk = (\prod_{i=1}^m w_i^{\floor{\Psi_i(\bk)}}) 
w_v$ for $v = \bk - \sum_{i=1}^m \floor{\Psi_i(\bk)}b_i \in \Bx$.  
By Lemma \ref{lem:OO_+}, we have 
\begin{equation} 
\label{eq:group_ring_OO_+} 
\C[\OO_+] = \bigoplus_{\bk \in \bN\cap |\Sigma|} \C[\Laa_+] w_\bk. 
\end{equation}
We choose a splitting $\varsigma \colon \bN \to \OO$ 
of the refined fan sequence \eqref{eq:refined_fanseq}. 
Using the splitting $\varsigma$, 
we let $x^\bk \in \C[\OO]$ with $\bk \in \bN\cap |\Sigma|$ 
denote the element corresponding to $\varsigma(\bk) \in \OO$. 
(Note that $\varsigma(\bk)$ may not lie in $\OO_+$). 
Then we have: 
\[
w_\bk = Q^{\lambda(\bk)} x^\bk,   
\qquad 
w_i = Q^{\lambda(b_i)} x^{b_i} 
\]
with $\bk \in \bN \cap |\Sigma|$, $\lambda(\bk) := (\Psi(\bk),\bk) - \varsigma(\bk) 
\in \Laa$. 
Finally, by choosing an isomorphism $\bN \cong \Z^n \times \bN_{\rm tor}$, we write 
$x^\bk = x_1^{k_1} \cdots x_n^{k_n} x^{\zeta}$ 
for $\bk = (k_1,\dots,k_n, \zeta)\in \bN$. 
\begin{remark} 
For a maximal cone $\sigma_0 \in\Sigma$, we can define 
a splitting $\varsigma \colon \bN \to \OO$ by the requirement 
that $\varsigma(b_i) = (e_i,b_i)$ for all $i\in \sigma_0$. 
For this choice of $\varsigma$, $\lambda(\bk) = (\Psi(\bk),\bk) - \varsigma(\bk)$ 
lies in $\Laa_+$. 
\end{remark}

With this choice of co-ordinates, we define the \emph{equivariant Landau--Ginzburg 
potential} to be the multi-valued function on $\Spec \C[\OO_+]$:
\[
f_\chi = w_1 + \cdots + w_m  - \sum_{i=1}^n \chi_i \log x_i 
\]
where $\chi = (\chi_1,\dots,\chi_n) \in \Lie(\T) = \bN_\C 
\cong \C^n$ are equivariant parameters. (We can also view 
$\chi_i$, $1\le i\le n$, as a basis of $\bM=\bN^\star$.) 
\begin{remark} 
\label{rem:dependence_splitting} 
The equivariant potential $f_\chi$ depends on the choice of splittings $\varsigma$ 
and $\bN/\bN_{\rm tor} \to \bN$.   If we choose different 
splittings then the equivariant potential is shifted by 
a term of the form $\sum_{i=1}^n \chi_i (\log Q^{d_i} 
+ \log w^{(0,\zeta_i)} )$ 
for some $d_i \in \Laa$, $\zeta_i \in \bN_{\rm tor}$. 
\end{remark}

\begin{example} 
Recall from Example~\ref{ex:P112_refined_fan_sequence} that the fan sequence and refined fan sequence for the toric stack 
$\frX = \PP(1,1,2)$ are:
\begin{align*} 
\xymatrix{
0 \ar[r] & \Z \ar@{_{(}->}[d]
\ar[r]^{ {}^t ( 1\  1\  2)}  & 
\Z^3 \ar[rr]^{(b_1\ b_2\ b_3)} 
\ar@{_{(}->}[d]
&  &  \Z^2 \ar@{=}[d]  \ar[r] & 0 \\  
0 \ar[r] & \frac{1}{2} \Z \ar[r] 
&  
\Z^3 + \Z {\tiny \begin{pmatrix} 1/2 \\ 1/2 \\ 0 \end{pmatrix}}
\ar[rr]
& & \Z^2 \ar[r] & 0 
} 
\end{align*} 
where $\bN = \Z^2$, $\LL = \Z$, $\Laa = \frac{1}{2} \Z$, 
$b_1 ={}^t(1,0)$, $b_2 = {}^t (-1,2)$, 
$b_3 = {}^t (0,-1)$ and 
$\OO = \Z^3 + \Z\, {}^t (1/2,1/2,0)$. 
When we construct a mirror Landau--Ginzburg model 
using the (unrefined) fan sequence 
as considered in \cite{Givental:ICM, Iritani:integral}, we obtain a family 
\[
\C^3 = \Spec \C[\Z_{\ge 0}^3] \to \C = \Spec \C[\LL_+] 
\qquad (w_1,w_2,w_3) 
\mapsto Q = w_1 w_2 w_3^2
\] 
equipped with the potential $f = w_1 + w_2 + w_3$, where 
we set $\LL_+ := \LL \cap \NE(\frX) \cong \Z_{\ge 0}$. 
When we pull back this family via the map 
$\Spec \C[\Laa_+] \to \Spec \C[\LL_+]$, $t \mapsto Q = t^2$, 
we obtain the family 
\[
\{(w_1,w_2,w_3,t) \in \C^4: w_1 w_2 w_3^2 = t^2 \} \to \C, \quad 
(w_1,w_2,w_3,t) \mapsto t 
\]
with non-normal total space. 
The Landau--Ginzburg model (based on the refined fan sequence) 
is given by the normalization of this: it is 
\[
\{(w_1,w_2,w_3, u)\in \C^4 : w_1w_2 = u^2 \} \to \C, \quad 
(w_1,w_2,w_3,u) \mapsto uw_3 
\]
where $t = u w_3$. This is the same as the family constructed by 
de Gregorio--Mann \cite{deGregorio-Mann}. 
\end{example} 

\begin{example} 
Recall from Example~\ref{ex:Bmu2_refined_fan_sequence} that the refined fan sequence for $\frX = B \mu_2$ is
\[
\xymatrix{
  0 \ar[r] & 0 \ar[r] & \Z/2\Z \ar[r]^{\cong} & \Z/2\Z \ar[r] & 0
}
\] 
where $\bN = \Z/2\Z$, $\OO = \Z/2\Z$ and $\Laa = 0$. 
Thus the Landau--Ginzburg model is the identity map 
from the two-point set to itself, equipped with the zero potential. 
\end{example} 

\subsection{An unfolding of the Landau--Ginzburg potential} 
\label{subsec:unfolding} 
We consider an unfolding of the Landau--Ginzburg potential 
given by choosing a finite set $G \subset \bN \cap |\Sigma|$. 
We assume that $G$ is disjoint from $\{b_1,\dots,b_m\}$ 
and set $S := \{b_1,\dots,b_m\} \cup G$. 
Introduce co-ordinates $y_\bk$ for each $\bk \in S$ 
and set $y = \{y_\bk : \bk \in S\}$; 
we sometimes write $y_i = y_{b_i}$ for $1\le i\le m$. 
Define  
\begin{align*} 
F(x;y) & := 
\sum_{\bk \in S} y_\bk w_\bk 
= \sum_{i=1}^m y_i w_i + 
\sum_{\bk \in G} y_\bk w_\bk  
= \sum_{\bk \in S} y_\bk Q^{\lambda(\bk)} x^{\bk}, \\
F_\chi(x;y) & := F(x;y) - \sum_{i=1}^n \chi_i \log x_i. 
\end{align*} 
We call $F(x;y)$, $F_\chi(x;y)$ the 
\emph{$G$-unfolded Landau--Ginzburg potentials}. 
The unfolding $F(x;y)$ is an element of $\C[\OO_+][y]$. 
Under the specialization 
\begin{align}
\label{eq:specialization_y}  
\begin{split} 
y_i &= 1   \qquad \text{for $1\le i\le m$} \\ 
y_\bk &= 0 \qquad \text{for $\bk\in G$} 
\end{split} 
\end{align} 
the potentials $F(x;y)$, $F_\chi(x;y)$ become, respectively, 
the original ones $f(x)$, $f_\chi(x)$. We refer to the 
shifted origin \eqref{eq:specialization_y} of $y$ as $y^*$. 

\begin{remark} 
\label{rem:redundancy} 
The deformation parameters $y_1,\dots,y_m$ for $F(x;y)$ 
are redundant in the sense that those deformations can be 
reduced to a deformation along $Q\in \Spec \C[\Laa_+]$ 
via a rescaling of the variables $x$. 
It is however convenient to keep $y_1,\dots,y_m$ 
as deformation parameters when we use 
$\C[\Laa_+]$ as a ground ring. 
See also Remark \ref{rem:redundancy_fanDmod}. 
\end{remark} 

\begin{remark} 
The paper \cite{Iritani:shift_mirror} introduced infinitely 
many deformation parameters $y_\bk$ for all $\bk\in \bN \cap |\Sigma|$; 
each $\bk$ corresponds to a basis element $\phi_\bk$ for $H^*_\T(X)$ over $\C$. 
This gives a natural identification between the deformation space of $F$ 
and equivariant cohomology. 
On the other hand, in the present paper, 
we restrict to finitely many deformation terms; 
this is compensated for by working over the ground ring 
$R_\T = H^*_\T(\pt,\C)$. 
\end{remark} 

\subsection{Galois action on the Landau--Ginzburg model} 
\label{subsec:Galois_LG}
Similarly to the Galois symmetry in quantum cohomology 
(\S \ref{subsec:Galois_A}), 
we can define an action of $\Pic^\st(\frX) =\Pic(\frX)/\Pic(X)$
on the Landau--Ginzburg model. 

Recall that $\Pic(\frX) \cong \LL^\vee$ and 
introduce a bilinear pairing 
\[
\age 
\colon \LL^\vee \times \OO \to \Q/\Z\cong [0,1)\cap \Q
\] 
as follows. We recall an explicit description of 
$\LL^\vee = H^1(\Cone(\beta)^\star)$ from \cite[\S 2]{BCS}. 
Choose a free resolution $0 \to K \overset{\iota}{\to} F \to \bN\to 0$ 
of $\bN$ and a lift $\tbeta \colon \Z^m \to F$ of $\beta$. 
Then $\Cone(\beta)$ is quasi-isomorphic to the complex 
$\iota \oplus \tbeta \colon K \oplus \Z^m \to F$ of free modules, 
and 
\[
\LL^\vee = \Cok\Big(\iota^\star \oplus \tbeta^\star \colon 
F^\star \to K^\star \oplus (\Z^m)^\star\Big).
\] 
Let $\xi\in \LL^\vee$ and $(\lambda,\bk) \in \OO$ be given. 
We choose a lift 
$\txi \in K^\star \oplus (\Z^m)^\star$ of $\xi$ 
and a lift $\tbk\in F$ of $\bk \in \bN$. 
Then $\tbk - \tbeta(\lambda) \in F_\Q$ lies in the kernel 
of $F_\Q \to \bN_\Q$ and thus lies in $K_\Q$. 
Hence we obtain the element $(\tbk-\tbeta(\lambda), \lambda)$ 
of $K_\Q \oplus \Q^m$. 
We define
\[
\age(\xi, (\lambda,\bk)) := \left( 
\txi \cdot (\tbk-\tbeta(\lambda), \lambda)  
\mod \Z \right) \in \Q/\Z. 
\]
It is easy to see that this is independent of the choices made 
and defines a bilinear form. 

\begin{lemma} 
\label{lem:age} 
{\rm (1)} Let $L_\xi$ denote the line bundle corresponding to 
$\xi \in \LL^\vee$. For $v\in \Bx$, $\age(\xi,(\Psi(v),v))$ 
equals the age $\age_v(L_\xi)$ of the line bundle $L_\xi$ along the  
sector $\frX_v$.   

{\rm (2)} $\age(\xi,(\lambda,\bk)) = 0$ for all $\xi \in \LL^\vee$ 
if and only if $(\lambda,\bk) \in \OO$ lies in the image of 
the inclusion $\Z^m \to \OO$, $e_i \mapsto (e_i,b_i)$. 

{\rm (3)} We have $\age(\xi,(\lambda,\bk))=0$ for all $(\lambda,\bk) \in \OO$, 
if and only if the line bundle $L_\xi$ corresponding to $\xi\in \LL^\vee$ 
is the pull-back of a line bundle on the coarse moduli space $X$. 
\end{lemma} 
\begin{proof} 
(1) Recall from \S \ref{subsec:torus_action} that 
the line bundle $L_\xi$ is given by the quotient of 
$\cU_\Sigma \times \C$ by the $\G$-action 
$g\cdot (Z,v) = (g\cdot Z, \xi(g) v)$, where 
we regard $\xi \in \LL^\vee \cong \Hom(\G,\C^\times)$ 
as a character of $\G$. 
On the other hand, the stabilizer $g_v \in \G$ associated 
with $v\in \Bx$ is defined as follows (see \cite[Lemma 4.6]{BCS}). 
We use the notation in the paragraph preceding this lemma. Choose 
a lift $\tv \in F$ of $v$. Then $\tv-\tbeta(\Psi(v))\in F_\Q$ lies 
in the kernel of $F_\Q \to \bN_\Q$, and thus 
$(\tv-\tbeta(\Psi(v)), \Psi(v))$ lies in 
$K_\Q \oplus \Q^m$. 
The stabilizer $g_v \in \G = H^{-1}(\Cone(\beta)\otimes^\LL \C^\times)$ 
is then given by
\[
g_v =e^{2 \pi \iu \left(\tv - \tbeta(\Psi(v)), \Psi(v) \right) } 
\in \Ker\Big((K\oplus \Z^m)\otimes \C^\times \to F\otimes \C^\times\Big).  
\]
Therefore $g_v$ acts on fibers of $L_\xi$ along $\frX_v$ 
by $\exp (2\pi\iu \age(\xi,(\Psi(v),v))$, and part (1) follows. 

(2) Suppose that $\age(\xi,(\lambda,\bk))=0$ for all $\xi \in \LL^\vee$. 
With notation as above, we have that $(\tbk-\tbeta(\lambda),\lambda) 
\in K \oplus \Z^m$. This implies that $\lambda \in \Z^m$ and 
$\beta(\lambda) = \bk$. Thus $(\lambda,\bk)
=\sum_{i=1}^m \lambda_i (e_i,b_i)$ lies in the image 
of $\Z^m$. 

(3) Note that $\OO$ is generated by $(\Psi(v),v)$ with 
$v\in \Bx$ and $(e_i,b_i)$, $i=1,\dots,m$. 
Therefore, by parts (1) and (2), we have that $\age(\xi,(\lambda,\bk))=0$ 
for all $(\lambda,\bk) \in \OO$ 
if and only if the age of $L_\xi$ along each twisted sector 
$\frX_v$ is 0. This happens if and only if $L_\xi$ is the pull-back of a line bundle 
from the coarse moduli space. 
\end{proof} 

The above lemma says that the bilinear pairing $\age(\cdot,\cdot)$ 
descends to a perfect pairing $\age \colon 
\Pic^{\st}(\frX) \times \OO/\Z^m \to \Q/\Z$. 
We define the action of $\Pic^\st(\frX)$ on $\C[\OO_+]$ 
by 
\[
\xi \cdot w^{(\lambda,\bk)} = e^{2\pi \iu \age(\xi,(\lambda,\bk))} 
w^{(\lambda,\bk)}
\]
for $\xi \in \Pic^\st(\frX)$ and $(\lambda,\bk)\in \OO_+$. 
This defines the $\Pic^\st(\frX)$-action on the total space $\Spec \C[\OO_+]$. 
The group $\Pic^\st(\frX)$ also acts on the unfolding parameters 
$y=\{y_i, y_\bk : 1\le i\le m, \bk \in G\}$ of the $G$-unfolded Landau--Ginzburg 
potential by 
\[
\xi \cdot y_i = y_i, \qquad 
\xi \cdot y_\bk = e^{-2\pi\iu \age(\xi,(\Psi(\bk),\bk))} y_\bk.  
\]
The $G$-unfolded potential $F(x;y)$ 
is invariant under the $\Pic^\st(\frX)$-action. 
The equivariant potential 
$F_\chi(x;y)$ is not $\Pic^\st(\frX)$-invariant, but the $\Pic^\st(\frX)$-action 
shifts $F_\chi(x;y)$ only by a (constant) linear form in $\chi_1,\dots,\chi_n$, 
hence its derivative in $x$ is $\Pic^\st(\frX)$-invariant. 

\begin{lemma}
\label{lem:age_Laa} 
{\rm (1)} For $\xi \in \LL^\vee$ and $\lambda \in \Laa$, we have 
$\age(\xi,\lambda) \equiv \xi \cdot \lambda \mod \Z$, 
where we regard $\lambda$ as an element of $\OO$ 
by the inclusion $\Laa \subset \OO$. 
In particular, $\xi \cdot Q^\lambda$ equals the Galois action 
$(g(\xi)^*)^{-1} Q^\lambda$ for the Novikov variables 
defined in \S\ref{subsec:Galois_A}. 

{\rm (2)} For $\xi \in \Pic(X)$ and $\lambda \in \Laa$, we have 
$\xi \cdot \lambda \in \Z$. Moreover, we have $\Pic(X) \cong \Laa^\star$. 
\end{lemma} 
\begin{proof} 
With notation as above, we have the exact sequence 
$0\to \LL \to K\oplus \Z^m \xrightarrow{\iota\oplus\tbeta} F$, 
where the map 
$\LL \to K\oplus \Z^m$ is given by $\lambda \mapsto 
(-\tbeta(\lambda), \lambda)$. 
Thus the natural pairing between $\LL^\vee=
\Cok(F^\star \to K^\star \oplus (\Z^m)^\star)$ and 
$\LL_\Q$ is given by $\xi \cdot \lambda = \txi \cdot 
(-\tbeta(\lambda),\lambda)$ where $\lambda\in \LL_\Q$ and $\txi\in K^\star \oplus (\Z^m)^\star$ 
is a representative of $\xi \in \LL^\vee$. 
Part (1) follows from this and the definition of $\age(\xi,\lambda)$. 
The first statement of part (2) follows from part (1) and Lemma \ref{lem:age}(3).  
This gives a natural map $\Pic(X) \to \Laa^\star$, which is injective 
since $\Pic(X)$ has no torsion \cite[Proposition 4.2.5]{CLS}. 
To show the second statement, 
we embed both $\Pic(X)$ and $\Laa^\star$ into $\Pic(\frX)\cong \LL^\vee$; 
the natural map $\Laa^\star \to \LL^\vee$ is given by taking the Gale dual 
of \eqref{eq:refined_fanseq}:
\[
\xymatrix{
0 \ar[r] & \bM \ar[r] & (\Z^m)^\star \ar[r] & \LL^\vee & \\ 
0\ar[r] & \bM \ar[r]\ar@{=}[u] & \OO^\star \ar[r]\ar@{^{(}->}[u] 
& \Laa^\star \ar[r] \ar@{^{(}->}[u]& 0 
}
\] 
Via these embeddings, 
we have $\Pic(X) \subset \Laa^\star \subset \LL^\vee$. 
To see the converse inclusion $\Laa^\star \subset \Pic(X)$ it suffices, 
in view of Lemma \ref{lem:age}(3),  
to show that an element $\xi \in \Laa^\star\subset \LL^\vee$ 
satisfies $\age(\xi, (\lambda,\bk))= 0$ for all $(\lambda,\bk)\in \OO$. 
Note that $\xi\in \Laa^\star$ comes from an element of $\OO^\star$ 
in the above diagram. By the definition of the age pairing, it follows 
easily that $\age(\xi,(\lambda,\bk))=0$. The conclusion follows. 
\end{proof} 

\subsection{The Gauss--Manin system} 
\label{subsec:GM} 
In this section we fix a subset $G \subset \bN \cap |\Sigma|$ 
disjoint from $\{b_1,\dots,b_m\}$, and construct a Gauss--Manin 
system associated to the $G$-unfolded Landau--Ginzburg 
potential (\S \ref{subsec:unfolding}). 
The Gauss--Manin system constitutes, together with the higher residue 
pairing introduced in \S\ref{sec:pairing}, the \emph{Saito structure} of the 
Landau--Ginzburg model. 

\subsubsection{Definition} 
We consider a formal completion of the total space 
$\Spec \C[\OO_+]$ along the fiber at $\{Q=0\}  
\in \Spec \C[\Laa_+]$. 
Let $K$ be a ring, and let $\frakm$ denote the ideal of 
$K[\Laa_+]$ generated by $Q^\lambda$ 
with $\lambda \in \Laa_+\setminus \{0\}$. 
Let $\tfrakm\subset K[\OO_+]$ be the ideal generated by 
$\frakm$. We set 
\begin{align*} 
K[\![\Laa_+]\!] &:= \text{the completion of $K[\Laa_+]$ 
with respect to the $\frakm$-adic topology,} \\ 
K\{\OO_+\} &:= \text{the completion of $K[\OO_+]$ with respect 
to the $\tfrakm$-adic topology}. 
\end{align*} 
Note that we have (cf.~\eqref{eq:group_ring_OO_+}) 
\begin{equation} 
\label{eq:completed_directsum}
K\{\OO_+\} = \hbigoplus_{\bk\in \bN\cap |\Sigma|} 
K[\![\Laa_+]\!] w_\bk  
\end{equation} 
where the right-hand side is the completed direct sum 
with respect to the $\frakm$-adic topology. 
We also write 
\begin{align} 
\label{eq:powerseries_y}
\begin{split} 
K[\![y]\!]
& = K[\![y_1-1,\dots,y_m-1,\{y_\bk : \bk \in G\}]\!] \\ 
K[\chi] & = K[\chi_1,\dots,\chi_n]. 
\end{split} 
\end{align} 
Note that we use the completion 
at the shifted origin \eqref{eq:specialization_y} of $y$. 
In this section we consider the 
\emph{formal} Landau--Ginzburg model 
given by the $G$-unfolded potential $F(x;y)$  
\[
\xymatrix@C=40pt{ 
\hcY \ar[r]^{F(x;y)}  \ar[d] &  \C \\ 
\hcM
}
\]
with $\hcY = \Spf \C\{\OO_+\}[\![y]\!]$ and 
$\hcM = \Spf \C[\![\Laa_+]\!][\![y]\!]$. 
We regard $\hcY$, $\hcM$ as formal log schemes 
whose log structures are given by their toric boundaries 
(see e.g.~ \cite[Ch.~3]{Gross:tropical_book}); 
the family $\hcY \to \hcM$ is then log smooth. 
The Gauss--Manin system is given as the top cohomology 
of the logarithmic twisted de Rham complex 
$(\Omega^\bullet_{\hcY/\hcM}\{z\}, z d + dF \wedge)$ 
where 
\[
\Omega^k_{\hcY/\hcM}\{z\} 
= \bigoplus_{i_1<\cdots < i_k} 
\C[z]\{\OO_+\}[\![y]\!] \frac{dx_{i_1} \cdots dx_{i_k}}{x_{i_1} \cdots x_{i_k}}.  
\]
In what follows, we give a more concrete definition. 
Introduce the relative differential forms\footnote
{The volume form $\omega$ is normalized so that the integral against the maximal 
compact subgroup of $\Hom(\bN,\C^\times)$ equals $(2\pi\iu)^n$. 
The factor $|\bN_{\rm tor}|^{-1}$ plays a role in \S\ref{sec:pairing}, when we 
show that the pairings match.} 
$\omega := |\bN_{\rm tor}|^{-1} 
\frac{dx_1}{x_1} \wedge \cdots \wedge \frac{dx_n}{x_n}$ 
and $\omega_i := \iota_{x_i \parfrac{}{x_i}} \omega$. 
\begin{definition} 
(1) The \emph{equivariant Gauss--Manin system}  
$\GM(F_\chi)$ 
is defined to be the cokernel of the map 
\[
z d + dF_\chi \wedge \colon 
\bigoplus_{i=1}^n \C[z]\{\OO_+\}[\![y]\!][\chi] \omega_i 
\to \C[z]\{\OO_+\}[\![y]\!][\chi] \omega,   
\]
where $d$ denotes the differential in the variable $x$ 
which is 
linear over $\C[z][\![\Laa_+]\!][\![y]\!][\chi]$; we define 
$d w_\bk = \sum_{i=1}^n k_i w^{\bk} \frac{dx_i}{x_i}$; 
and the map $zd + dF_\chi \wedge$ is given explicitly by 
\begin{equation} 
\label{eq:image_twdR_diff} 
(zd + dF_\chi\wedge) w_\bk \omega_i 
= \left( z k_i 
+ \sum_{\bl \in S} l_i y_\bl w_\bl
- \chi_i \right) w_\bk \omega
\end{equation} 
where $k_i$, $l_i$ denote the 
$i$th components of $\overline{\bk}$,~$\overline{\bl} \in \bN_\Q \cong \Q^n$ respectively. 

(2) The \emph{non-equivariant Gauss--Manin system} 
$\GM(F)$ is defined to be the cokernel of the map 
\[
zd + dF\wedge \colon \bigoplus_{i=1}^n 
\C[z]\{\OO_+\}[\![y]\!] \omega_i \to 
\C[z]\{\OO_+\}[\![y]\!] \omega. 
\]
Clearly we have $\GM(F) \cong \GM(F_\chi)/\sum_{i=1}^n \chi_i \GM(F_\chi)$. 
\end{definition} 

\begin{remark} 
The quantity \eqref{eq:image_twdR_diff} gives a relation in 
$\GM(F_\chi)$ which can be viewed as defining the action of 
$\chi_i$. In fact, it is easy to check that 
\begin{equation}
\label{eq:action_chi_i} 
\chi_i \colon w_\bk \omega \mapsto 
\left( zk_i  
+ \sum_{\bl\in S} l_i y_\bl w_\bl\right) w_\bk \omega, \quad 1\le i\le n 
\end{equation} 
defines commuting $\C[z][\![\Laa_+]\!][\![y]\!]$-module 
endomorphisms $\chi_i$ on $\C[z]\{\OO_+\}[\![y]\!]\omega$, and 
that there is a canonical isomorphism: 
\begin{equation} 
\label{eq:equiv_GM_decomp} 
\GM(F_\chi) \cong \C[z]\{\OO_+\}[\![y]\!] \omega
\cong \hbigoplus_{\bk \in \bN\cap |\Sigma|} 
\C[z][\![\Laa_+]\!][\![y]\!] w_\bk \omega. 
\end{equation} 
It is also easy to see that, under the action \eqref{eq:action_chi_i}, 
$\GM(F_\chi) = \C[z]\{\OO_+\}[\![y]\!]$ 
is a module over $R_\T[z][\![\Laa_+]\!][\![y]\!]$, 
where $R_\T = H^*_\T(\pt,\C) = \C[\chi]$. 
\end{remark} 

\subsubsection{The Gauss--Manin connection and the grading operator} 
The equivariant Gauss--Manin system is equipped 
with a natural flat connection $\nabla$, 
called the \emph{Gauss--Manin connection}, and a \emph{grading operator}. 
For $\xi \in \LL^\star_\C$, let $\xi Q \parfrac{}{Q}$ 
be the derivation of $K[\![\Laa_+]\!]$ such that 
$\xi Q\parfrac{}{Q}\cdot Q^\lambda = (\xi\cdot \lambda) Q^\lambda$. 
For a vector field $\vec{v}$ in the parameters $(Q,y)$, 
the connection operator $\nabla_{\vec{v}} \colon \GM(F_\chi) 
\to z^{-1} \GM(F_\chi)$ is defined by the formula: 
\[
\nabla_{\vec{v}} = \partial_{\vec{v}} + z^{-1} (\vec{v} F_\chi)  
\]
where we use the splitting from \S\ref{subsec:LG} and think of 
elements in $\GM(F_\chi) \cong \C[z]\{\OO_+\}[\![y]\!]$ as functions in 
$(z,x,Q,y)$. 
More precisely, we define the actions of $\nabla_{\xi Q\parfrac{}{Q}}$, 
$\nabla_{\parfrac{}{y_\bl}}$ 
on the topological $\C[z][\![\Laa_+]\!][\![y]\!]$-basis 
$w_\bk \omega$ of $\GM(F_\chi)$ 
-- see \eqref{eq:equiv_GM_decomp} -- by 
\begin{align}
\label{eq:GM_conn_explicit}  
\begin{split} 
\nabla_{\xi Q\parfrac{}{Q}} w_\bk \omega & := 
(\xi \cdot \lambda(\bk)) w_{\bk} \omega 
+\frac{1}{z} 
\left(\sum_{\bl \in S} (\xi \cdot \lambda(\bl)) y_\bl w_\bl
\right) w_\bk \omega \\ 
\nabla_{\parfrac{}{y_\bl}} w_\bk \omega & 
:= \frac{1}{z} w_\bl w_\bk \omega 
= \frac{1}{z} Q^{d(\bk,\bl)} w_{\bl+\bk} \omega, 
\end{split} 
\end{align} 
where $\bk\in \bN\cap |\Sigma|$,~$\bl \in S$,~$\xi\in \LL^\star_\C$, 
and extend them (continuously) to the whole of $\GM(F_\chi)$ by the Leibnitz rule. 
Recall that $\lambda(\bl)\in \Laa$ 
was introduced in \S\ref{subsec:LG} 
by choosing a splitting $\varsigma \colon \bN \to \OO$
and that $d(\bk,\bl)$ was defined in \eqref{eq:d(,)}. 
Define the \emph{Euler vector field} by 
\begin{equation} 
\label{eq:Euler_B} 
\cE^{\rm B} = c_1(\frX) Q \parfrac{}{Q} + 
\sum_{\bk \in S} (1 - |\bk|) y_\bk \parfrac{}{y_\bk}  
+ \sum_{i=1}^n \chi_i \parfrac{}{\chi_i} 
\end{equation} 
where $c_1(\frX) := D_1+ \cdots +D_m \in \LL^\star$ 
and $|\bk| = \sum_{i=1}^m \Psi_i(\bk)$ is the age function 
from Notation~\ref{nota:Psi}.
Define the grading operator 
$\Gr^{\rm B} \colon \GM(F_\chi) \to \GM(F_\chi)$ by requiring that 
\begin{align}
\label{eq:Gr_B} 
\begin{split} 
\Gr^{\rm B}(w_\bk \omega) & = |\bk| w_\bk \omega \\ 
\Gr^{\rm B}( c \Omega ) & =  
\left(\left( \textstyle z\parfrac{}{z} + \cE^{\rm B} \right) 
c \right) \Omega + c \Gr^{\rm B}(\Omega) 
\end{split} 
\end{align} 
for $c = c(z,Q,y) \in \C[z][\![\Laa_+]\!][\![y]\!]$ 
and $\Omega \in \GM(F_\chi)$. 
The following proposition is an immediate consequence of the definitions.

\begin{proposition} 
\label{pro:GM_conn_flat} 
The connection operators $\nabla_{\xi Q\parfrac{}{Q}}$ and  
$\nabla_{\parfrac{}{y_{\bk}}}$ are linear over $R_\T$, that is, they commute with the 
action \eqref{eq:action_chi_i} of equivariant parameters $\chi_i$.
Moreover:
\begin{align*} 
& \left[\nabla_{\xi Q\parfrac{}{Q}}, \nabla_{\eta Q \parfrac{}{Q}}\right] 
= \left[\nabla_{\xi Q\parfrac{}{Q}}, \nabla_{\parfrac{}{y_\bk}}\right] 
= \left[\nabla_{\parfrac{}{y_\bk}}, 
\nabla_{\parfrac{}{y_\bl}}\right] =0,  \\ 
& \left[\nabla_{\xi Q \parfrac{}{Q}}, \Gr^{\rm B} \right] = 0, \quad 
\left[\nabla_{\parfrac{}{y_\bk}}, \Gr^{\rm B} \right] 
= \nabla_{\left[\parfrac{}{y_\bk},\cE^{\rm B} \right]} 
= \nabla_{(1-|\bk|) \parfrac{}{y_\bk}}, \\ 
& \Gr^{\rm B}( c \Omega ) =  \left(\left( \textstyle z\parfrac{}{z} 
+  \cE^{\rm B} \right) 
c \right) \Omega + c \Gr^{\rm B}(\Omega)
\end{align*}    
where $\xi$,~$\eta\in \LL^\star_\C$; $\bk$,~$\bl\in S$; 
$c=c(z,\chi,Q,y) \in R_\T[z][\![\Laa_+]\!][\![y]\!]$; and 
$\Omega \in \GM(F_\chi)$. 
\end{proposition} 

\begin{remark} 
We can also work with a different Euler vector field 
$\tcE^{\rm B}$ and grading operator 
$\tGr^{\rm B}$ defined by  
\[
\tcE^{\rm B} = \sum_{\bk\in S} y_\bk \parfrac{}{y_\bk} 
+ \sum_{i=1}^n \chi_i\parfrac{}{\chi_i},   
\qquad \tGr^{\rm B}(w_\bk \omega) = 0
\] 
together with  
$\tGr^{\rm B}(c \Omega) = \left( \left( z\parfrac{}{z} + \tcE^{\rm B}
\right) c \right) \Omega + c \tGr^{\rm B}(\Omega)$ 
for $c\in \C[z][\![\Laa_+]\!][\![y]\!]$, $\Omega\in \GM(F_\chi)$. 
This choice of Euler vector field and the grading operator 
was made\footnote
{The Euler vector field in \cite{Iritani:shift_mirror} 
does not contain the term $\sum_{i=1}^n \chi_i\parfrac{}{\chi_i}$ 
since the $\chi_i$-direction is contained as part of the (infinite-dimensional) 
$y$-directions in the setting of \cite{Iritani:shift_mirror}.}
 in \cite{Iritani:shift_mirror}. It has the 
advantage that they are linear over $\C[\![\Laa_+]\!]$, whereas 
the current choice has the advantage that the origin $Q=0$, 
$y=y^*$  -- where $y^*$ is the shifted origin in 
\eqref{eq:specialization_y} -- is a fixed 
point of the Euler flow. 
The two choices are, however, essentially equivalent. We can easily 
check that 
\[
\Gr^{\rm B}- \tGr^{\rm B} 
= \nabla_{\cE^{\rm B} - \tcE^{\rm B}} + \frac{\kappa}{z}
\]
where $\kappa \in \bM_\Q$ is the equivariant parameter 
defined by $\kappa = (e_1^\star + \cdots + e_m^\star) \circ \overline{\varsigma} 
\colon \bN \to \Q$, and $\varsigma(\bk) = (\overline{\varsigma}(\bk),
\bk)$ with $\overline{\varsigma} \colon \bN \to \Q^m$. 
\end{remark} 

\begin{remark} 
By Remark \ref{rem:dependence_splitting}, a different choice 
of the splitting $\varsigma$ shifts $F_\chi(x;y)$ only by 
a quantity which does not depend on $x$, and thus 
the Gauss--Manin system $\GM(F_\chi)$ 
is independent of the choice of splitting 
as an $R_\T[z][\![\Laa_+]\!][\![y]\!]$-module.  
This shift of $F_\chi(x;y)$, however, depends on $Q$. 
Therefore the Gauss--Manin connection $\nabla_{\xi Q\parfrac{}{Q}}$ 
in the $Q$-direction depends on the choice of $\varsigma$; the difference 
of the connections $\nabla_{\xi Q\parfrac{}{Q}}$ corresponding 
to two splittings 
is the multiplication by $\frac{1}{z} \sum_{i=1}^n 
\chi_i (\xi \cdot d_i)$ for some $d_i \in \Laa$. 
\end{remark}

\begin{remark} 
The connection operators and grading operators 
descend to the \emph{non-equivariant} Gauss--Manin system. 
In the non-equivariant case, 
the operator $\nabla_{\xi Q\parfrac{}{Q}}$ does not 
depend on the choice of splitting $\varsigma$. 
Also we can introduce the connection $\nabla_{z\parfrac{}{z}}$ 
in the $z$-direction by the formula: 
\[
\nabla_{z\parfrac{}{z}} = \Gr^{\rm B} - \nabla_{\cE^{\rm B}} 
-\frac{n}{2}. 
\]
Compare with the quantum connection in the $z$-direction \eqref{eq:conn_z}.  
\end{remark} 

\subsubsection{Galois symmetry} 
Since $dF_\chi$ is invariant under the Galois action 
in \S\ref{subsec:Galois_LG}, the $\Pic^\st(\frX)$-action 
on $\C[z]\{\OO_+\}[\![y]\!]$ induces a $\Pic^\st(\frX)$-action 
on the equivariant Gauss--Manin system $\GM(F_\chi)$ 
such that $\omega$ is $\Pic^\st(\frX)$-invariant. 
It is easy to check that the Gauss--Manin connection satisfies 
\begin{align*} 
\nabla_{\eta Q\parfrac{}{Q}} (\xi \cdot \Omega) 
& = \xi \cdot \nabla_{\eta Q\parfrac{}{Q}} \Omega \\ 
\nabla_{\parfrac{}{y_\bl}} (\xi \cdot \Omega)
& = e^{-2\pi\iu \age(\xi,(\Psi(\bk),\bk))} 
\xi \cdot \nabla_{\parfrac{}{y_\bl}} \Omega
\end{align*} 
where $\xi \in \Pic^\st(\frX)$, $\eta \in \LL^\star_\C$, 
$\bl \in S$, $\Omega \in \GM(F_\chi)$.  
Moreover the Galois action commutes with the grading operator 
$\Gr^{\rm B}$.

\subsection{Solution and freeness} 
\label{subsec:sol_free}
We next describe a cohomology-valued solution (localization map) 
to the equivariant Gauss--Manin system and show that the 
equivariant Gauss--Manin system 
is free over $R_\T[z][\![\Laa_+]\!][\![y]\!]$.

\begin{definition}[cf.~Definition \ref{def:K}]
For $\bk\in \bN\cap |\Sigma|$, 
let $\K^G_\bk$ denote the set of $\lambda =(\lambda_i, \lambda_\bl : 1\le i\le m, 
\bl\in G) 
\in \Q^m \times \Z^G$ 
such that $\overline{\bk} + 
\sum_{i=1}^m \lambda_i \overline{b}_i + 
\sum_{\bl \in G} \lambda_\bl \overline{\bl} =0$ 
and that $\{ 1\le i\le m : \lambda_i \notin \Z\}\in \Sigma$. 
For $\lambda\in \K^G_\bk$, we define 
\[
v(\lambda) := \bk+ \sum_{i=1}^m \ceil{\lambda_i} b_i + \sum_{\bl\in G} 
\lambda_\bl \bl.  
\]
Since $\overline{v(\lambda)} = \sum_{i=1}^m \fract{-\lambda_i} b_i$, 
$v(\lambda)$ belongs to $\Bx$. We also set 
\[
d(\lambda) := \Psi(\bk) + (\lambda_1,\dots,\lambda_m) 
+ \sum_{\bl\in G} \lambda_\bl \Psi(\bl) 
\in \LL_\Q\subset \Q^m 
\]
where $\Psi$ is given in Notation \ref{nota:Psi}. 
\end{definition}

The proof of the following lemma 
is similar to Lemma \ref{lem:d(lambda)}, and is omitted. 
\begin{lemma}
For $\lambda \in \K^G_\bk$, $d(\lambda) \in \Laa$. 
\end{lemma} 

\begin{definition}[cf.~Definition \ref{def:I}]
The \emph{localization map} $\Loc$ 
\[
\Loc \colon \GM(F_\chi) \to H^*_{\CR,\T}(\frX) 
\otimes_{R_\T} R_\T(\!(z^{-1})\!)[\![\Laa_+]\!][\![y]\!]   
\]
is a $\C[z][\![\Laa_+]\!][\![y]\!]$-linear map defined by 
\begin{equation}
\label{eq:def_Loc} 
\Loc(w_\bk \omega) =e^{\sum_{i=1}^m u_i \log y_i/z} 
\sum_{\lambda \in \K^G_\bk} 
Q^{d(\lambda)} y^{\lambda} 
\left(\prod_{i \in \{1,\dots,m\} \cup G} 
\frac{\prod_{c\le 0, \fract{c} = \fract{\lambda_i}} u_i + cz}
{\prod_{c\le \lambda_i, \fract{c} = \fract{\lambda_i}} u_i+cz} 
\right) 
\unit_{v(\lambda)},    
\end{equation} 
where 
\begin{itemize} 
\item $y^\lambda := \prod_{i=1}^m y_i^{\lambda_i} 
\prod_{\bk\in G} y_\bk^{\lambda_\bk}$;  
\item for $1\le i\le m$, $u_i$ is the equivariant Poincar\'e dual 
of a toric divisor in \S \ref{subsec:torus_action}; 
for $i\in G$, we set $u_i:=0$; 
\item $\unit_{v(\lambda)}$ is the identity class 
supported on the twisted sector $\frX_{v(\lambda)}$. 
\end{itemize} 
For $1\le i\le m$, $\lambda_i$ is a (possibly negative) 
rational number, and both $\log y_i$ and 
$y_i^{\lambda_i}$ should be expanded in 
Taylor series at $y_i=1$ -- see \eqref{eq:powerseries_y}.  
\end{definition} 

The following lemma shows that the localization map takes 
values in $H^*_{\CR,\T}(\frX) 
\otimes_{R_\T} R_\T(\!(z^{-1})\!)[\![\Laa_+]\!][\![y]\!]$. 
\begin{lemma}
\label{lem:Loc_well-defined} 
Take $\lambda \in \K^G_\bk$. If either $\lambda_\bl <0$ for some $\bl \in G$ 
or $d(\lambda) \notin \bigcup_{\sigma\in \Sigma} C_\sigma$, 
then the summand corresponding to $\lambda$ 
in the left-hand side of \eqref{eq:def_Loc} vanishes 
(see \eqref{eq:cone_C} for $C_\sigma$). 
\end{lemma} 
\begin{proof} 
If $\lambda_\bl<0$ for some $\bl\in G$, then the summand contains a 
factor $\prod_{\lambda_\bl<c\le 0} cz =0$ and thus vanishes. 
Suppose that $\lambda_\bl\ge 0$ for all $\bl\in G$ 
and  that
$d(\lambda) \notin \bigcup_{\sigma \in \Sigma}
C_\sigma$. 
Let $\sigma_0\in \Sigma$ be the minimal cone containing 
$v(\lambda)$. Note that $\sigma_0 = \{1\le i\le m: \lambda_i\notin\Z\}$. 
By assumption we have $I := \{1\le i\le m: D_i \cdot d(\lambda)<0\} 
\notin \Sigma$. 
Set $J := \{1\le i\le m : \lambda_i \in\Z_{<0}\}$. 
Since $D_i\cdot d(\lambda) = \Psi_i(\bk) + \lambda_i  
+ \sum_{\bl\in G}\lambda_\bl \Psi_i(\bl)$, we have 
\[
I \subset J \cup \sigma_0. 
\]
This implies that $J \cup \sigma_0 \notin \Sigma$. 
The summand corresponding to $\lambda$ contains 
a factor $\prod_{i\in J} u_i$. 
But $(\prod_{i\in J} u_i)  \unit_{v(\lambda)}= 0$ 
since 
\[
\bigcap_{i\in J} \{Z_i =0\} \cap \frX_{v(\lambda)} 
= \bigcap_{i\in J\cup \sigma_0} \{Z_i=0\} = \varnothing
\]
by the definition of $\cU_\Sigma$ in \S \ref{subsec:definition}. 
The conclusion follows. 
\end{proof} 

\begin{remark} 
$\Loc(\omega)$ is essentially equivalent to the $G$-extended $I$-function 
in Definition \ref{def:I}. This will be explained in the next section.
\end{remark} 

\begin{remark} 
The name ``localization map'' stems from Givental's heuristic 
argument \cite{Givental:ICM} involving $S^1$-equivariant Floer theory. 
In fact, we can compute $\Loc$ via equivariant localization on polynomial 
loop spaces, generalizing the method in \cite{Givental:homological, Iritani:efc, 
CCLT:wp}. 
\end{remark} 

\begin{remark} 
\label{rem:rational} 
Let $S_{\T\times \C^\times}$ denote the fraction field of 
$H^*_{T\times \C^\times}(\pt)
=R_\T[z]$, where $z$ is regarded as the equivariant parameter 
for $\C^\times$.  
The localization map takes values also in 
$H^*_{\CR,\T}(\frX) \otimes_{R_\T} S_{\T\times \C^\times}[\![\Laa_+]\!]
[\![y]\!]$. This fact will be used in \S\ref{subsec:pairing_match} below. 
\end{remark}

We show that the localization map gives a solution to 
the Gauss--Manin system. Recall that we introduced 
a splitting $\varsigma \colon \bN \to \OO$ in \S 
\ref{subsec:LG}. We write $\varsigma(\bk) = 
(\overline{\varsigma}(\bk), \bk)$ for 
$\bk\in \bN$. The map $\overline{\varsigma}(\bk) 
\colon \bN_\Q \to \Q^m$ defines a splitting 
of the fan sequence \eqref{eq:fanseq} over $\Q$, 
and thus induces a splitting 
$\LL_\C^\star \to (\C^m)^\star$, $\xi \mapsto \hxi$ of the divisor 
sequence \eqref{eq:divseq} such that 
$\hxi$ vanishes on the image of $\overline{\varsigma}$.   
\begin{proposition} 
\label{pro:Loc_diffeq} 
The localization map is linear over $R_\T[z][\![\Laa_+]\!][\![y]\!]$ 
and satisfies the following differential equations: 
\begin{alignat*}{3} 
\Loc(\nabla_{\xi Q\parfrac{}{Q}} \Omega) &= 
\left(\xi Q\parfrac{}{Q} + \frac{1}{z}\hxi \right) \Loc(\Omega), & 
\qquad & \xi \in \LL^\star_\C 
\\ 
\Loc(\nabla_{\parfrac{}{y_\bl}} \Omega) & =  
\parfrac{}{y_\bl} \Loc(\Omega), 
&\qquad & \bl \in S  \\
\Loc(\Gr^{\rm B} \Omega) & = \left( z \parfrac{}{z} + \cE^{\rm B}+ 
\Gr_0\right) \Loc(\Omega) 
\end{alignat*} 
where $\Omega \in \GM(F_\chi)$, $\hxi\in (\C^m)^\star \cong 
H^2_\T(X,\C)$ is the lift of $\xi$ described above,
and $\Gr_0 \in \End_\C(H^*_{\CR,\T}(\frX))$ is the grading 
operator in \eqref{eq:Gr_0}. 
We note that $z\parfrac{}{z} + \cE^{\rm B}$ acts on 
the coefficient ring $R_\T[z][\![\Laa_+]\!][\![y]\!]$ 
and $\Gr_0$ acts on $H^*_{\CR,\T}(\frX)$, 
cf.~\eqref{eq:action_vf+Gr0}. 
\end{proposition} 
\begin{proof} 
It suffices to establish these formulas for $\Omega = w_\bk \omega$. 
For $\lambda \in \Q^m\times \Z^G$, we set 
\[
\Box_\lambda = \prod_{i \in \{1,\dots,m\} \cup G} 
\frac{\prod_{c\le 0, \fract{c} = \fract{\lambda_i}} u_i + cz}
{\prod_{c\le \lambda_i, \fract{c} = \fract{\lambda_i}} u_i+cz}. 
\]
For notational convenience, write $\lambda_{b_i} = \lambda_i$ 
for the $i$th component of $\lambda\in \Q^m\times \Z^G$,  $1\le i\le m$, and think of $\lambda$ as an element of 
$\Q^S$. Hence $\K^G_\bk$ is regarded as a subgroup 
of $\Q^S$. 
We also set $u_{b_i} = u_i$ for $1\le i\le m$. 
For $\bl\in S$, we have the natural identification 
\[
\K^G_{\bk+\bl} \cong \K^G_{\bk}, \qquad 
\lambda \mapsto \lambda' = \lambda + e_\bl. 
\]
Under this identification, we can easily check that
\[ 
d(\lambda' ) = d(\lambda) + d(\bk,\bl),  \quad 
v(\lambda' ) = v(\lambda), \quad 
\Box_\lambda = (u_\bl + \lambda'_\bl z) \Box_{\lambda'}. 
\] 
Using this, for $\bl\in S$, $\bk \in \bN\cap |\Sigma|$, we have 
\begin{align}
\label{eq:Loc_y_bl}
\begin{split}  
& \Loc(\nabla_{\parfrac{}{y_\bl}} w_\bk \omega) 
= \Loc(z^{-1} Q^{d(\bl,\bk)} w_{\bl+\bk}) \\ 
& = z^{-1} Q^{d(\bl,\bk)} 
e^{\sum_{i=1}^m u_i \log y_i/z} 
\sum_{\lambda\in \K^G_{\bk+\bl}} Q^{d(\lambda)} y^\lambda \Box_\lambda 
\unit_{v(\lambda)} \\ 
& = z^{-1} e^{\sum_{i=1}^m u_i \log y_i/z} 
\sum_{\lambda' \in \K^G_\bk} (u_\bl + \lambda'_\bl z)y_\bl^{-1} 
Q^{d(\lambda')} y^{\lambda'} \Box_{\lambda'} 
\unit_{v(\lambda')} = \parfrac{}{y_\bl} \Loc(w_\bk \omega). 
\end{split} 
\end{align} 
For $\xi\in \LL^\star_\C$ and $\bk\in \bN\cap |\Sigma|$, 
\begin{align} 
\label{eq:Loc_xi_Q}
\begin{split} 
\Loc(\nabla_{\xi Q\parfrac{}{Q}} w_\bk \omega) 
& = \Loc\left(\textstyle
(\xi \cdot \lambda(\bk)) w_\bk \omega + z^{-1} \sum_{\bl \in S} 
(\xi \cdot \lambda(\bl)) y_\bl \nabla_{\parfrac{}{y_\bl}} 
w_{\bk}\omega \right) \\ 
& =\left( \textstyle (\xi \cdot \lambda(\bk)) +  
 \sum_{\bl\in S} (\xi\cdot \lambda(\bl)) y_\bl \parfrac{}{y_\bl} 
\right) 
\Loc(w_\bk \omega), 
\end{split} 
\end{align} 
where we used the previous calculation \eqref{eq:Loc_y_bl} in the second line. 
By the choice of the lift $\hxi$, we have  
for $\lambda \in \K^G_\bk$, 
\begin{align*} 
\xi \cdot \lambda(\bk) 
+ \sum_{\bl\in S} 
(\xi \cdot \lambda(\bl)) \lambda_\bl 
&= \hxi \cdot \Psi(\bk) +  
\sum_{\bl\in S} 
(\hxi \cdot \Psi(\bl)) \lambda_\bl 
= \xi \cdot d(\lambda), \\ 
\sum_{\bl\in S} (\xi \cdot \lambda(\bl)) u_\bl 
& = \sum_{i=1}^m (\hxi \cdot \Psi(b_i)) u_i = \hxi.  
\end{align*} 
This implies, for $\lambda \in \K^G_\bk$, 
\[
\left(
\xi \cdot \lambda(\bk) + \sum_{\bl\in S} (\xi \cdot \lambda(\bl)) 
y_\bl \parfrac{}{y_\bl}\right) y^{u/z} Q^{d(\lambda)} 
y^\lambda 
= \left(\xi Q\parfrac{}{Q} + \frac{1}{z} \hxi \right) 
y^{u/z}  
Q^{d(\lambda)} y^\lambda 
\] 
where we write $y^{u/z} = e^{\sum_{i=1}^m u_i \log y_i/z}$.  
Therefore \eqref{eq:Loc_xi_Q} equals 
$(\xi Q\parfrac{}{Q} + z^{-1}\hxi) \Loc(w_\bk \omega)$ 
as required. 
Note that we have for $\lambda \in \K^G_\bk$, 
\begin{align*}
& \left( z\parfrac{}{z} + \Gr_0 \right) \Box_\lambda= 
- \bigg( \sum_{\bl\in S} \ceil{\lambda_\bl}\bigg) \Box_\lambda, 
\quad 
\left(z\parfrac{}{z} + \cE^{\rm B} + \Gr_0\right) y^{u/z} = 0, \\ 
& \cE^{\rm B} (Q^{d(\lambda)} y^\lambda) =
\bigg( |\bk| +  
\sum_{\bl \in S} \lambda_\bl\bigg) Q^{d(\lambda)} y^\lambda, 
\quad 
\Gr_0 \unit_{v(\lambda)} = 
\bigg(\sum_{\bl \in S} \fract{-\lambda_\bl}\bigg) \unit_{v(\lambda)}.  
\end{align*} 
These formulas imply that $(z\parfrac{}{z} + \cE^{\rm B} + \Gr_0) 
\Loc (w_\bk\omega) = |\bk| \Loc(w_\bk \omega) = 
\Loc(\Gr^{\rm B}(w_\bk \omega))$. 
Finally we check that $\Loc$ is linear over $R_\T$. 
We have, for $1\le i\le n$,   
\begin{align*} 
\Loc(\chi_i \cdot w_\bk \omega) & = z k_i  \Loc(w_\bk \omega) 
+ \sum_{\bl\in S} l_i y_\bl \Loc(z \nabla_{\parfrac{}{y_\bl}} 
w_\bk \omega) \\ 
& = z \left( 
k_i + \sum_{\bl \in S} l_i y_\bl \parfrac{}{y_\bl}\right) \Loc(w_\bk\omega) 
\end{align*} 
where we again used \eqref{eq:Loc_y_bl}. 
Since
\[
z \left( k_i + \sum_{\bl \in S} l_i y_\bl \parfrac{}{y_\bl}\right) 
y^{u/z} y^\lambda = \chi_i y^{u/z} y^\lambda 
\]
for $\lambda \in \K^G_\bk$, the conclusion follows.
\end{proof} 

\begin{proposition} 
\label{pro:Loc_Galois}
The localization map is equivariant with respect to the Galois action 
in \S\ref{subsec:Galois_A} and \S\ref{subsec:Galois_LG}, 
that is, we have 
\[
\Loc (\xi \cdot \Omega) = g_0(\xi) \left( \xi \cdot \Loc(\Omega) \right)
\]
where $\xi \in \Pic^\st(\frX)$, $\Omega\in \GM(F_\chi)$, and
$\xi$ on the right-hand side acts on the coefficient ring 
$R_\T(\!(z^{-1})\!)[\![\Laa_+]\!][\![y]\!]$.  
\end{proposition} 
\begin{proof} 
By the definitions of $\Loc(w_\bk\omega)$ and the Galois action, 
it suffices to check that 
\[
\age(\xi, (\Psi(\bk),\bk)) \equiv 
\age(\xi,d(\lambda)) - \sum_{\bl\in G} \lambda_\bl 
\age(\xi,\bl) + \age_{v(\bk)}(L_\xi) \mod \Z. 
\] 
Since $\age_{v(\bk)}(L_\xi) \equiv \age(\xi,(\Psi(v(\bk)),v(\bk)))$ by 
Lemma \ref{lem:age}(1), this reduces to showing that 
\[
(\Psi(\bk),\bk) \equiv d(\lambda) - \sum_{\bl\in G} 
\lambda_\bl \bl + (\Psi(v(\bk)),v(\bk)) \mod \Z^m 
\]
by Lemma \ref{lem:age}(2). This is straightforward. 
\end{proof}

Let $\Loc^{(0)}$ denote the restriction of $\Loc$ 
to $Q=0$, $y = y^*$ (where $y^*$ is the shifted origin 
from equation~\ref{eq:specialization_y}).  
Write $\frakm^G$ for the ideal of $\C[\![\Laa_+]\!][\![y]\!]$ 
generated by $\frakm$, $y_i-1$ $(1\le i\le m)$ and 
$y_\bl$ ($\bl \in G$). 
Then $\Loc^{(0)}$ defines a map: 
\begin{align*} 
\Loc^{(0)} \colon & \GM(F_\chi)/
\frakm^G \GM(F_\chi) 
\cong \bigoplus_{\bk\in \bN\cap |\Sigma|} \C[z]w_\bk \omega 
\longrightarrow  
H^*_{\CR,\T}(\frX) \otimes_{R_\T} R_\T(\!(z^{-1})\!). 
\end{align*} 

\begin{proposition}
\label{pro:Loc0} 
The image of $\Loc^{(0)}$ is $H^*_{\CR,\T}(\frX) \otimes_{R_\T} R_\T[z]$ 
and $\Loc^{(0)}$ is an isomorphism onto its image. 
\end{proposition} 
\begin{proof} 
Only the summand with $\lambda = (-\Psi(\bk),0)\in \Q^m\times \Z^G$ 
in \eqref{eq:def_Loc} 
contributes to $\Loc^{(0)}(w_\bk \omega)$. 
For $\bk \in \bN \cap |\Sigma|$, we have 
\[  
\Loc^{(0)}(w_\bk \omega) 
= \left(\prod_{-\Psi_i(\bk) < c \le 0, 
\fract{c} = \fract{-\Psi_i(\bk)}} (u_i + c z) \right)  
\unit_{\bk - \sum_{i=1}^m \floor{\Psi_i(\bk)} b_i} 
= \phi_\bk +O(z), 
\]
where recall that $\{\phi_\bk : \bk\in \bN\cap |\Sigma|\}$ 
is a $\C$-basis of $H_\T^*(\frX)$ (see \S \ref{subsec:inertia}). 
Since $\Loc^{(0)}(w_\bk \omega)$ is homogeneous of 
$(z\parfrac{}{z} + \Gr_0)$-degree $|\bk|$, the conclusion 
follows. 
\end{proof} 

\begin{remark} 
The map $\Loc^{(0)}$ determines a ``limit opposite subspace''  
at the large radius limit $Q=0, y=y^*$ of the Gauss--Manin system, 
in the sense of \cite{SaitoM:Brieskorn}, \cite[Definition 2.10]{CIT:wallcrossing}, 
\cite[Definition 4.103]{Coates-Iritani:Fock}. This together with the 
limit primitive section $\omega$ at $Q=0$, $y=y^*$ determines 
a Frobenius manifold structure corresponding to 
quantum cohomology.  
\end{remark} 

\begin{theorem}
\label{thm:freeness} 
The equivariant Gauss--Manin system $\GM(F_\chi)$ is a free module over 
$R_\T[z][\![\Laa_+]\!][\![y]\!]$ of rank $\dim H^*_{\CR} (\frX)$. 
Moreover we can choose an $R_\T[z][\![\Laa_+]\!][\![y]\!]$-basis 
$\Omega_1,\dots,\Omega_N$ of $\GM(F_\chi)$ 
which is homogeneous with 
respect to the $\Gr^{\rm B}$-grading. 
\end{theorem} 
\begin{proof}
Using Proposition \ref{pro:Loc0}, we can choose 
$\Gr^{\rm B}$-homogeneous elements 
$\Omega_1,\dots,\Omega_N$ in $\GM(F_\chi)$ 
such that $\Loc^{(0)}(\Omega_1),\dots,\Loc^{(0)}(\Omega_N)$ 
is an $R_\T[z]$-basis of $H^*_{\CR,\T}(\frX)\otimes_{R_\T} R_\T[z]$, 
where $N= \dim H_{\CR}^*(\frX)$. 
We claim that $\Omega_1,\dots,\Omega_N$ give a 
free $R_\T[z][\![\Laa_+]\!][\![y]\!]$-basis of $\GM(F_\chi)$. 
That $\Omega_1,\dots,\Omega_N$ generate $\GM(F_\chi)$ 
is implied by the following well-known fact. 
Let $K$ be a commutative ring with an ideal 
$\frakm\subset K$ such that 
$K$ is complete with respect the $\frakm$-adic topology.  
If a $K$-module $M$ is Hausdorff 
with respect to the $\frakm$-adic topology 
($\bigcap_{p} \frakm^p M = \{0\}$) 
and $\Omega_1,\dots,\Omega_N \in M$ generate 
$M/\frakm M$ over $K/\frakm$, then 
$\Omega_1,\dots,\Omega_N$ generate $M$ over $K$. 
See for instance \cite[Corollary 2, \S 3, Ch.VIII]{Zariski-Samuel}.  
That $\Omega_1,\dots,\Omega_N$ are linearly independent 
over $R_\T[z][\![\Laa_+]\!][\![y]\!]$ follows from 
the fact that $\Loc(\Omega_1),\dots,\Loc(\Omega_N)$ are linearly 
independent in  
$H^*_{\CR,\T}(\frX)\otimes_{R_\T} R_\T(\!(z^{-1})\!)
[\![\Laa_+]\!][\![y]\!]$. 
\end{proof}

\begin{corollary} 
The non-equivariant Gauss--Manin system $\GM(F)$ 
is a free module over $\C[z][\![\Laa_+]\!][\![y]\!]$ of 
rank $\dim H^*_{\CR}(\frX)$. 
\end{corollary} 

\subsection{Mirror isomorphism} 
\label{subsec:mirror_isom} 
In this section we prove the following theorem. 

\begin{theorem} 
\label{thm:mirror_isom}
There is a mirror map 
\[
\tau =\tau(y) \in H^*_{\CR,\T}(\frX) 
\otimes_{R_\T} R_\T[\![\Laa_+]\!][\![y]\!]
\] 
with $\tau(y^*)|_{Q=0} =0$ (see \eqref{eq:specialization_y} for $y^*$) 
and an $R_\T[z][\![\Laa_+]\!][\![y]\!]$-linear mirror isomorphism 
\[
\Theta \colon \GM(F_\chi) \xrightarrow{\cong}  
H^*_{\CR,\T}(\frX) \otimes_{R_\T} R_\T[z][\![\Laa_+]\!][\![y]\!]
\] 
with $\Theta|_{Q=0,y=y^*} = \Loc^{(0)}$ 
such that the following holds. 
\begin{enumerate} 
\item $\Theta$ intertwines the Gauss--Manin connection 
with the pull-back $\tau^*\nabla$ of the quantum connection 
by $\tau$; 

\item the Euler vector fields \eqref{eq:Euler_A}, \eqref{eq:Euler_B} 
correspond under $\tau$: $\tau_* \cE^{\rm B} = \cE^{\rm A}$;  

\item $\Theta$ intertwines the grading operators: 
$\Theta \circ \Gr^{\rm B} =(z\parfrac{}{z} + 
\cE^{\rm B} + \Gr_0) \circ \Theta$; 

\item $\tau$ and $\Theta$ intertwine the Galois symmetries: $\tau(y) = g_0(\xi) (\xi\cdot \tau(y))$ and  
$\Theta(\xi \cdot \Omega) = g_0(\xi) (\xi \cdot \Theta(\Omega))$ 
for $\xi\in \Pic^\st(\frX)$, where $\xi$ on the right-hand side acts on 
the coefficient ring $\C[z][\![\Laa_+]\!][\![y]\!]$,
and $g_0(\xi)\in \End_{R_\T}(H^*_{\CR,\T}(\frX))$ 
is defined in \S \ref{subsec:Galois_A}. 
\end{enumerate} 
\end{theorem} 

\begin{remark} 
Recall that we chose a splitting $\LL_\C^\star \to (\C^m)^\star$ 
when defining the quantum connection in the $Q$-direction 
(see \S \ref{subsec:qcoh})
and we chose a splitting $\varsigma \colon \bN \to \OO$ when 
defining the Landau--Ginzburg potential $F_\chi$ 
(see \S\ref{subsec:LG}). 
In the above theorem, it is assumed that 
these two choices are compatible 
in the sense explained before Proposition \ref{pro:Loc_diffeq}. 
\end{remark} 

\begin{remark} 
\label{rem:Theta_Loc}
The construction of the isomorphism $\Theta$ in the proof below gives the following commutative diagram: 
\[
\xymatrix{
\GM(F_\chi) \ar[rr]^{\Theta\phantom{AAAAA}} \ar[rd]_{\Loc} &  
& H^*_{\CR,\T}(\frX) \otimes_{R_\T} 
R_\T[z][\![\Laa_+]\!][\![y]\!] 
\ar[ld]^{\phantom{A} M(\tau(y),z)} 
\\ 
& \cH[\![y]\!] & 
} 
\]
where $\cH = H^*_{\CR,T}(\frX) \otimes_{R_\T} 
S_\T(\!(z^{-1})\!)[\![\Laa_+]\!]$ 
is Givental's symplectic space from \S \ref{subsec:Givental_cone} 
and $M(\tau,z)$ is the fundamental solution in 
Proposition \ref{pro:fundsol}. 
Note that the images of $\Loc$ and $M(\tau(y),z)$ lie in the 
smaller space $H^*_{\CR,T}(\frX) \otimes_{R_\T}
R_\T(\!(z^{-1})\!)[\![\Laa_+]\!][\![y]\!]$. 
\end{remark} 

\begin{remark} 
The Galois action $\xi\cdot(-)$ on $\C[\![\Laa_+]\!][\![y]\!]$ 
corresponds, under the mirror map, to the pull-back of functions 
by $g(\xi)^{-1} \colon \Spf \C[\![\Laa_+]\!][\![\tau]\!] 
\to \Spf \C[\![\Laa_+]\!][\![\tau]\!]$. 
We have already seen in Lemma \ref{lem:age_Laa}(1) 
that the Galois actions on the Novikov variables $Q$ 
are the same on both sides. 
\end{remark} 

First we explain the relationship between 
the $I$-function and the localization map. 
Let $I(Q,\fry,t,z)$ denote the $G$-extended $I$-function 
from Definition \ref{def:I}. 
For convenience, we denote by 
$\overline{\cL}_\frX = \cL_\frX|_{z\to -z}$ the Givental cone 
with the sign of $z$ flipped. 

\begin{lemma} 
The functions $z \Loc(\omega)$ and $I(Q,\fry,t,z)$ coincide 
under the change of variables 
\[
t_i = \log y_i, \qquad \fry_\bl = y_\bl \prod_{i=1}^m y_i^{-\Psi_i(\bl)}. 
\]
Therefore $z \Loc(\omega)$ defines an 
$S_\T[\![\Laa_+]\!][\![y]\!]$-valued point on the equivariant Givental 
cone $\overline{\cL}_\frX$. 
Moreover, $\Loc(w_\bk \omega)$ lies in the tangent space 
of $\overline{\cL}_\frX$ at this point for every $\bk \in \bN\cap |\Sigma|$. 
\end{lemma} 
\begin{proof} 
The first statement follows from a straightforward computation.  
We show that $\Loc(w_\bk \omega)$ lies in the tangent space 
at $z \Loc(\omega)$. 
If $\bk \in S$, Proposition \ref{pro:Loc_diffeq} gives that 
$\Loc(w_\bk\omega) = \parfrac{}{y_\bk} z \Loc(\omega)$, 
and thus $\Loc(w_\bk\omega)$ is a tangent vector. 
Suppose $\bk \notin S$. 
We consider the $(G\cup \{\bk\})$-extended 
version of the equivariant Gauss--Manin system $\GM'(F_\chi)$ 
and the localization map $\Loc'$. It is obvious that $\GM'(F_\chi) |_{y_\bk=0} 
\cong \GM(F_\chi)$ and that $\Loc'|_{y_\bk =0} = \Loc$. 
Hence we obtain 
\[
\Loc(w_\bk \omega) = \Loc'(w_\bk \omega)\bigg|_{y_\bk=0} 
= \parfrac{}{y_\bk} z \Loc'(\omega) \bigg|_{y_\bk=0}. 
\]
This, together with the fact that $z \Loc'(\omega)$ 
lies in $\overline{\cL}_\frX$, implies the lemma. 
\end{proof} 

The above lemma implies that for every $\Omega \in \GM(F_\chi)$, 
the vector $\Loc(\Omega)$ lies in the tangent space 
at $z\Loc(\omega)\in \overline{\cL}_\frX$. 
By Theorem \ref{thm:freeness}, we can take a $\Gr^{\rm B}$-homogeneous 
$R_\T[\![\Laa_+]\!][\![y]\!]$-basis $\Omega_1,\dots,\Omega_N$ 
of $\GM(F_\chi)$. Since $\Loc(\Omega_i)$ is a tangent vector, 
Proposition \ref{pro:tangentsp_cone} implies that 
there exist $\tau(y) \in H^*_{\CR,\T}(\frX) 
\otimes_{R_\T} S_\T[\![\Laa_+]\!][\![y]\!]$ 
and $\Upsilon_i(y,z) \in H^*_{\CR,\T}(\frX) 
\otimes_{R_\T} S_\T[z][\![\Laa_+]\!][\![y]\!]$ 
such that $\tau(y^*)|_{Q=0} = 0$ and 
\begin{equation} 
\label{eq:tau_Upsilon} 
\Loc(\Omega_i) = M(\tau(y),z) \Upsilon_i(y,z) \qquad 
1\le i\le N. 
\end{equation} 
This $\tau(y)$ gives the mirror map. 
First we claim that $\tau(y)$ and $\Upsilon_i(y,z)$ 
are defined over $R_\T$, i.e.~they are elements 
of $H^*_{\CR,\T}(\frX)\otimes_{R_\T} R_\T[z][\![\Laa_+]\!][\![y]\!]$. 
As noted in \cite{Coates-Givental, Guest:D-mod}, 
we can view the relation \eqref{eq:tau_Upsilon} as 
a Birkhoff factorization \cite{Pressley-Segal} 
of the matrix $(\Loc(\Omega_1),\dots,\Loc(\Omega_N))$: 
\begin{equation} 
\label{eq:Birkhoff}
\begin{pmatrix} 
\vert &  & \vert \\ 
\Loc(\Omega_1) & \cdots & \Loc(\Omega_N) \\ 
\vert &  & \vert 
\end{pmatrix} 
= M(\tau(y), z) 
\begin{pmatrix} 
\vert &  & \vert \\ 
\Upsilon_1(y,z) & \cdots & \Upsilon_N(y,z) \\ 
\vert &  & \vert
\end{pmatrix}.   
\end{equation} 
Here we regard both sides as elements of the loop group $\LGL_N(\C)$ 
with loop parameter $z$; note that $M(\tau(y),z) = I + O(z^{-1})$  
and that $(\Upsilon_1(y,z),\dots, \Upsilon_N(y,z))$ does not contain 
negative powers of $z$. 
The left-hand side belongs to 
\[
\End_{R_\T}(H^*_{\CR,\T}(\frX)) 
\otimes_{R_\T} R_\T(\!(z^{-1})\!)[\![\Laa_+]\!][\![y]\!]
\] 
and is invertible (over $R_\T$) at $Q=0$ and $y=y^*$, 
by the choice of $\Omega_1,\dots, \Omega_N$. 
Therefore the Birkhoff factorization can be performed over $R_\T$ uniquely 
and recursively in powers of $y-y^*$ and $Q$. 
This shows that $\Upsilon_i(y,z)$ and $M(\tau(y),z)$ are defined 
over $R_\T$. Moreover the asymptotics 
$J(\tau(y),z) = z M(\tau(y),z) \unit = z \unit +\tau(y) + O(z^{-1})$ 
shows that $\tau(y)$ is also defined over $R_\T$. 

We define an $R_\T[\![\Laa_+]\!][\![y]\!]$-linear isomorphism 
$\Theta$ by 
\[
\Theta(\Omega_i) := \Upsilon_i(y,z). 
\]
That $\Theta$ intertwines the quantum connection 
follows from the fact that $\Loc$ and $M(\tau,z)$ 
are solutions (see Propositions \ref{pro:fundsol}, \ref{pro:Loc_diffeq}). 
By differentiating \eqref{eq:tau_Upsilon} 
by $\xi Q\parfrac{}{Q} + z^{-1}\hxi$ and $\parfrac{}{y_\bl}$ 
(with $\bl\in S$), we obtain 
\begin{align*} 
\Loc(\nabla_{\xi Q\parfrac{}{Q}} \Omega_i) 
& = M(\tau(y),z) \nabla_{\xi Q\parfrac{}{Q}} 
\Upsilon_i(y,z), \\ 
\Loc(\nabla_{\parfrac{}{y_\bl}} \Omega_i) 
& = M(\tau(y),z) (\tau^*\nabla)_{\parfrac{}{y_{\bl}}} \Upsilon_i(y,z),  
\end{align*} 
where note that $(\tau^*\nabla)_{\parfrac{}{y_\bl}} = \parfrac{}{y_\bl} 
+ z^{-1} (\parfrac{\tau}{y_\bl}\star)$. 
This implies that $\Theta(\nabla_{\xi Q\parfrac{}{Q}} \Omega_i) = 
\nabla_{\xi Q\parfrac{}{Q}} \Upsilon_i(y,z)$ 
and $\Theta(\nabla_{\parfrac{}{y_\bl}} \Omega_i) 
= (\tau^*\nabla)_{\parfrac{}{y_\bl}} \Theta(\Omega_i)$. 
Thus $\Theta$ intertwines the flat connections. 

Next we show that $\Theta$ preserves grading. 
By the choice of $\Omega_i$, $T_i := \Loc(\Omega_i) |_{Q=0,y = y^*, z=0}$ 
with $1\le i\le N$ form a homogeneous basis of $H^*_{\CR,\T}(\frX)$ 
over $R_\T$. 
Expanding in the basis $\{T_i\}$, we regard \eqref{eq:Birkhoff} 
as an equation in matrices whose entries lie in 
$R_\T(\!(z^{-1})\!)[\![\Laa_+]\!][\![y]\!]$. 
Proposition \ref{pro:Loc_diffeq} implies that 
the matrix $L := (\Loc(\Omega_1),\dots,\Loc(\Omega_N))$ 
satisfies the homogeneity equation: 
\[
\hcE  L + [\gr_0, L] = 0 
\]
where $\hcE = z\parfrac{}{z} + \cE^{\rm B}$ 
and $\gr_0=\diag (|1|, |2|,\dots, |N|)$ 
with $|i| = \frac{1}{2} \deg T_i$.  
This implies the following equation 
for the matrices $M = M(\tau(y),z)$ and 
$Y = (\Upsilon_1(y,z),\dots,\Upsilon_N(y,z))$: 
\[
M^{-1} (\hcE M + [\gr_0,M]) = -  (\hcE Y + [\gr_0,Y]) Y^{-1} 
\] 
Since $M = I + O(z^{-1})$ and $Y$ does not contain 
negative powers of $z$, both sides have to vanish. 
Therefore we have 
\begin{align*} 
& \left(z\parfrac{}{z} + \cE^{\rm B} + \Gr_0\right) M(\tau(y),z) \unit = 0, \\
& \left(z\parfrac{}{z} + \cE^{\rm B} + \Gr_0\right) 
\Upsilon_i(y,z) = |i| \Upsilon_i(y,z).  
\end{align*} 
The second equation implies that $\Theta \circ \Gr^{\rm B} = 
(z\parfrac{}{z} + \cE^{\rm B} + \Gr_0) \circ \Theta$. 
The first equation gives 
\[
(\cE^{\rm B} + \Gr_0) \tau(y) = \tau(y) 
\]
which is equivalent to $\tau_* \cE^{\rm B} = \cE^{\rm A}$. 

Finally we show that $\tau$ and $\Theta$ intertwine the Galois actions. 
We may assume that the basis $\Omega_1,\dots,\Omega_N$ of 
$\GM(F_\chi)$ consists of simultaneous eigenvectors for the $\Pic^\st(\frX)$-action. 
The relation $\Loc(\Omega_i) = M(\tau(y),z) \Theta(\Omega_i)$ implies  
\begin{align*} 
M(\tau(y),z) \Theta(\xi \cdot \Omega_i) 
& = \Loc(\xi \cdot \Omega_i)  
= g_0(\xi)(\xi \cdot \Loc(\Omega_i)) \\ 
& = g_0(\xi) (\xi \cdot M(\tau(y),z) ) (\xi \cdot \Theta(\Omega_i)) \\
& = \left( \xi \cdot (g(\xi)^*M(\tau,z))|_{\tau=\tau(y)} \right) 
 g_0(\xi) (\xi \cdot \Theta(\Omega_i))   
\end{align*} 
where we used Proposition \ref{pro:Loc_Galois} in the first line 
and Proposition \ref{pro:Galois_A} in the third line. 
From uniqueness of Birkhoff factorization, we conclude that
\[
M(\tau(y),z) = \xi \cdot (g(\xi)^* M(\tau,z))|_{\tau = \tau(y)}, 
\quad 
\Theta(\xi \cdot \Omega_i) = g_0(\xi) (\xi \cdot \Theta(\Omega_i)).  
\]
It follows that $\Theta(\xi \cdot \Omega) = g_0(\xi) (\xi \cdot \Theta(\Omega))$ 
for all $\Omega \in \GM(F_\chi)$. 
Using $M(\tau,z) \unit = \unit + z^{-1} \tau + O(z^{-2})$, 
we find that $\tau(y) = g_0(\xi) (\xi \cdot \tau(y))$.  
The theorem is proved.

\section{Presentations of the Quantum $D$-Module and Quantum Cohomology Ring} 
\label{sec:presentation} 
In this section, we recast the construction of the equivariant 
Gauss--Manin system in combinatorial terms, and give  presentations of the quantum $D$-module and 
the quantum cohomology ring of $\frX$. 

\subsection{The fan $D$-module} 
\label{subsec:fanD-mod}
Recall that the equivariant Gauss--Manin system has a topological 
basis $\{w_\bk \omega : \bk\in \bN\cap |\Sigma|\}$ over 
$\C[z][\![\Laa_+]\!][\![y]\!]$ -- see \eqref{eq:equiv_GM_decomp}. 
In this section, replacing $w_\bk\omega$ with the the symbol 
$\bunit_\bk$, we define the \emph{fan $D$-module} in an abstract 
and combinatorial way. 
This is closely related to the \emph{better-behaved GKZ system} 
of Borisov and Horja \cite{Borisov-Horja:bbGKZ},
also known as the \emph{multi-GKZ system} \cite{Iritani:periods}. 

Recall the following notation from \S \ref{sec:toric_stacks} and 
\S \ref{subsec:unfolding}: 
\begin{itemize} 
\item the stacky fan $\bSigma=(\bN,\Sigma,\beta)$ (see \S \ref{subsec:definition}); 
\item the map $\beta\colon \Z^m \to \bN$ sending $e_i$ to $b_i$; 
\item the fan sequence 
$0\to \LL \to \Z^m \overset{\beta}{\to} \bN$ 
from \eqref{eq:fanseq};  
\item the divisor sequence 
$0\to \bM \to (\Z^m)^\star \to \LL^\vee$
from \eqref{eq:divseq}; 
tensored with $\Q$, this yields 
$0\to H_\T^2(\pt,\Q) \to H_\T^2(X,\Q) \to H^2(X,\Q) \to 0$; 
\item the monoid and lattice $\Laa_+ \subset \Laa \subset \LL_\Q$ 
(see \S \ref{subsec:refined_fanseq}--\S\ref{subsec:Mori_cone}); 
we write $Q^d$, $d\in \Laa$, for the element in the group ring 
$\C[\Laa]$ that corresponds to $d$;  
\item the finite subset $G \subset \bN \cap |\Sigma|$ disjoint 
from $\{b_1,\dots,b_m\}$; set $S:=G \cup \{b_1,\dots,b_m\}$; 
\item deformation parameters $y_{\bk}$ with $\bk\in S$; 
we also write $y_i := y_{b_i}$ for $1\le i\le m$.  
\end{itemize} 
To avoid the use of a splitting of the divisor sequence over $\C$ 
(see \S\ref{subsec:qcoh}), we introduce a vector field $z\vartheta_\rho$ 
corresponding to an equivariant cohomology class $\rho\in H^2_\T(X,\C)
\cong (\C^m)^\star$. 
We \emph{require} that $z\vartheta_\rho$ acts on 
the parameters $Q^d$, $y_\bk$ as 
\[
z\vartheta_\rho  \cdot Q^d = z (\overline{\rho}\cdot d) Q^d, \qquad 
z\vartheta_\rho  \cdot y_\bk = 0 
\]
where $d\in \Laa$, $\bk \in S$, and $\overline{\rho} \in \LL_\C^\star$ is the 
image of $\rho$ under $(\C^m)^\star \to \LL_\C^\star$. 
We identify an equivariant parameter $\chi \in \bM_\C$ 
with the vector field $z \vartheta_\chi$ 
via the inclusion $\bM_\C \subset (\C^m)^\star$, 
i.e.~$\chi = z \vartheta_\chi$. 
By choosing a splitting of the divisor sequence, we can write 
\begin{equation} 
\label{eq:vartheta_splitting} 
z\vartheta_\rho = \chi^\rho + z \overline{\rho} Q \parfrac{}{Q}  
\end{equation} 
where $\rho \in H^2_\T(X,\C)\cong (\C^m)^\star$ corresponds to 
$(\chi^\rho,\overline{\rho}) \in \bM_\C \oplus \LL_\C^\star$ under 
the splitting. 
Set 
\[
z\vartheta_i := z\vartheta_{u_i} \qquad 1\le i\le m
\]
where $u_i \in H^2_\T(X,\C)$ corresponds to the 
$i$th standard basis element $e_i^\star\in (\C^m)^\star$. 
We write $K[y]$ for the 
polynomial ring in the variables $\{y_\bk : \bk \in S\}$ over 
a ring $K$, and write $\partial_\bk := \partial/\partial y_{\bk}$ 
for the partial derivative in $y_\bk$. 
We also write $\partial_i := \partial_{b_i } = \partial/\partial y_i$.  

\begin{definition} 
\label{def:fanD-mod}
The \emph{fan $D$-module} is the $\C[z][\Laa_+][y]$-module 
\[
\cM(\bSigma,G) := \bigoplus_{\bk \in \bN \cap |\Sigma|} 
\C[z][\Laa_+][y] \bunit_\bk 
\]
equipped with the action of differential operators 
$z \vartheta_1,\dots, z\vartheta_m$, $z \partial_\bl$ 
(with $\bl \in S$) 
as follows: 
\begin{align*} 
z \vartheta_i\cdot \bunit_\bk & := z \Psi_i(\bk) \bunit_\bk 
+ \sum_{\bl \in S} \Psi_i(\bl) y_\bl Q^{d(\bk,\bl)} \bunit_{\bk+\bl}  \\
z \partial_\bl\cdot \bunit_\bk  & := Q^{d(\bk,\bl)} \bunit_{\bk+\bl} 
\end{align*} 
where $\bk \in \bN \cap |\Sigma|$, 
$\Psi \colon \bN \cap |\Sigma| \to (\Q_{\ge 0})^m$ is defined 
in Notation \ref{nota:Psi}, and $d(\bk,\bl)\in \Laa_+$ is defined in \eqref{eq:d(,)}. 
These actions on the basis $\bunit_\bk$ are extended to 
$\cM(\bSigma,G)$ by the standard Leibnitz rule. 
Equivariant parameters $\chi \in \bM_\C$ act 
via the identification $\chi = z \vartheta_\chi$ as follows: 
\[
\chi \cdot \bunit_\bk := z \vartheta_\chi \cdot \bunit_\bk 
= z (\chi \cdot \bk) \bunit_\bk + 
\sum_{\bl \in S} (\chi \cdot \bl) y_\bl Q^{d(\bk,\bl)}\bunit_{\bk+\bl}. 
\]
\end{definition} 
\begin{remark} 
The fan $D$-module arises from the Gauss--Manin connection 
in \S \ref{subsec:GM}. 
The generator $\bunit_\bk$ of $\cM(\bSigma,G)$ corresponds to 
$w_\bk \omega \in \GM(F_\chi)$ and the 
actions of $z\vartheta_\rho$, 
$z\partial_\bk$ correspond respectively to the Gauss--Manin connections 
$\chi^\rho + z \nabla_{\overline{\rho} Q\parfrac{}{Q}}$, 
$z \nabla_{\parfrac{}{y_\bk}}$ once we choose a splitting 
as in \eqref{eq:vartheta_splitting} -- compare the above definition 
with \eqref{eq:action_chi_i} and \eqref{eq:GM_conn_explicit}.  
It is easy to check that the actions of 
$z\vartheta_i$, $z \partial_\bk$ commute with each other. 
\end{remark} 
\begin{remark} 
The fan $D$-module is graded with respect to 
the following (complex) grading: 
\begin{gather*} 
\deg \bunit_\bk = |\bk|, \quad 
\deg z = \deg \chi = 1, \quad 
\deg y_\bk = 1- |\bk|, \\
\deg Q^d = c_1(\frX) \cdot d,  
\quad \deg z\vartheta_\rho = \deg z \partial_{\bk} = 1, 
\end{gather*} 
where $|\bk| = \sum_{i=1}^m \Psi_i(\bk)$. 
This corresponds to the grading operator $\Gr^{\rm B}$ 
-- see \eqref{eq:Gr_B} -- on the Gauss--Manin system. 
\end{remark} 

\begin{remark}[reduced fan $D$-module]  
\label{rem:redundancy_fanDmod}
As noted in Remark \ref{rem:redundancy}, there is redundancy 
among the parameters $(Q,y_1,\dots,y_m, \{y_\bk : \bk \in G\})$. 
Defining new parameters and basis as 
\[
\Qbar^d := \left(\prod_{i=1}^m y_i^{D_i\cdot d} \right) Q^d, \quad 
\ybar_\bl := \frac{y_\bl}{\prod_{i=1}^m y_i^{\Psi_i(\bl)}}, 
\quad 
\bunitbar_\bk := \left(\prod_{i=1}^m y_i^{\Psi_i(\bk)}\right) \bunit_\bk 
\]
where $d\in \Laa_+$, $\bl \in G$, and $\bk\in \bN \cap |\Sigma|$, 
we can remove $y_1,\dots,y_m$ from the presentation of 
the $D$-module. 
Indeed, the vector fields corresponding to these new parameters 
$(\Qbar,\ybar_\bl)$ 
are given by 
\[
z \varthetabar_\rho = z \vartheta_\rho, \quad 
z \partialbar_\bl = \left( \prod_{i=1}^m y_i^{\Psi_i(\bl)}\right) 
z\partial_\bl 
\]
with $\rho \in (\C^m)^\star$, $\bl \in G$, and we have 
\begin{align*} 
z \varthetabar_i \cdot \bunitbar_\bk & = z \Psi_i(\bk) \bunitbar_\bk 
+ \Qbar^{d(\bk,b_i)} \bunitbar_{\bk+b_i} 
+ \sum_{\bl\in G} \ybar_\bl \Qbar^{d(\bk,\bl)} \bunitbar_{\bk+\bl} \\ 
z \partialbar_\bl \cdot \bunitbar_\bk & = \Qbar^{d(\bk,\bl)} \bunitbar_{\bk+\bl} 
\end{align*} 
with $1\le i\le m$ and $\bl \in G$. 
We also have $z y_i \partial_i \cdot \bunitbar_\bk = 
(z \varthetabar_i - \sum_{\bl \in G} 
\Psi_i(\bl) \ybar_\bl \partialbar_{\bl})\bunitbar_\bk$ 
for $\partial_i := \partial_{b_i} = \partial/\partial y_i$. 
This implies that the fan $D$-module descends to 
the space of the parameters $(\Qbar,\{\ybar_\bl : \bl \in G\})$. 
We can alternatively get this reduction 
by setting $y_1= \cdots = y_m=1$ in the definition 
of $z\vartheta_i$, $z\partial_\bl$. 
We call $\cM_{\rm red}(\bSigma,G) := \cM(\bSigma,G)/(y_1-1,\dots,y_m-1)$ the 
\emph{reduced fan $D$-module}. 
\end{remark}

\begin{definition} 
\label{def:completed_fanDmod} 
Let $\frakm$ denote the ideal of $\C[z][\Laa_+][y]$ 
generated by $Q^d$ with $d\in \Laa_+\setminus \{0\}$,  
$y_1-1,\dots,y_m-1$ and $y_\bk$ with $\bk \in G$. 
We define the \emph{completed fan $D$-module} 
to be the $\frakm$-adic completion of $\cM(\bSigma,G)$: 
\[
\hcM(\bSigma,G) := \varprojlim_{k} \cM(\bSigma,G)/\frakm^k \cM(\bSigma,G) 
= \hbigoplus_{\bk\in \bN \cap |\Sigma|} \C[z][\![\Laa_+]\!][\![y]\!] \bunit_\bk
\] 
where we used the convention \eqref{eq:powerseries_y} for the ring 
of power series in $\{y_\bk: \bk \in S\}$. 
This naturally becomes a module over $R_\T[z][\![\Laa_+]\!][\![y]\!]$, 
where $R_\T = \Sym^*(\bM_\C)$, and is equipped with the action 
of $z \vartheta_1,\dots, z\vartheta_m$ and $z \partial_\bk$ with 
$\bk \in S$.  
We can similarly define the \emph{completed reduced fan $D$-module} 
$\hcM_{\rm red}(\bSigma,G)$. 
\end{definition} 

It is clear from the definition 
that the completed fan $D$-module is isomorphic 
to the equivariant Gauss--Manin system. 
The \emph{equivariant quantum $D$-module} of $\frX$ is defined 
to be the module 
$H^*_{\CR,\T}(\frX)\otimes_{R_\T} R_\T[z][\![\Laa_+]\!][\![\tau]\!]$ 
equipped with the action of the quantum connection 
$\nabla$ multiplied by $z$  -- see \eqref{eq:q_conn}. 
Note that the quantum connection $z \nabla_{\xi Q\parfrac{}{Q}}$ 
in the $Q$-direction and equivariant parameter $\chi$ together give 
the action of $H^2_\T(X,\C) \cong (\C^m)^\star$ 
similarly to \eqref{eq:vartheta_splitting}; 
the action of $H^2_\T(X,\C)$ is canonical in the sense that 
it does not depend on the choice of a splitting.  
Theorem \ref{thm:mirror_isom} gives the following: 
\begin{theorem} 
\label{thm:quantum_D-mod_presentation} 
The completed fan $D$-module $\hcM(\bSigma,G)$ is isomorphic to 
the pull-back of the equivariant quantum $D$-module of the 
toric stack $\frX$ by the mirror map $\tau=\tau(y)$ 
from Theorem \ref{thm:mirror_isom}. 
The isomorphism here preserves the grading. 
\end{theorem}  

\subsection{GKZ-style presentation} 
We give a Gelfand--Kapranov--Zelevinsky-style presentation~\cite{GKZ:hypergeom} 
of the (completed) fan $D$-module. 
The following two propositions follow immediately from the 
definition. 
\begin{proposition} 
\label{pro:relation} 
We have the relations $\cR_{i,\bk} = \cP_{d_1,d_2;a,a';\bk_1,\bk_2} =0$
in the fan $D$-module, where 
\begin{align*} 
\cR_{i,\bk} & := z \vartheta_i \cdot \bunit_\bk - \left( z\Psi_i(k) + 
z y_i \partial_i 
+ \sum_{\bl\in G} z y_\bl \partial_\bl \right) \bunit_\bk, \\
\cP_{d_1,d_2;a,a';\bk_1,\bk_2}&:= 
Q^{d_1} \prod_{\bl\in S} (z \partial_\bl)^{a_\bl} \bunit_{\bk_1} 
- Q^{d_2} \prod_{\bl \in S}(z\partial_\bl)^{a'_\bl} \bunit_{\bk_2}, 
\end{align*} 
Here $1\le i\le m$, $\bk \in \bN\cap |\Sigma|$, and 
$(d_1,d_2,a,a',\bk_1,\bk_2)$ ranges over all 
$d_1,d_2\in \Laa_+$, $a,a'\in (\Z_{\ge 0})^S$, 
$\bk_1,\bk_2 \in \bN\cap |\Sigma|$ such that 
$\bk_1 + \sum_{\bl\in S} a_\bl \bl = \bk_2 + \sum_{\bl \in S} 
a'_{\bl} \bl$ in $\bN$ and that 
$d_1 + \sum_{\bl\in S} a_\bl \Psi(\bl) + \Psi(\bk_1) 
= d_2 + \sum_{\bl \in S} a'_\bl \Psi(\bl) + \Psi(\bk_2)$ 
in $\Q^m$. 
\end{proposition} 

\begin{proposition}[reduced case]
\label{pro:relation_reduced} 
In the reduced fan $D$-module in Remark \ref{rem:redundancy_fanDmod}, 
we have the relations 
$\cP'_{d_1,d_2;a,a';\bk_1,\bk_2} =0$ with 
\begin{multline*} 
\cP'_{d_1,d_2;a,a';\bk_1,\bk_2} := 
Q^{d_1} \prod_{\bl \in G} (z \partial_\bl)^{a_\bl} \cdot 
\prod_{i=1}^m \prod_{\nu=0}^{a_i-1}
\left(z \vartheta_i - z (\nu+\Psi_i(\bk_1)) -\sum_{\bl\in G} z  y_\bl \partial_\bl 
\right) \bunit_{\bk_1} \\ 
- Q^{d_2} \prod_{\bl \in G}(z\partial_\bl)^{a'_\bl} \cdot 
\prod_{i=1}^m \prod_{\nu=0}^{a'_i-1} 
\left( z\vartheta_i - z (\nu+\Psi_i(\bk_2)) - \sum_{\bl\in G} z y_\bl \partial_\bl 
\right) \bunit_{\bk_2} 
\end{multline*} 
where $(d_1,d_2,a,a',\bk_1,\bk_2)$ ranges over the same set 
as in Proposition \ref{pro:relation}. 
\end{proposition} 

These relations give a presentation of the (reduced) fan $D$-module. 
We set 
\begin{align*} 
\cD &= \C[z][\Laa_+][y] \left\langle  z\vartheta_1,\dots,z\vartheta_m, 
\{z \partial_{\bl}\}_ {\bl \in S} \right\rangle, \\  
\cD' & = \C[z][\Laa_+][\{y_\bl\}_{\bl\in G}] 
\left\langle 
z\vartheta_1,\dots, z\vartheta_m, 
\{z\partial_\bl\}_{\bl\in G} \right\rangle,  
\end{align*} 
where the standard commutation relations 
$[z\vartheta_i, Q^d] = z (D_i\cdot d)Q^d$, 
$[z\partial_\bl, y_\bk]= z\delta_{\bk,\bl}$, 
$[z\vartheta_i, y_\bk] = [z\partial_\bk, Q^d] =[z\vartheta_i, z\vartheta_j] 
= [z \vartheta_i, z\partial_\bl] = [z\partial_\bl,z\partial_\bk] =0$ 
are implicitly imposed. 
The fan $D$-module $\cM(\bSigma,G)$ (resp.~$\cM_{\rm red}(\bSigma,G)$) 
is a $\cD$-module (resp.~$\cD'$-module). 
\begin{theorem} 
\label{thm:fanDmod_presentation}
Let $\bk_1,\dots,\bk_s$ be elements of $\bN \cap |\Sigma|$ such that 
for every maximal cone $\sigma \in \Sigma(n)$ we have 
\[
\bN \cap \sigma = \bigcup_{1\le i\le s : \overline{\bk_i} \in \sigma} 
\left( \bk_i + \sum_{\bl\in S: \overline{\bl}\in \sigma} 
\Z_{\ge 0} \bl \right).  
\]
Then we have the following. 
\begin{itemize} 
\item[(1)] As a $\cD$-module, 
the fan $D$-module $\cM(\bSigma,G)$ is generated by 
$\bunit_{\bk_1},\dots,\bunit_{\bk_s}$. 
All the relations among $\bunit_{\bk_1},\dots,\bunit_{\bk_s}$ 
are generated by 
$\cR_{\bk_j,i}$, $\cP_{d_1,d_2;a,a';\bk_j,\bk_l}$ 
with $1\le j,l\le s$ 
in Proposition \ref{pro:relation}. 
\item[(2)] As a $\cD'$-module, 
the reduced fan $D$-module $\cM_{\rm red}(\bSigma,G)$ 
is generated by 
$\bunit_{\bk_1},\dots,\bunit_{\bk_s}$. 
All the relations among $\bunit_{\bk_1},\dots,\bunit_{\bk_s}$ 
are generated by $\cP'_{d_1,d_2;a,a';\bk_j,\bk_l}$ 
with $1\le j,l\le s$  in Proposition \ref{pro:relation_reduced}. 
\end{itemize} 
\end{theorem} 
\begin{proof} 
We give a proof of part (1). 
The proof of part (2) is similar and is left to the reader. 
Our assumption on $\bk_1,\dots,\bk_s$ implies that 
for any $\bk \in \bN \cap \sigma$, there exist 
$1\le i\le s$ and $\bl_1,\dots,\bl_t \in S$ such that 
$\bk = \bk_i + \bl_1 + \cdots + \bl_t$ and 
that $\overline{\bk}_i,\overline{\bl}_1,
\dots,\overline{\bl}_t \in \sigma$. 
Then we have 
$\bunit_\bk = z\partial_{\bl_1} \cdots z\partial_{\bl_t} \bunit_{\bk_i}$. 
Therefore $\cM(\bSigma,G)$ is generated by 
$\bunit_{\bk_1},\dots,\bunit_{\bk_s}$ as a $\cD$-module. 
Suppose that we have a relation 
$\sum_{j=1}^s f_j(z,Q,y,z \vartheta,z\partial) \bunit_{\bk_j} =0$ 
with $f_j \in \cD$ in the fan $D$-module.  
Modulo the relations $\cR_{\bk_j,i}$, this relation can be 
reduced to a relation which does not involve $z\vartheta_i$. 
Thus we may assume that $f_j$ does not contain $z\vartheta_1,\dots,z\vartheta_m$. 
Then we can expand 
\[
f_j = \sum_{a\in (\Z_{\ge 0})^S} f_{j,a}(z,Q,y) 
\prod_{\bl\in S} (z \partial_\bl)^{a_\bl}
\]
for some $f_{j,a} \in \C[z][\Laa_+][y]$. The relation implies that
\begin{equation} 
\label{eq:relation_coefficientwise} 
\sum_{(j,a): \bk_j + \sum_{\bl\in S} a_\bl \bl = \bk}
f_{j,a}(z,Q,y) Q^{\Psi(\bk_j) + \sum_{\bl\in S} a_\bl \Psi(\bl) -\Psi(\bk)} =0 
\end{equation} 
for every $\bk \in \bN \cap |\Sigma|$. 
The relation $\sum_{j=1}^s f_j \bunit_{\bk_j}=0$ 
is the sum of the relations
\[
\sum_{(j,a):\bk_j + \sum_{\bl\in S} a_\bl \bl =\bk} 
f_{j,a}(z,Q,y) \prod_{\bl\in S}  (z\partial_\bl)^{a_\bl} \cdot 
\bunit_{\bk_j} =0.  
\]
One can easily see that each of these is 
generated by $\cP_{d_1,d_2;a,a';\bk_j, \bk_l}$'s 
over $\C[z][y]$ by using \eqref{eq:relation_coefficientwise}.  
\end{proof} 

The completed fan $D$-module is described in terms of the 
closure of relations in the $\frakm$-adic topology. 
We introduce the following rings of differential operators: 
\begin{align*} 
\hcD &= \C[z][\![\Laa_+]\!][\![y]\!] \langle 
z \vartheta_1,\dots, z\vartheta_m, 
\{z \partial_\bl \}_{\bl \in S} \rangle, \\ 
\hcD' & = \C[z][\![\Laa_+]\!][\![\{y_\bl\}_{\bl\in G}]\!] 
\langle 
z \vartheta_1,\dots, z\vartheta_m, 
\{z \partial_\bl \}_{\bl \in G} \rangle 
\end{align*} 
where we use the convention \eqref{eq:powerseries_y} 
for the ring of power series in $\{y_\bl: \bl \in S\}$. 
We define the topology on $\hcD$ (resp.~$\hcD'$) by 
the decreasing $\C[z][\![\Laa_+]\!][\![y]\!]$-submodules $\frakm^p \hcD$ 
(resp.~$\C[z][\![\Laa_+]\!][\![\{y\}_{\bl \in G}]\!]$-submodules 
${\frakm'}^p \hcD'$), where $\frakm\subset \C[z][\Laa_+][y]$
is as in Definition \ref{def:completed_fanDmod} and 
$\frakm' := \frakm \cap \C[z][\Laa_+][\{y_\bl\}_{\bl\in G}]$. 
Note that $\frakm^p \hcD$ and ${\frakm'}^p \hcD'$ are only right ideals. 
Note also that $\hcD$ and $\hcD'$ are not complete with respect to 
their topologies. 

\begin{theorem} 
\label{thm:completed_fanDmod_presentation} 
Let $\bk_1,\dots,\bk_s \in \bN \cap |\Sigma|$ be elements 
satisfying the condition in Theorem \ref{thm:fanDmod_presentation}. 
Then we have the following. 
\begin{itemize} 
\item[(1)] The completed fan $D$-module $\hcM(\bSigma,G)$ 
has the following presentation as a $\hcD$-module: 
\[
\hcM(\bSigma,G) \cong \bigoplus_{j=1}^s \hcD \bunit_{\bk_j} \Big/\, \overline{\frI}
\] 
where $\frI$ is the left $\hcD$-submodule of 
$\bigoplus_{j=1}^s \hcD \bunit_{\bk_j}$ generated by 
$\cR_{i,\bk_j}$, $\cP_{d_1,d_2;a,a';\bk_j,\bk_l}$ 
in Proposition \ref{pro:relation} with 
$1\le j,l\le s$, and $\overline{\frI}$ is the closure\footnote
{One can check that the closure of a left $\hcD$-submodule 
in $\bigoplus_{j=1}^s \hcD \bunit_{\bk_j}$ becomes a 
left $\hcD$-submodule using the fact that 
$z\vartheta_i \frakm^p \hcD \subset \frakm^p \hcD$ and 
$z \partial_\bl \frakm^p \hcD \subset \frakm^{p-1} \hcD$. } 
of $\frI$ in $\bigoplus_{j=1}^s \hcD \bunit_{\bk_j}$. 

\item[(2)] 
The completed reduced fan $D$-module 
$\hcM_{\rm red}(\bSigma,G)$ has the following presentation as 
a $\hcD'$-module: 
\[
\hcM_{\rm red}(\bSigma,G) \cong 
\bigoplus_{j=1}^s \hcD' \bunit_{\bk_j} \Big/\, \overline{\frI'}
\] 
where $\frI'$ is the left $\hcD'$-submodule of 
$\bigoplus_{j=1}^s \hcD' \bunit_{\bk_j}$ generated by 
$\cP'_{d_1,d_2;a,a';\bk_j,\bk_l}$ 
in Proposition \ref{pro:relation_reduced} 
with 
$1\le j,l\le s$, and $\overline{\frI'}$ is the closure of $\frI'$
in $\bigoplus_{j=1}^s \hcD' \bunit_{\bk_j}$. 
\end{itemize} 
\end{theorem} 
\begin{proof} 
We only give a proof of part (1). The proof of part (2) is similar. 
The fact that $\bunit_{\bk_1},\dots,\bunit_{\bk_s}$ generate 
$\hcM(\bSigma,G)$ as a $\hcD$-module follows from 
the discussion in the proof of Theorem \ref{thm:fanDmod_presentation}. 
It is easy to show that elements of $\overline{\frI}$ are relations 
in $\hcM(\bSigma,G)$. 
Suppose we have a relation 
$\sum_{j=1}^s f_j(z,Q,y,z\vartheta, z\partial) \bunit_{\bk_j} =0$ 
in $\hcM(\bSigma,G)$ for some $f_j\in \hcD$. 
For each $p\in \Z_{\ge 0}$ we can write 
$f_j = f_j^{(p)} + r_j^{(p)}$ with $f_j^{(p)} \in \cD$ 
and $r_j^{(p)} \in \frakm^p \hcD$. 
Then we have that
\[
x := \sum_{j=1}^s f_j^{(p)} \bunit_{\bk_j} = 
-\sum_{j=1}^s r_j^{(p)} \bunit_{\bk_j} 
\]  
belongs to $\frakm^p \cM(\bSigma,G) = \bigoplus_{\bk \in \bN\cap |\Sigma|}
\frakm^p \bunit_\bk$. 
By the surjectivity of $\bigoplus_{j=1}^s \frakm^p \cD \bunit_{\bk_j} 
\to \frakm^p \cM(\bSigma,G)$, we can write 
$x = \sum_{j=1}^s g^{(p)}_j \bunit_{\bk_j}$ 
for some $g^{(p)}_j \in \frakm^p \cD$.  
We now have that $\sum_{j=1}^s (f^{(p)}_j-g^{(p)}_j) \bunit_{\bk_j} =0$, 
and thus $h^{(p)}: = \bigoplus_{j=1}^s (f^{(p)}_j - g^{(p)}_j) \bunit_{\bk_j}$ 
belongs to $\frI$ by Theorem~\ref{thm:fanDmod_presentation}. 
Since $h^{(p)}$ converges to $\bigoplus_{j=1}^s f_j \bunit_{\bk_j}$ 
as $p\to \infty$, 
we have that $\bigoplus_{j=1}^s f_j\bunit_{\bk_j} 
\in \overline{\frI}$. 
\end{proof} 

\begin{remark} 
\label{rem:GKZ}
When $\frX$ is a toric manifold, $\cM(\bSigma,G)$, $\hcM(\bSigma,G)$ 
and their reduced versions are generated by $\bunit_0$.  The 
 relations $\cP_{d_1,d_2;a,a';0,0}$ then define the standard 
GKZ system \cite{GKZ:hypergeom} 
and also appear as relations in the quantum $D$-module 
\cite[Theorem 1]{Givental:ICM}. 
(For general $\frX$, these $D$-modules 
are generated by $\bunit_0$ for a sufficiently 
large $G$.) The closure of the GKZ ideal appeared in 
\cite[Proposition 5.4]{Iritani:coLef} for compact toric manifolds. 
A closely related presentation has been discussed in 
\cite[\S 4.2]{Iritani:integral}, \cite[\S 5.2]{Acosta-Shoemaker:toric}, 
\cite[Theorem 6.6]{Mann-Reichelt} 
for compact toric stacks. 
The relations $\cP'_{d_1,d_2;a,a';0,0}$ for $\bunit_0$ 
in the reduced fan $D$-module 
also appeared in \cite[Lemma 4.7]{Mann-Reichelt}. 
\end{remark} 

\begin{remark}
The same result as Theorem \ref{thm:completed_fanDmod_presentation} 
holds, with the same proof, 
when we replace $\hcD$ and $\hcD'$ respectively with 
\begin{align*} 
& R_\T[z][\![\Laa_+]\!][\![y]\!]
\left\langle z \vartheta_1,\dots,z\vartheta_m, 
\{z \partial_\bl\}_{\bl \in S} \right\rangle\\ 
\text{and} \quad & R_\T[z][\![\Laa_+]\!][\![\{y_\bl\}_{\bl \in G}]\!] 
\left\langle z \vartheta_1,\dots,z\vartheta_m, 
\{z \partial_{\bl} \}_{\bl\in G}\right\rangle
\end{align*}  
where the relations $\chi = z \vartheta_\chi$ with $\chi \in \bM_\C$ 
are implicitly imposed.  
\end{remark} 

\subsection{Quantum cohomology ring} 
We next give a quantum Stanley--Reisner (or Batyrev-style) description 
of the quantum cohomology algebra of $\frX$.
\begin{theorem} 
\label{thm:qcoh_presentation} 
Let $\tau=\tau(y)$ be the mirror map from Theorem \ref{thm:mirror_isom}.  
The pull-back of the big and equivariant quantum cohomology ring 
$\big(H^*_{\CR,\T}(\frX) \otimes_{R_\T} R_\T[\![\Laa_+]\!][\![y]\!], 
\tau^*(\star) \big)$ by $\tau$ is isomorphic to any one of the following 
rings as an $R_\T[\![\Laa_+]\!][\![y]\!]$-algebra:  
\begin{itemize} 
\item[(a)] the Jacobian ring of the equivariant Landau--Ginzburg potential $F_\chi$, 
\begin{align*} 
\Jac(F_\chi) & := \C\{\OO_+\}[\![y]\!][\chi]
\bigg/\left\langle x_i \parfrac{F_\chi}{x_i} : 
1\le i\le n \right\rangle \cong \C\{\OO_+\}[\![y]\!]; 
\end{align*} 
\item[(b)] the vector space $\hbigoplus_{\bk \in \bN \cap|\Sigma|} 
\C[\![\Laa_+]\!][\![y]\!] \bunit_\bk$ equipped with the following 
product and the $R_\T$-module structure: 
\begin{align*} 
\bunit_\bk \star \bunit_\bl = Q^{d(\bk,\bl)} \bunit_{\bk+\bl},  
\qquad 
\chi = \sum_{\bl \in S} (\chi \cdot \bl) y_\bl \bunit_\bl, 
\end{align*} 
where $\hbigoplus$ denotes the completed direct sum 
with respect to the $\frakm$-adic topology 
(see Definition \ref{def:completed_fanDmod}), 
$\chi \in \bM_\C$, and $d(\bk,\bl)$ is defined in \eqref{eq:d(,)}. 
\end{itemize} 
\end{theorem} 
\begin{remark} 
In part (a) above, 
we follow the notation from \S\ref{subsec:LG}--\S\ref{subsec:unfolding} 
and define co-ordinates $x_1,\dots,x_n$ by choosing 
an isomorphism $\bN \cong \Z^n \times \bN_{\rm tor}$; 
we have $x_i \parfrac{F_\chi}{x_i} = 
\sum_{\bk \in S} y_\bk k_i w_\bk - \chi_i$ 
with $k_i$ being the $i$th component of 
$\overline{\bk} \in \bN/\bN_{\rm tor} \cong \Z^n$ 
and $\chi_i$ being the $i$th basis of $\bM \cong \Z^n$. 
\end{remark} 
\begin{remark} 
Note that the presentation in part (b) yields the description 
of the Chen--Ruan cup product in \S\ref{subsec:inertia} 
at the classical limit $Q=0$, $y=y^*$ -- see 
\eqref{eq:specialization_y} for $y^*$. 
\end{remark} 

\begin{remark} 
A presentation of the quantum cohomology of Fano toric manifolds 
was originally found by Givental \cite{Givental:fixedpoint_toric} 
and Batyrev \cite{Batyrev:qcoh_toric}, and generalizations to 
arbitrary toric manifolds were discussed in 
\cite{Givental:toric_mirrorthm,McDuff-Tolman,
Iritani:coLef,FOOO:toricI,Brown:toric_fibration}. 
The Jacobian description of the quantum cohomology of toric stacks, in the non-equivariant case,
was also given in \cite{Iritani:integral, Gonzalez-Woodward:tmmp}. 
\end{remark}

\begin{proof}[Proof of Theorem \ref{thm:qcoh_presentation}]  
We remark that for an extension $G_1 \subset G_2$ of the finite set $G$, 
the corresponding mirror maps and mirror isomorphisms 
in Theorem \ref{thm:mirror_isom} 
are related by restriction. 
Therefore we can define the partial derivative 
$(\partial \tau/\partial y_\bk) \in H_{\CR,\T}^*(\frX)\otimes_{R_\T} 
R_\T[\![\Laa_+]\!][\![y]\!]$ for every $\bk\in \bN\cap |\Sigma|$ 
by adding $\bk$ to $G$ (if necessary) and then restricting to $y_\bk=0$ 
if $\bk\notin S$. 
We claim that the $\C[\![\Laa_+]\!][\![y]\!]$-module homomorphism 
sending $w_\bk \in \Jac(F_\chi)$ to $(\partial \tau/\partial y_\bk)$ 
gives\footnote{This is well-defined, since $\{w_\bk\}_{\bk \in \bN \cap |\Sigma|}$ 
is a topological $\C[\![\Laa_+]\!][\![y]\!]$-basis of $\Jac(F_\chi)$.} the desired isomorphism in part (a). 
The mirror isomorphism $\Theta$ induces the isomorphism 
of $R_\T[\![\Laa_+]\!][\![y]\!]$-modules: 
\[
\Jac(F_\chi) \cdot \omega \cong \GM(F_\chi)/z \GM(F_\chi) 
\xrightarrow{\Theta|_{z=0}} H^*_{\CR,\T}(\frX)\otimes_{R_\T}
R_\T[\![\Laa_+]\!][\![y]\!]. 
\]
Again by adding $\bk$ to $G$ if necessary, we can show that 
this intertwines the action of $w_\bk$ with the quantum multiplication 
by $\partial \tau/\partial y_\bk$ (since 
$\Theta$ intertwines the Gauss--Manin connection $z \nabla_{\partial/\partial y_\bk}$ 
with the quantum connection $z (\tau^*\nabla)_{\partial/\partial y_\bk}$). 
Part (a) follows. 
Part (b) follows easily from part (a) by noting that 
$\C\{\OO_+\}[\![y]\!]$ is isomorphic to $\hbigoplus_{\bk \in \bN \cap |\Sigma|} 
\C[\![\Laa_+]\!][\![y]\!]$ -- see \eqref{eq:completed_directsum}. 
\end{proof}

\subsection{Examples}

We give several examples of the fan $D$-module. 

\subsubsection{The case where $\frX = \PP^1$ and $G=\varnothing$} 
The stacky fan is given by 
$\bN= \Z$, $b_1 = 1$, $b_2 = -1$. 
We have 
$\LL = \Laa = \{(l,l)\in \Z^2\;| \; l \in\Z\} \subset \Z^2$. 
Let $Q\in \C[\Laa]$ correspond to the positive generator 
$(1,1)\in \Laa_+$ and let $\chi \in \bM = \Hom(\bN,\Z)\cong \Z$ 
be the standard generator. 
We have the identification 
\[
\chi = z \vartheta_1 - z \vartheta_2 
\]
By choosing $\chi_1$,~$\chi_2 \in \bM_\C$ 
with $\chi_1 - \chi_2  = \chi$, we can also write 
\[
z\vartheta_1 = z Q \parfrac{}{Q} + \chi_1, \quad 
z\vartheta_2 = z Q \parfrac{}{Q} + \chi_2.  
\]
The action of these operators on the reduced fan $D$-module 
$\cM_{\rm red}(\bSigma_{\PP^1})$ is shown in 
Figure \ref{fig:P1_Dmod}. 
\begin{figure}[htbp]
\begin{center} 
\begin{picture}(300,55) 
\path(0,30)(300,30) 
\put(250,30){\makebox(0,0){\circle*{3}}}
\put(200,30){\makebox(0,0){\circle*{3}}}
\put(150,30){\makebox(0,0){\circle*{5}}}
\put(100,30){\makebox(0,0){\circle*{3}}} 
\put(50,30){\makebox(0,0){\circle*{3}}} 

\put(250,20){\makebox(0,0){$2$}} 
\put(200,20){\makebox(0,0){$1$}}
\put(150,20){\makebox(0,0){$0$}}
\put(100,20){\makebox(0,0){$-1$}}
\put(50,20){\makebox(0,0){$-2$}}

\put(275,45){\makebox(0,0){$
\begin{CD}@>{z \vartheta_1-2z}>>\end{CD}$}}
\put(225,45){\makebox(0,0){$
\begin{CD}@>{z \vartheta_1 - z}>>\end{CD}$}}
\put(175,45){\makebox(0,0){$
\begin{CD}@>{z\vartheta_1}>>\end{CD}$}}
\put(125,45){\makebox(0,0){$
\begin{CD}@>{Q^{-1}z\vartheta_1}>>\end{CD}$}}
\put(75,45){\makebox(0,0){$
\begin{CD}@>{Q^{-1}z\vartheta_1}>>\end{CD}$}}
\put(25,45){\makebox(0,0){$
\begin{CD}@>{Q^{-1}z\vartheta_1}>>\end{CD}$}}

\put(275,10){\makebox(0,0){$
\begin{CD} @<<{Q^{-1}z\vartheta_2}< \end{CD}$}}   
\put(225,10){\makebox(0,0){$
\begin{CD} @<<{Q^{-1}z\vartheta_2}< \end{CD}$}}   
\put(175,10){\makebox(0,0){$
\begin{CD} @<<{Q^{-1}z\vartheta_2}< \end{CD}$}}   
\put(125,10){\makebox(0,0){$
\begin{CD} @<<{z\vartheta_2}< \end{CD}$}}   
\put(75,10){\makebox(0,0){$
\begin{CD} @<<{z\vartheta_2-z}< \end{CD}$}}   
\put(25,10){\makebox(0,0){$
\begin{CD} @<<{z\vartheta_2-2z}< \end{CD}$}}   

\Thicklines
\put(150,30){\vector(1,0){50}} 
\put(150,30){\vector(-1,0){50}} 
\put(197,30){\makebox(0,0){$\blacktriangleright$}}
\put(103,30){\makebox(0,0){$\blacktriangleleft$}}

\end{picture} 
\end{center} 
\caption{$\cM_{\rm red}(\bSigma_{\PP^1})$}
\label{fig:P1_Dmod}
\end{figure} 

From the relation $(z \vartheta_1 z\vartheta_2 - Q)\bunit_0 = 0$,  
we obtain the presentation: 
\begin{align*} 
\cM_{\rm red} (\bSigma_{\PP^1}) 
\cong \C[z,Q] \langle z \vartheta_1, z\vartheta_2\rangle 
\big/ \langle z \vartheta_1 z \vartheta_2 - Q  \rangle.  
\end{align*} 
The $\C[\chi,z,Q]$-basis $\bunit_0, \bunit_1$ 
of $\cM_{\rm red}(\bSigma_{\PP^1})$  
corresponds to the basis $\{1, u_1 \} \subset H_\T^*(\PP^1)$ 
under the mirror isomorphism. 
(In this case, we do not need the completion). 

\subsubsection{The case where $\frX = \PP(1,2)$ and 
$G \neq \varnothing$} 
The stacky fan is given by $\bN = \Z$, $b_1 =2$, $b_2 = -1$. 
We choose $G=\{1\}$. 
Then $\LL =\{(l,2l) : l\in \Z \} \subset 
\Laa =\frac{1}{2} \LL$. 
Let $Q \in \C[\Laa_+]$ be the variable corresponding 
to $(1,2)\in \Laa_+$ and let $y = y_1$ be the 
variable corresponding to $1\in G$. 
We write $\chi$ for the standard generator 
of $\bM = \Hom(\bN,\Z) \cong \Z$. 
We have 
\begin{equation} 
\label{eq:chi_relation_P12}
2 z\vartheta_1 - z \vartheta_2 = \chi.  
\end{equation} 
The actions of $z\vartheta_1,z\vartheta_2,z \partial_y$ on the 
reduced fan $D$-module $\cM_{\rm red}(\bSigma_{\PP(1,2)},\{1\})$ 
are shown in Figure \ref{fig:P12_Dmod}. 
\begin{figure}[htbp]
\begin{center} 
\begin{picture}(300,90)(-20,5) 
\path(-30,30)(300,30) 

\put(280,30){\makebox(0,0){\circle*{3}}} 
\put(230,30){\makebox(0,0){\circle*{3}}} 
\put(180,30){\makebox(0,0){\circle*{3}}} 
\put(130,30){\makebox(0,0){\circle*{5}}} 
\put(80,30){\makebox(0,0){\circle*{3}}} 
\put(30,30){\makebox(0,0){\circle*{3}}} 
\put(-20,30){\makebox(0,0){\circle*{3}}} 

\put(280,20){\makebox(0,0){$3$}} 
\put(230,20){\makebox(0,0){$2$}} 
\put(180,20){\makebox(0,0){$1$}} 
\put(130,20){\makebox(0,0){$0$}} 
\put(80,20){\makebox(0,0){$-1$}} 
\put(30,20){\makebox(0,0){$-2$}} 
\put(-20,20){\makebox(0,0){$-3$}}

\put(255,10){\makebox(0,0){
$\begin{CD} @<<{Q^{-\frac{1}{2}}z\vartheta_2}< \end{CD}$}} 
\put(205,10){\makebox(0,0){
$\begin{CD} @<<{Q^{-\frac{1}{2}}z\vartheta_2}< \end{CD}$}} 
\put(155,10){\makebox(0,0){
$\begin{CD} @<<{Q^{-\frac{1}{2}}z\vartheta_2}< \end{CD}$}} 
\put(105,10){\makebox(0,0){
$\begin{CD} @<<{z\vartheta_2}< \end{CD}$}} 
\put(55,10){\makebox(0,0){
$\begin{CD} @<<{z\vartheta_2-z}< \end{CD}$}} 
\put(5,10){\makebox(0,0){
$\begin{CD} @<<{z\vartheta_2-2z}< \end{CD}$}} 

\put(280,63){\makebox(0,0){
$\begin{CD} @>{\phantom{AB}
z \vartheta_1 - \frac{1}{2} z y \partial_y -z\phantom{AB}}>> 
\end{CD}$}}
\put(180,63){\makebox(0,0){
$\begin{CD} @>{\phantom{ABC} 
z\vartheta_1 - \frac{1}{2}z y\partial_y
\phantom{ABC}}>> \end{CD}$}}
\put(80,63){\makebox(0,0){
$\begin{CD} @>{\phantom{A}
Q^{-1}(z\vartheta_1 - \frac{1}{2}z y \partial_y)
\phantom{A}}>> \end{CD}$}}

\put(230,83){\makebox(0,0){
$\begin{CD} @>{\phantom{AB}z\vartheta_1 
- \frac{1}{2}z y\partial_y-\frac{1}{2}z\phantom{AB}}>> 
\end{CD}$}}
\put(130,83){\makebox(0,0){
$\begin{CD} @>{\phantom{A}Q^{-\frac{1}{2}} 
(z\vartheta_1 - \frac{1}{2}z y\partial_y) \phantom{A}}>> 
\end{CD}$}}
\put(30,83){\makebox(0,0){
$\begin{CD} @>{\phantom{A}Q^{-1} 
(z\vartheta_1 - \frac{1}{2}z y\partial_y) 
\phantom{A}}>> \end{CD}$}}

\put(255,45){\makebox(0,0){
$\begin{CD} @>{z\partial_y}>> \end{CD}$}} 
\put(205,45){\makebox(0,0){
$\begin{CD} @>{z\partial_y}>> \end{CD}$}} 
\put(155,45){\makebox(0,0){
$\begin{CD} @>{z\partial_y}>> \end{CD}$}} 
\put(105,45){\makebox(0,0){
$\begin{CD} @>{Q^{-\frac{1}{2}}z \partial_y}>> \end{CD}$}} 
\put(55,45){\makebox(0,0){
$\begin{CD} @>{Q^{-\frac{1}{2}} z\partial_y}>> \end{CD}$}} 
\put(5,45){\makebox(0,0){
$\begin{CD} @>{Q^{-\frac{1}{2}}z\partial_y}>> \end{CD}$}} 

\Thicklines 
\put(130,30){\vector(1,0){100}} 
\put(130,30){\vector(-1,0){50}} 
\put(83,30){\makebox(0,0){$\blacktriangleleft$}}
\put(227,30){\makebox(0,0){$\blacktriangleright$}}
\end{picture}
\end{center} 
\caption{$\cM_{\rm red}(\bSigma_{\PP(1,2)},\{1\})$} 
\label{fig:P12_Dmod} 
\end{figure} 

Equation \eqref{eq:chi_relation_P12} gives relations  
among consecutive 4 basis elements, for example:
\begin{align*} 
0 = (2z\vartheta_1 - z\vartheta_2 -\chi) \bunit_0 
& = - \bunit_{-1} -\chi \bunit_0 + y \bunit_1  + 2 \bunit_2, \\ 
0  = (2z \vartheta_1 - z \vartheta_2 -\chi) \bunit_{-1} 
& = -\bunit_{-2} - (\chi+z) \bunit_{-1} + y Q^{\frac{1}{2}} \bunit_0 
+ 2 Q^{\frac{1}{2}} \bunit_1. 
\end{align*} 
In fact, $\{\bunit_0, \bunit_{-1}, \bunit_1\}$ gives  
a free $\C[z,\chi,Q^{\frac{1}{2}}, y]$-basis of 
$\cM_{\rm red}(\bSigma_{\PP(1,2)},\{1\})$. 
The actions of $z \vartheta_2$ and $z\partial_y$ in this basis 
are represented by the following matrices: 
\[
\begin{pmatrix}
0 & y Q^{\frac{1}{2}} &  Q^{\frac{1}{2}} \\ 
1    & -\chi &  0  \\ 
0  & 2Q^{\frac{1}{2}} &  0 
\end{pmatrix}, \quad 
\begin{pmatrix} 
0 & Q^{\frac{1}{2}} & \frac{1}{2} \chi \\
0 & 0 & \frac{1}{2} \\
1 & 0 & -\frac{1}{2} y
\end{pmatrix}. 
\]
The basis $\{\bunit_0,\bunit_{-1},\bunit_1\}$ corresponds to 
$\{1,u_2,\unit_1\} \subset H^*_{\CR,\T}(\PP(1,2))$ 
under the mirror isomorphism, where $\unit_1$ is the twisted 
sector supported on $\PP(2)$, and 
the above matrices represent the quantum multiplication 
by $u_2$ and $\unit_1$. 
Here the completion is (again) unnecessary and 
the mirror map is given by 
\[
\tau(y) = y \unit_1.  
\]
The reduced fan $D$-module is generated by $\bunit_0$ 
and defined by the following relations: 
\begin{align*} 
& R_1 = 
\left(z\vartheta_1-\frac{1}{2} z y\partial_y \right) 
z\vartheta_2 (z\vartheta_2 - z) - Q, 
& R_2 & = z \partial_y  z\vartheta_2  - Q^{\frac{1}{2}}, \\
& R_3 = z\vartheta_1 -\frac{1}{2} z y \partial_y - (z \partial_y)^2, 
& R_4 & =Q^{\frac{1}{2}} z \partial_y -  
\left(z\vartheta_1-\frac{1}{2} z y \partial_y \right) 
z\vartheta_2. 
\end{align*} 

\subsubsection{The case where $\frX= [\C^2/\mu_2]$ 
and $G\neq \varnothing$} 
The stacky fan is given by $\bN = \Z^2$, 
$b_1 = (0,1)$, $b_2 = (2,1)$. 
We choose $G = \{ b_3 \}$ with $b_3 = (1,1)$. 
We have $\Laa =0$ (all curves in $\frX$ are constant). 
Let $y$ be the variable corresponding to $b_3 \in G$ 
and let $\chi_1$, $\chi_2$ be the basis of $\bM_\C$ 
dual to $b_1, b_2$. 
We have 
\[
\chi_1 = z \vartheta_1, \quad \chi_2 = z \vartheta_2.  
\]

\begin{figure}[htbp]
\begin{center} 
\begin{picture}(400,115)(0,5)
\path(100,10)(100,120) 
\path(100,10)(320,120)

\put(50,10){\makebox(0,0){$\cdot$}} 
\put(50,60){\makebox(0,0){$\cdot$}} 
\put(50,110){\makebox(0,0){$\cdot$}} 
\put(150,10){\makebox(0,0){$\cdot$}} 
\put(200,10){\makebox(0,0){$\cdot$}} 
\put(200,60){\makebox(0,0){$\cdot$}} 
\put(250,10){\makebox(0,0){$\cdot$}} 
\put(250,60){\makebox(0,0){$\cdot$}} 
\put(300,10){\makebox(0,0){$\cdot$}} 
\put(300,60){\makebox(0,0){$\cdot$}} 
\put(300,110){\makebox(0,0){$\cdot$}} 

\put(100,10){\makebox(0,0){\circle*{3}}}
\put(100,60){\makebox(0,0){\circle*{3}}} 
\put(100,110){\makebox(0,0){\circle*{3}}} 
\put(150,60){\makebox(0,0){\circle*{3}}} 
\put(150,110){\makebox(0,0){\circle*{3}}} 
\put(200,60){\makebox(0,0){\circle*{3}}} 
\put(200,110){\makebox(0,0){\circle*{3}}} 
\put(250,110){\makebox(0,0){\circle*{3}}}
\put(300,110){\makebox(0,0){\circle*{3}}}  

\put(90,10){\makebox(0,0){$\boldsymbol{0}$}} 
\put(90,60){\makebox(0,0){$b_1$}} 
\put(210,55){\makebox(0,0){$b_2$}}
\put(212,112){\makebox(0,0){$2b_3$}}

\put(78,35){\makebox(0,0){${}^{z\vartheta_1- \frac{ z}{2} y\partial_y}$}} 
\put(140,15){\rotatebox{27}{${}^{z\vartheta_2- \frac{z }{2} y\partial_y}$}} 
\put(92,10){\rotatebox{45}{
$\begin{CD}@>{\phantom{ABC}z \partial_y\phantom{ABC}}>>\end{CD}$}}
\put(142,60){\rotatebox{45}{
$\begin{CD}@>{\phantom{AB}z\partial_y \phantom{ABCD}}>>\end{CD}$}}
\put(95,62){\rotatebox{27}{
$\begin{CD}@>{\phantom{ABCD}z\vartheta_2- \frac{z}{2} y\partial_y 
\phantom{ABC}}>>\end{CD}$}}  
\put(198.2,95){\makebox(0,0){
$\begin{CD} @AAA \end{CD}$}}
\put(200,65){\line(0,1){20}} 
\put(222,84){\makebox(0,0){${}^{z\vartheta_1- \frac{z}{2} y\partial_y}$}} 

\Thicklines 
\put(100,10){\vector(0,1){50}} 
\put(100,10){\vector(2,1){100}}

\end{picture}
\end{center} 
\caption{$\cM_{\rm red}(\bSigma_{[\C^2/\mu_2]},\{(1,1)\})$} 
\label{fig:A1_Dmod} 
\end{figure} 

We have the following relation 
in $\cM_{\rm red}(\bSigma_{[\C^2/\mu_2]},G)$, illustrated in Figure \ref{fig:A1_Dmod}):
\[
(z\partial_y )^2 \bunit_0 
= \left(z\vartheta_1 - \frac{1}{2} z y\partial_y \right)
\left(z\vartheta_2 - \frac{1}{2} z y\partial_y\right) 
\bunit_0
\]
Thus, if we invert $4 - y^2$, 
\[
\bunit_{2 b_3} = 
(z\partial_y)^2 \bunit_0 
= \frac{4 \chi_1 \chi_2}{4- y^2} \bunit_0 + 
\frac{(z-2(\chi_1+\chi_2))y}{4-y^2} \bunit_{b_3}. 
\]
The elements $\bunit_0$, $\bunit_{b_3}$ generate 
$\cM_{\rm red}(\bSigma_{[\C^2/\Z_2]},G)[(4-y^2)^{-1}]$ 
freely over $\C[z,\chi_1,\chi_2,y,(4-y^2)^{-1}]$. 
In an analytic neighbourhood of $y=0$, we define 
\[
\hbunit_{b_3}:= \sqrt{1-(y^2/4)} \bunit_{b_3}
\] 
and make 
the co-ordinate change $y = 2 \sin(\theta/2)$. 
In the basis $\{\bunit_0,\hbunit_{b_3}\}$, 
the action of $z (\partial/\partial \theta)$ 
is represented by the following $z$-independent matrix: 
\[
\begin{bmatrix}
0 & \chi_1 \chi_2 \\ 
1 & - (\chi_1 + \chi_2) \sin(\frac{\theta}{2})  
\end{bmatrix}.  
\]
The basis $\{\bunit_0,\hbunit_{b_3}\}$ corresponds to 
$\{1, \unit_{b_3}\} 
\subset H^2_{\CR, \T}([\C^2/\mu_2])$ under 
the mirror isomorphism,
and the mirror map is given by $\tau(y) = \theta \unit_{b_3}$. 
The above matrix gives the quantum multiplication 
by $\unit_{b_3}$. 

\subsubsection{The case 
where $\frX=\PP^2$ and $G\neq \varnothing$} 
We take $\bN = \Z^2$, $b_1=(1,0)$, $b_2=(0,1)$, 
$b_3 =(-1,-1)$ and $G=\{b_4\}$ with $b_4 = (1,1) = b_1 + b_2$. 
Let $Q$ be the variable corresponding to $(1,1,1) \in \Laa 
= \LL = \Z(1,1,1)$ 
and let $y$ be the variable corresponding to $b_4\in G$. 
We have
\[
\chi_1 = z \vartheta_1 - z \vartheta_3, \quad 
\chi_2 = z \vartheta_2 - z \vartheta_3. 
\]
In the reduced fan $D$-module $\cM_{\rm red}(\bSigma_{\PP^2}, G)$, 
we have the relation 
\begin{align*} 
\left( (z\vartheta_2 - z y\partial_y) (z \vartheta_1 - z y\partial_y) 
- z \partial_y \right) \bunit_0 & =0, \\
\left( z\vartheta_3 (z\vartheta_2 - z y\partial_y) (z \vartheta_1 - z y\partial_y) 
- Q \right) \bunit_0 & = 0. 
\end{align*} 
Let us consider the non-equivariant limit where $\chi_1=\chi_2=0$.  
Then we have $z \vartheta := z\vartheta_1 = z\vartheta_2 = z \vartheta_3
= z Q (\partial/\partial Q)$. 
We can see that the $D$-module 
$\cM_{\rm red}(\bSigma_{\PP^2},G)/(\chi_1,\chi_2)$
is of rank $4$ at $y\neq 0$. 
On the other hand, the $y$-adic completion of this $D$-module 
has a basis $\{\bunit_0, z\vartheta \bunit_0, (z\vartheta)^2\bunit_0\}$ 
over $\C[z,Q][\![y]\!]$ and is isomorphic to the quantum $D$-module 
of $\PP^2$. 
For example, $z \partial_y \bunit_0$ can be expressed 
as a linear combination of these basis elements, by using the above two equations 
recursively. 

\section{The higher residue pairing and the Poincar\'e pairing match}
\label{sec:pairing} 

We now construct a version of K.~Saito's \emph{higher residue pairing} 
\cite{SaitoK:higherresidue} 
on our (equivariant) Gauss--Manin system, and show that it matches 
the (equivariant) Poincar\'e pairing on quantum cohomology 
under the mirror isomorphism in Theorem \ref{thm:mirror_isom}. 
As in \S \ref{subsec:unfolding}, 
we fix a finite subset $G\subset (\bN\cap |\Sigma|)  
\setminus \{b_1,\dots,b_m\}$ and consider the unfolding 
$F_\chi(x;y)$ associated with $G$. We again set 
$S := \{b_1,\dots,b_m\} \cup G$. 

\subsection{Critical points} 
We start with a description of critical points of $F_\chi$. 
The (logarithmic) critical scheme of $F_\chi$ is defined to be the formal 
spectrum of the Jacobian ring 
\[
\Jac(F_\chi) = 
\C\{\OO_+\}[\![y]\!][\chi]\bigg/\left\langle x_i \parfrac{F_\chi}{x_i} : 
1\le i\le n \right\rangle  
\]
where $x_i \parfrac{F_\chi}{x_i} = \sum_{\bk \in S} y_\bk k_i w_\bk - \chi_i$, 
with $k_i$ being the $i$th component of $\overline{\bk} \in \bN_\R \cong \R^n$.  
The mirror isomorphism $\Theta$ 
in Theorem \ref{thm:mirror_isom} induces, at $z=0$, an isomorphism 
between $\Jac(F_\chi)$ and 
$H^*_{\CR,\T}(\frX)\otimes_{R_\T} 
R_\T[\![\Laa_+]\!][\![y]\!]$ 
as $R_\T[\![\Laa_+]\!][\![y]\!]$-modules. 
In particular, $\Jac(F_\chi)$ is a free module 
over $R_\T[\![\Laa_+]\!][\![y]\!]$ of rank $\dim H^*_{\CR}(\frX)$. 
Set $N := \dim H^*_{\CR}(\frX)$. 
Let $\overline{S}_\T$ be the algebraic closure of the fraction field $S_\T$ of $R_\T$. 
We show that, after base change to 
$\overline{S}_\T[\![\Laa_+]\!][\![y]\!]$, the critical 
scheme consists of $N$ distinct points over $\overline{S}_\T[\![\Laa_+]\!][\![y]\!]$, 
each of which is characterized by its limit at $Q=0, y =y^*$ 
-- see \eqref{eq:specialization_y} for $y^*$. 
We write $\Jac(F_\chi)_{\overline{S}_\T} 
:= \Jac(F_\chi) \otimes_{R_\T[\![\Laa_+]\!][\![y]\!]} 
{\overline{S}_\T}[\![\Laa_+]\!][\![y]\!]$ for the base change. 

\begin{notation} 
\label{nota:fixed_points}
Recall that maximal cones $\sigma\in \Sigma$ are in one-to-one 
correspondence with 
$\T$-fixed points $z_\sigma$ of $\frX$. 
We write $u_i(\sigma) \in H^2_\T(\pt) = \bM_\Q$ 
for the restriction of the toric divisor class $u_i$ to the fixed point $z_\sigma$. 
We set $\bN(\sigma) := \bN/\sum_{i\in \sigma} \Z b_i$; this gives the 
orbifold isotropy group at the fixed point $z_\sigma$. 
We also write $\Sigma(n)$ for the set of 
$n$-dimensional (i.e.~maximal) cones in $\Sigma$. 
\end{notation} 

\begin{lemma}
\label{lem:critical_points} 
{\rm (1)}  
The ring $\Jac(F_\chi)_{\overline{S}_\T}$ is isomorphic to 
${\overline{S}_\T}[\![\Laa_+]\!][\![y]\!]^{\oplus N}$ as an 
${\overline{S}_\T}[\![\Laa_+]\!][\![y]\!]$-algebra, 
with $N= \dim H^*_{\CR,\T}(\frX)$. 
In other words, the critical scheme $\Spf(\Jac(F_\chi)_{\overline{S}_\T})$ 
consists of $N$ distinct points 
over ${\overline{S}_\T}[\![\Laa_+]\!][\![y]\!]$. 

{\rm (2)} For a maximal cone $\sigma$, define $\Crit(\sigma)$ 
to be the set:  
\[
\Crit(\sigma) := \left\{ c= (c_\bk) \in {\overline{S}_\T}^{\bN \cap \sigma} : 
\text{$c_\bk c_\bl = c_{\bk+\bl}$, $c_{b_i} = u_i(\sigma)$ for all $i\in \sigma$}
\right\},  
\] 
where we set $\bN \cap \sigma = \{ \bk\in \bN : \overline{\bk} \in \sigma\}$. 
Then $\Crit(\sigma)$ is a torsor over the character group 
$\hbN(\sigma)= \Hom(\bN(\sigma),\C^\times)$ 
of $\bN(\sigma)$. 

{\rm (3)} 
Note that a critical point $p$ over ${\overline{S}_\T}[\![\Laa_+]\!][\![y]\!]$ 
is by definition 
an ${\overline{S}_\T}[\![\Laa_+]\!][\![y]\!]$-algebra 
homomorphism $\Jac(F_\chi)_{\overline{S}_\T}\to 
{\overline{S}_\T}[\![\Laa_+]\!][\![y]\!]$, 
$[w_\bk] \mapsto w_\bk(p)$. For each maximal cone $\sigma\in \Sigma$ 
and an element $c\in \Crit(\sigma)$, there exists a unique critical point $p$ 
such that 
\[
w_\bk(p)\Big|_{Q=0, y=y^*} = \begin{cases} c_\bk 
&  \bk \in \bN\cap \sigma  \\ 
0 & \text{otherwise}. 
\end{cases} 
\]
This gives a bijection between critical points over 
${\overline{S}_\T}[\![\Laa_+]\!][\![y]\!]$ 
and the set $\bigcup_{\sigma\in \Sigma(n)} \Crit(\sigma)$. 
\end{lemma} 
\begin{proof} 
Part (2) is obvious. 
Note that $\Jac(F_\chi)_{\overline{S}_\T}$ is a free 
${\overline{S}_\T}[\![\Laa_+]\!][\![y]\!]$-module 
of rank $N$. Therefore $\Jac(F_\chi)_{\overline{S}_\T}$ 
is a direct sum of copies of ${\overline{S}_\T}[\![\Laa_+]\!][\![y]\!]$ 
as a ring if and only if the restriction 
to $Q=0, y=y^*$
\[
\Jac(F_\chi)_{{\overline{S}_\T},0} := 
\Jac(F_\chi)_{\overline{S}_\T} 
\otimes_{{\overline{S}_\T}[\![\Laa_+]\!][\![y]\!]} {\overline{S}_\T} 
\] 
is a direct sum of copies of ${\overline{S}_\T}$ as a ring. 
To establish (1) and (3), therefore, it suffices to prove that  
$\Spec(\Jac(F_\chi)_{{\overline{S}_\T},0})$ is a finite set of reduced 
points and equals $\bigcup_\sigma \Crit(\sigma)$. 
Note that we have $F_\chi|_{Q=0,y=y^*} = 
w_1+ \cdots + w_m - \sum_{i=1}^n \chi_i \log x_i$ and 
\[
\Jac(F_\chi)_{{\overline{S}_\T},0} \cong  
\frac{\bigoplus_{\bk\in \bN\cap |\Sigma|} {\overline{S}_\T} w_\bk} 
{\langle \chi_i - \sum_{j=1}^m b_{j,i}w_j : 1\le i \le n 
\rangle} 
\]
where the product structure on $\bigoplus_{\bk \in \bN \cap |\Sigma|} 
{\overline{S}_\T} w_\bk$ is defined as in \eqref{eq:equivariant_CR_product} 
(the Stanley--Reisner presentation of $H^*_{\CR,\T}(\frX)$),  
but replacing $\phi_\bk$ with $w_\bk$, 
and $b_{j,i}$ denotes the $i$th component of 
$\overline{b}_j \in \bN_\R \cong \R^n$. 
Therefore an ${\overline{S}_\T}$-algebra homomorphism 
$\Jac(F_\chi)_{{\overline{S}_\T},0}\to {\overline{S}_\T}$ 
is specified by a tuple $(c_\bk) \in {\overline{S}_\T}^{\bN\cap |\Sigma|}$ 
satisfying the conditions: 
\[
c_\bk c_\bl = \begin{cases} 
c_{\bk+\bl}  & \text{if $\bk,\bl$ are in the same cone of $\Sigma$}; \\
0 & \text{otherwise} 
\end{cases} 
\quad \text{and} 
\quad 
\chi_i = \sum_{j=1}^m b_{i,j} c_{b_j}. 
\]
The first condition implies that the support $\{\bk : c_{\bk} \neq 0\}$ 
has to be contained in some maximal cone $\sigma$ of $\Sigma$. 
The second condition then determines $c_{b_1},\dots,c_{b_m}$ 
uniquely. 
Notice that $u_j(\sigma) = 0$ for $j \notin \sigma$ since 
the toric divisor $\{Z_j=0\}$ does not pass through 
the fixed point $z_\sigma$, and that 
$\sum_{j=1}^m b_{j,i} u_j = \chi_i$ by the description of  the 
$R_\T$-module structure of $H^*_{\CR,\T}(\frX)$ in \S \ref{subsec:inertia}. 
It follows that $c_{b_j} = 0$ for $j\notin \sigma$ and 
$c_{b_j} = u_j(\sigma)$ for $j\in \sigma$. 
Thus closed points of $\Spec(\Jac(F_\chi)_{{\overline{S}_\T},0})$ correspond 
bijectively with elements of $\bigcup_\sigma \Crit(\sigma)$. 
Since we have $\dim_{\overline{S}_\T} \Jac(F_\chi)_{{\overline{S}_\T},0} =N = 
\# \bigcup_\sigma \Crit(\sigma) $, 
$\Spec (\Jac(F_\chi)_{{\overline{S}_\T},0})$ consists only of reduced 
points and is identified with $\bigcup_\sigma \Crit(\sigma)$. 
\end{proof} 

\begin{remark} 
In the above proof, we have shown that 
\begin{equation} 
\label{eq:chi_u_sigma}
\chi_i = \sum_{j=1}^m b_{j,i} u_j(\sigma) = \sum_{j\in \sigma} b_{j,i} u_j(\sigma)
\end{equation}
for any maximal cone $\sigma$. 
\end{remark} 
We study the $Q\to 0$, $y\to y^*$ asymptotics of 
critical values of $F_\chi$. 
The problem here is that 
the critical value $F_\chi(p)$ of $F_\chi$ does not lie in the ring 
$\overline{S}_\T[\![\Laa_+]\!][\![y]\!]$ because 
of the $\log x_i$ term. 
Let $p$ be a critical point corresponding to 
$c =(c_\bk)_{\bk\in \bN\cap \sigma} \in 
\Crit(\sigma)$, where $\sigma$ is a maximal cone of $\Sigma$, 
as in Lemma \ref{lem:critical_points}.  
We extend the function $\bk \mapsto c_\bk\in \overline{S}_\T$ 
for arbitrary $\bk \in \bN$ by requiring that 
$c_\bk c_\bl = c_{\bk+\bl}$ for all $\bk,\bl \in \bN$. 
This is possible since $c_\bk \neq 0$ for all $\bk \in \bN \cap \sigma$. 
Recall that we chose a decomposition 
$\bN \cong \Z^n \times \bN_{\rm tor}$ in \S\ref{subsec:LG}. 
We write $\be_i$, $1\le i\le n$, for the element of $\bN$ 
corresponding to the $i$th basis vector of $\Z^n$. 
Let $\overline\varsigma_\sigma \colon \bN_\Q \to \Q^m$ 
denote the splitting of the fan sequence defined by 
$\overline{\varsigma}_\sigma(b_i) = e_i$ for $i\in \sigma$. 
Then we have 
\begin{equation} 
\label{eq:x_i_sigmachart} 
x_i = w^{(\overline{\varsigma}(\be_i),\be_i)} 
= Q^{(\overline{\varsigma} - \overline{\varsigma}_\sigma)(\be_i)}
w^{(\overline{\varsigma}_\sigma(\be_i),\be_i)}
\end{equation} 
and 
\[
w^{(\overline{\varsigma}_\sigma(\be_i),\be_i)}(p) =c_{\be_i} + \frakm^G 
\overline{S}_\T[\![\Laa_+]\!][\![y]\!] 
\]
by the extension of the definition of $c_\bk$ above. 
(Recall that $\frakm^G$ is the ideal corresponding to $Q=0$, $y=y^*$ 
introduced before Proposition \ref{pro:Loc0}.) 
Therefore we can decompose the critical value $F_\chi(p)$ as 
\[
F_\chi(p) = F^\cl_\chi(p) + F^\qu_\chi(p)  
\]
where the classical part $F^\cl_\chi(p)$ given by 
\[
F^\cl_\chi(p) = \sum_{j\in \sigma} c_{b_j} - \sum_{i=1}^n 
\chi_i \left( \log Q^{(\overline\varsigma-\overline\varsigma_\sigma)(\be_i)} 
+ \log c_{\be_i} \right) 
\]
and the quantum part $F^\qu_\chi(p)$ belongs to $\frakm^G 
\overline{S}_\T[\![\Laa_+]\!][\![y]\!]$: 
\begin{equation}
\label{eq:qu_crit} 
F^\qu_\chi(p) = \sum_{j\in S} y_\bk w_\bk(p) 
-\sum_{j\in \sigma} c_{b_j} 
+ \sum_{i=1}^n \chi_i 
\log(c_{\be_i}^{-1}w^{(\varsigma_\sigma(\be_i),\be_i)}(p)). 
\end{equation} 
\begin{lemma} 
\label{lem:cl_criticalvalue} 
Let $p$ be a critical point corresponding to $c\in \Crit(\sigma)$, 
where $\sigma$ is a maximal cone of $\Sigma$. 
Write $b_j = \sum_{i=1}^n b_{j,i} \be_i +\zeta_j$ with 
$\zeta_j \in \bN_{\rm tor}$.  
Then we have 
\begin{align*} 
F^\cl_\chi(p) & = \sum_{j\in \sigma} 
\left (u_j(\sigma) - u_j(\sigma) 
\log \frac{u_j(\sigma)}{c_{\zeta_j}} \right)
+ \log Q^{\X(\sigma)} 
\end{align*} 
where $\X \in H^2_\T(X,\Q) \otimes \LL 
= \Hom(\LL_\Q^\star,H^2_\T(X,\Q))$ 
denotes the element corresponding to the splitting 
$\LL_\Q^\star \to (\Q^m)^\star = H^2_\T(X,\Q)$ 
dual to $\varsigma$ (in the sense explained before Proposition \ref{pro:Loc_diffeq}), 
and $\X(\sigma)\in \bM_\Q \otimes \LL$ 
denotes the restriction of $\X$ to the $\T$-fixed point 
$z_\sigma\in \frX$. 
\end{lemma} 
\begin{proof} 
Recall that $c_{b_j} = u_j(\sigma)$ for $j\in \sigma$. 
Using \eqref{eq:chi_u_sigma}, we find that 
\begin{align*} 
\sum_{i=1}^n \chi_i \log c_{\be_i} &= \sum_{i=1}^n \sum_{j\in \sigma} 
u_j(\sigma) b_{j,i} \log c_{\be_i} 
= \sum_{j\in \sigma} u_j(\sigma) 
\log ( c_{\zeta_j}^{-1} u_j(\sigma) ) 
\end{align*} 
where we used $u_j(\sigma) = c_{b_j} =  c_{\zeta_j}\prod_{i=1}^n 
c_{\be_i}^{b_{j,i}}$. It remains to show that 
\begin{equation} 
\label{eq:Xsigma}
\X(\sigma) = \sum_{i=1}^n \chi_i 
(\overline{\varsigma}_\sigma(\be_i) - \overline{\varsigma}(\be_i)). 
\end{equation} 
Evaluating the right-hand side at $\xi \in \LL_\Q^\star$, 
we obtain 
\begin{align*} 
\sum_{i=1}^n \chi_i 
\left( \xi\cdot (\overline{\varsigma}_\sigma- \overline{\varsigma})(\be_i) 
\right) 
& = 
\sum_{i=1}^n \chi_i \left( 
\hxi \cdot \overline{\varsigma}_\sigma(\be_i)  
- \hxi \cdot \overline{\varsigma}(\be_i) \right) \\ 
& = \sum_{j\in \sigma} u_j(\sigma) (\hxi\cdot e_j) = \hxi(\sigma) 
\end{align*} 
where $\hxi\in (\Q^m)^\star \cong H^2_\T(X,\Q)$ denotes the lift of $\xi$ 
with respect to the splitting $\LL_\Q^\star \to (\Q^m)^\star$ 
and 
$\hxi(\sigma)\in \bM_\Q$ denotes the restriction of $\hxi$ to 
$z_\sigma$. We also used \eqref{eq:chi_u_sigma} and 
the fact that $\hxi$ vanishes on the image of $\overline\varsigma$. 
Evaluation of the left-hand side at $\xi$ gives the same answer, 
and the conclusion follows. 
\end{proof} 

We also study the limit of the Landau--Ginzburg potential 
at the central fiber $Q=0$, $y=y^*$. 
The fiber at $Q=0$ of the total space $\Spec \C[\OO_+]$ of the mirror 
is reducible and decomposed as follows: 
\[
\Spec \C[\OO_+]\Bigr|_{Q=0} = \bigcup_{\sigma\in \Sigma(n)} 
\bigcup_{\theta \in \Hom(\bN_{\rm tor},\C^\times)} 
A_{\sigma,\theta} 
\]
where $A_{\sigma,\theta} = \Spec \C[(\bN/\bN_{\rm tor}) \cap \sigma]$ 
is an affine toric variety. 
The restriction of $F_\chi$ to the central fiber is ill-defined 
because of the logarithmic term, but we show that 
for each critical point $p$, the difference $F_\chi(x;y)- F_\chi^\cl(p)$ 
has a well-defined limit as $y$ approaches $y^*$ and 
$(x,Q)$ approaches the component 
$A_{\sigma,\theta}$ on which the critical point $p$ lies. 

\begin{lemma}
\label{lem:limit_on_component} 
Let $p$ be a critical point of $F_\chi$ corresponding to 
$c \in \Crit(\sigma)$, where $\sigma$ is a maximal cone of $\Sigma$. 
Let $A_{\sigma,\theta}$ be the component of the central fiber 
determined by the character $\theta$ defined as 
the restriction of $\bk \mapsto c_\bk$ to $\bk\in \bN_{\rm tor}$.  
Then we have 
\[
F_\chi - F^\cl_\chi(p) \Bigr|_{ Q=0, y=y^*, A_{\sigma,\theta}} 
= \sum_{j\in \sigma} 
\left(w_j - u_j(\sigma) - u_j(\sigma) \log\frac{w_j}{u_j(\sigma)} 
\right).  
\]
\end{lemma} 
\begin{proof} 
  Let $p$,~$c$,~$\sigma$,~$\theta$ be as given. 
We have 
\begin{align*}
F_\chi & = \sum_{\bk \in S} y_\bk w_\bk - \sum_{i=1}^n \chi_i 
\log (Q^{(\overline{\varsigma}-\overline{\varsigma}_\sigma)(\be_i)} 
w^{(\overline{\varsigma}_\sigma(\be_i),\be_i)}) \\
& = \sum_{\bk\in S} y_\bk w_\bk  + \log Q^{\X(\sigma)} 
- \sum_{j\in \sigma} u_j(\sigma) \log \frac{w_i}{w^{(0,\zeta_j)}}
\end{align*} 
where we used \eqref{eq:x_i_sigmachart}, \eqref{eq:Xsigma} 
and \eqref{eq:chi_u_sigma}.
(The quantity $\zeta_j \in \bN_{\rm tor}$ here is given in Lemma \ref{lem:cl_criticalvalue}). 
Note that $w_\bk|_{A_{\sigma,\theta}}$ equals $u_j(\sigma)$ if $\bk = b_j \in \sigma$ 
and is zero otherwise; also  
$w^{(0,\zeta_j)}|_{A_{\sigma,\theta}} = c_{\zeta_j}$. 
The conclusion follows easily from this and 
Lemma \ref{lem:cl_criticalvalue}. 
\end{proof}

\subsection{Higher residue pairing via asymptotic expansion} 
\label{subsec:HRP}
We define the higher residue pairing in our setting 
in terms of (formal) asymptotic expansion, following the method of 
Pham \cite[2\`{e}me Partie, 4]{Pham:Lefschetz}. 
We first recall asymptotic expansion in analytic setting. 
Let $f(t) = f(t^1,\dots,t^n)$ be a holomorphic function 
on $\C^n$ with a non-degenerate critical point $p$. 
Let $\Gamma(p)$ denote a 
stable manifold for the Morse function $\Re(f(t))$ associated to $p$
and consider the oscillatory integral: 
\[
\int_{\Gamma(p)} e^{f(t)/z} g(t) dt^1 \cdots dt^n 
\]
with $z<0$ and $g(t)$ a holomorphic function. 
As $z \nearrow 0$, the integral is dominated by the contribution 
near the critical point $t=p$, and 
we obtain its asymptotic expansion by expanding 
the integrand in Taylor series at $p$ (under appropriate 
assumptions on $f$ and $g$). A concrete method 
is as follows: 
we expand the functions $f(t)$, $g(t)$ in 
Taylor series at $p$ as 
\[
f(t) = \sum_{k\ge 0} \frac{1}{k!} \sum_{i_1,\dots,i_k} 
f_{i_1,\dots, i_k}(p) s^{i_1} \cdots 
s^{i_k}, 
\quad 
g(t) = \sum_{k\ge 0} \frac{1}{k!} 
\sum_{i_1,\dots,i_k} g_{i_1,\dots,i_k}(p) s^{i_1} 
\cdots s^{i_k}
\] 
with $s^i = t^i - p^i$, and make a linear change of variables 
$s^i = \sqrt{-z} \sum_{j=1}^n c^i_j v^j$ such that 
\[
\frac{1}{2z} \sum_{i,j} f_{ij}(p) s^i s^j =  - \frac{1}{2}
\sum_{i} (v^i)^2.  
\]
Then the above integral can be expanded as: 
\begin{equation} 
\label{eq:Gaussian}
e^{f(p)/z} 
\frac{(-z)^{n/2}}{\sqrt{\det(f_{ij}(p))}}
\int_{\R^n} 
e^{-\frac{1}{2} \sum_{i} (v^i)^2} 
\bigg( \sum_{\substack{ k,l\ge 0, \\ k\equiv l \bmod 2} } \frac{1}{k!} 
a_{i_1,\dots,i_k}^{(l)} v^{i_1} \cdots v^{i_k} z^{l/2} 
\bigg) 
dv^1 \cdots dv^n 
\end{equation} 
where $\sum_{k,l\ge 0} \frac{1}{k!} 
a_{i_1,\dots,i_k}^{(l)} v^{i_1} \cdots v^{i_k} z^{l/2}$ 
is the expansion of  
\begin{multline*} 
\exp\left( - 
\sum_{k \ge 3} \sum_{i_1,\dots,i_k} \sum_{j_1,\dots,j_k} 
\frac{(-z)^{k/2-1}}{k!} f_{i_1,\dots, i_k}(p) c^{i_1}_{j_1} 
\cdots c^{i_k}_{j_k} v^{j_1} \cdots v^{j_k} \right) \\ 
\times 
\sum_{k\ge 0} \sum_{i_1,\dots,i_k} 
\sum_{j_1,\dots,j_k} \frac{(-z)^{k/2}}{k!} g_{i_1,\dots,i_k}(p) 
c^{i_1}_{j_1} \cdots c^{i_k}_{j_k} v^{j_1} \cdots v^{j_k}. 
\end{multline*} 
By performing the above Gaussian integral termwise, 
we obtain a \emph{formal asymptotic expansion} of the original oscillatory 
integral (note that half-integer powers of $z$ vanish automatically). 
We denote this expansion as 
\[
e^{f(p)/z} (-2\pi z)^{n/2} \Asym_p\left( e^{f(t)/z} g(t) dt \right) 
\] 
so that $\Asym_p(e^{f(t)/z} g(t) dt)$ is of the form: 
\[
\frac{1}{\sqrt{\det(f_{ij}(p))}} \left( g(p) + a_1 z + a_2 z^2+ \cdots \right). 
\]
We remark that the formal asymptotic expansion vanishes if 
$e^{f(t)/z} g(t) dt$ is an exact form (regardless of 
the formal asymptotic expansion being the actual asymptotic 
expansion, in which case the remark is obvious), 
since the integrand of the corresponding Gaussian integral \eqref{eq:Gaussian} 
becomes also exact, and it is straightforward 
to check that the termwise Gaussian integral of a formal exact 
form is zero. 

Note that the above procedure only involves 
the Taylor expansions of the functions $f(t)$, $g(t)$ at $p$  
and an orthogonalization of the Hessian form $(f_{ij}(p))$. 
Therefore we can generalize the above procedure to 
our setting where $f(t)$ is given by $F_\chi(x;y)$ 
and $g(t) dt$ is an element of 
$\C[z]\{\OO_+\}[\![y]\!][\chi] \omega$, 
by using $(\log x_1,\dots,\log x_n)$ as co-ordinates 
$(t^1,\dots,t^n)$ here. 
To be more precise, we work over the ring ${\overline{S}_\T}[\![\Laa_+]\!][\![y]\!]$ 
and consider the asymptotic expansion at one of 
$N$ non-degenerate critical points over ${\overline{S}_\T}[\![\Laa_+]\!][\![y]\!]$ 
from Lemma \ref{lem:critical_points}. 
Let $p$ be a critical point of $F_\chi$ over ${\overline{S}_\T}[\![\Laa_+]\!][\![y]\!]$. 
Then the Taylor expansion of a function in $\C[z]\{\OO_+\}[\![y]\!][\chi]$ 
at $p$ (with respect to the co-ordinates $\log x_1,\dots,\log x_n$) 
is well-defined as formal power series with coefficients in 
${\overline{S}_\T}[z][\![\Laa_+]\!][\![y]\!]$. 
When the critical point $p$ 
corresponds to an element of $\Crit(\sigma)$ 
(under the bijection in Lemma \ref{lem:critical_points}), 
\[
\parfrac{^2 F_\chi}{\log x_i \partial \log x_j}(p)\biggr|_{Q=0,y=y^*} 
= \sum_{k\in \sigma} b_{k,i} b_{k,j} u_k(\sigma) 
\] 
is diagonalizable 
by the matrix $(b_{k,i} \sqrt{u_k(\sigma)})_{k\in \sigma, 1\le i\le n}$, 
and thus the Hessian form is diagonalizable over 
${\overline{S}_\T}[\![\Laa_+]\!][\![y]\!]$. 
Thus we obtain a well-defined map: 
\[
\Asym_p \colon e^{F_\chi/z} \C[z]\{\OO_+\}[\![y]\!][\chi] \omega 
\to {\overline{S}_\T}[\![\Laa_+]\!][\![y]\!][\![z]\!] 
\]
for each critical point $p$ over ${\overline{S}_\T}[\![\Laa_+]\!][\![y]\!]$. 
By the remark above, $\Asym_p(e^{F_\chi/z} \Omega)$ vanishes 
if $e^{F_\chi/z}\Omega$ is exact, and thus $\Asym_p$ 
descends to cohomology: 
\[
\Asym_p \colon e^{F_\chi/z} \GM(F_\chi) \to 
{\overline{S}_\T}[\![\Laa_+]\!][\![y]\!][\![z]\!]. 
\]

\begin{definition} 
\label{def:HRP} 
We define the higher residue pairing $P \colon \GM(F_\chi) 
\times \GM(F_\chi) \to {\overline{S}_\T}[\![\Laa_+]\!][\![y]\!][\![z]\!]$ by 
\[
P(\Omega_1,\Omega_2) = \sum_{p} \overline{\Asym_p 
(e^{F_\chi/z} \Omega_1)} \Asym_p( e^{F_\chi/z}\Omega_2) 
\]
for $\Omega_1,\Omega_2 \in \GM(F_\chi)$, where  $\overline{\Asym_p(e^{F_\chi/z}\Omega_1)} 
= \Asym_p(e^{F_\chi/z} \Omega_1)|_{z\to -z}$, and the sum is over critical 
points $p$ of $F_\chi$ over ${\overline{S}_\T}[\![\Laa_+]\!][\![y]\!]$ 
in Lemma \ref{lem:critical_points}. 
Note that the higher residue pairing is invariant under the Galois group 
$\Gal({\overline{S}_\T}/S_\T)$, which permutes critical points, and thus $P$ takes values in 
$S_\T[\![\Laa_+]\!][\![y]\!][\![z]\!]$. 
(In fact it takes values in $S_\T[z][\![\Laa_+]\!][\![y]\!]$ 
by Theorem \ref{thm:pairings_match}.) 
\end{definition} 

We establish standard properties of the higher residue pairing. 
We first observe that $\Asym_p$ gives another solution 
to the equivariant Gauss--Manin system (cf.~the localization 
map $\Loc$ in \S\ref{subsec:sol_free}). 
We remark that, despite the fact that $e^{F^\qu_\chi(q)/z}$ is 
a formal power series in $z^{-1}$ and that $\Asym_p(\Omega)$ 
is a formal power series in $z$, the product 
$e^{F^\qu_\chi(p)/z} \Asym_p(\Omega)$ is well-defined 
as an element of $\overline{S}_\T(\!(z)\!)[\![\Laa_+]\!][\![y]\!]$ 
because $F^\qu_\chi(p)\in \frakm^G \overline{S}_\T[\![\Laa_+]\!][\![y]\!]$.

\begin{lemma} 
\label{lem:asymp_diffeq}
Let $p$ be a critical point of $F_\chi$ over ${\overline{S}_\T}[\![\Laa_+]\!][\![y]\!]$. 
The map $\Asym_p$ is linear over $R_\T[z][\![\Laa_+]\!][\![y]\!]$  
and satisfies the following differential equations: 
\begin{align*} 
e^{F_\chi^\qu(p)/z} \Asym_p(e^{F_\chi/z}\nabla_{\xi Q\parfrac{}{Q}} \Omega) 
& = 
\left(\xi Q\parfrac{}{Q} + \frac{1}{z}\hxi(\sigma)\right)
\left( e^{F_\chi^\qu(p)/z} \Asym_p(e^{F_\chi/z}\Omega) \right), \\ 
e^{F^\qu_\chi(p)/z} \Asym_p(e^{F_\chi/z}\nabla_{\parfrac{}{y_\bk}} \Omega)
& = 
\parfrac{}{y_\bk} \left( e^{F^\qu_\chi(p)/z} 
\Asym_p(e^{F_\chi/z} \Omega) \right), \\ 
\Asym_p(e^{F_\chi/z}\Gr^{\rm B} \Omega) & = 
\left( z\parfrac{}{z} + \cE^{\rm B} + \frac{n}{2}\right) 
\Asym_p(e^{F_\chi/z}\Omega), 
\end{align*} 
where $\xi \in \LL^\star_\C$, $\hxi\in (\C^m)^\star = H^2_\T(X,\C)$ 
is the lift of $\xi$ 
introduced before Proposition \ref{pro:Loc_diffeq}, 
and $\hxi(\sigma)\in \bM_\C$ denotes the restriction of 
$\hxi$ to the fixed point $z_\sigma \in \frX$. 
Moreover the quantum part $F^\qu_\chi(p)$ of the critical 
value is homogeneous of degree one, i.e.~$\cE^{\rm B} F^\qu_\chi(p) 
= F^\qu_\chi(p)$. 
\end{lemma} 
\begin{proof} 
It is clear from the definition that $\Asym_p$ is linear 
over $\C[z][\![\Laa_+]\!][\![y]\!][\chi]$, 
and since $\Asym_p$ is continuous with respect to the $\frakm^G$-adic topology 
(see the discussion before Proposition \ref{pro:Loc0} for $\frakm^G$), it is 
linear over $R_\T[z][\![\Laa_+]\!][\![y]\!]$. 
The first two differential equations follow from:
\begin{itemize}
\item  the fact that 
$e^{F_\chi/z}\Omega \mapsto e^{F_\chi(p)/z} 
\Asym_p(e^{F_\chi/z}\Omega)$ commutes 
with differentiation in the parameters $(y,Q)$;
\item  $\partial_{\vec{v}} ( e^{F_\chi/z}\Omega ) = e^{F_\chi/z} 
\nabla_{\vec{v}} \Omega$; and
\item $\xi Q\parfrac{}{Q} e^{F^\cl_\chi(p)/z} 
= \frac{1}{z} \hxi(\sigma) e^{F^\cl_\chi(p)/z}$, 
which follows from Lemma \ref{lem:cl_criticalvalue}. 
\end{itemize}
Let us establish the third equation. 
Recall that the grading operator on $\C[z]\{\OO_+\}[\![y]\!]\omega$ 
is induced by $e_1^\star + \cdots +e_m^\star \in (\Q^m)^\star 
= \Hom(\OO,\Q)$, $\deg z =1$, $\deg y_\bk = 1-|\bk|$, 
$\deg \omega =0$. 
The potential function $F_\chi(x;y)$ is not homogeneous because 
of the $\log x_i$ term, but the logarithmic derivative $x_i\parfrac{}{x_i} F_\chi(x;y)$ 
is homogeneous of degree $1$. 
Thus the critical point $p$ is homogeneous in the sense that 
$\cE^{\rm B} w_\bk(p) = |\bk|w_\bk(p)$ for all $\bk\in \bN\cap|\Sigma|$, 
and so the quantum part $F^\qu_\chi(p)$ \eqref{eq:qu_crit}  
of the critical value is homogeneous of degree $1$. 
We can also see that the variables $v_i$ appearing 
in the Gaussian integral \eqref{eq:Gaussian} 
is of degree zero. 
Therefore, if $\Omega$ is of degree $k$, 
$(-2\pi z)^{n/2} e^{F^\qu_\chi(p)/z} \Asym_p(e^{F_\chi/z} \Omega)$ 
is of degree $k$, and $\Asym_p(e^{F_\chi/z}\Omega)$ 
is of degree $k-\frac{n}{2}$. The conclusion follows. 
\end{proof} 

In view of Definition \ref{def:HRP}, Lemma \ref{lem:asymp_diffeq} 
implies the following. 

\begin{proposition}
The higher residue pairing satisfies the following properties: 
\begin{enumerate}
\item $P(c(-z) \Omega_1,\Omega_2) = P(\Omega_1,c(z) \Omega_2) 
= c(z) P(\Omega_1, \Omega_2)$ 
for $c(z) \in R_\T[z][\![\Laa_+]\!][\![y]\!]$; 

\item $\xi Q\parfrac{}{Q} P(\Omega_1,\Omega_2) 
= P(\nabla_{\xi Q\parfrac{}{Q}} \Omega_1, \Omega_2) 
+ P(\Omega_1, \nabla_{\xi Q\parfrac{}{Q}} \Omega_2)$ 
for $\xi \in \LL_\C^\star$; 
\item $\parfrac{}{y_\bk} P(\Omega_1,\Omega_2) 
= P(\parfrac{}{y_\bk} \Omega_1, \Omega_2) 
+ P(\Omega_1, \parfrac{}{y_\bk}\Omega_2)$ 
for $\bk\in S$; 
\item $(z\parfrac{}{z} + \cE^{\rm B} + n)
P(\Omega_1,\Omega_2) = P(\Gr^{\rm B} \Omega_1,\Omega_2) 
+ P(\Omega_1, \Gr^{\rm B} \Omega_2)$. 
\end{enumerate}
\end{proposition}

\subsection{The pairings match}
\label{subsec:pairing_match}

\begin{proposition}    \label{pro:Asymp_is_Loc}
Let $p$ be the critical point of $F_\chi$ that corresponds, 
via Lemma~\ref{lem:critical_points}, to $c \in \Crit(\sigma)$ where 
$\sigma \in \Sigma$ is a maximal cone.  We have:
  \begin{equation}
    \label{eq:Asymp_is_Loc}
    e^{F^{\qu}_\chi(p)/z} \Asym_p\big( e^{F_\chi/z} \Omega \big) = 
    \sum_{v \in \Bx(\sigma)} 
\Loc( \Omega)\big|_{(\sigma,v)} \, c_v \, \Delta_{(\sigma,v)}(z) 
  \end{equation}
  for $\Omega \in \GM(F_\chi)$, 
  where $(\cdots)|_{(\sigma,v)}$ denotes the restriction to the fixed point 
  $z_\sigma \in \frX_v$ on the sector $\frX_v\subset I \frX$, and 
  \[
  \Delta_{(\sigma,v)}(z)  = 
\frac{1}{|\bN(\sigma)|}   \prod_{i \in \sigma} 
\frac{1}{\sqrt{u_i(\sigma)}}\exp \left( 
- \sum_{k=2}^\infty \frac{B_k(v_i)}{k(k-1)} 
\left(\frac{z}{u_i(\sigma)}\right)^{k-1}\right)  
  \]
  with $B_k(h)$ the Bernoulli polynomial defined by 
  $\sum_{k=0}^\infty B_k(h) \frac{t^k}{k!} = t e^{ht}/(e^t-1)$ 
  and $v_i = \Psi_i(v)$. 
\end{proposition}

\begin{remark}
\label{rem:interpretation_Asymp_Loc} 
Equality \eqref{eq:Asymp_is_Loc} here should be interpreted with care.  
Each coefficient of $Q^\lambda (y-y^*)^I$ of $\Loc(\Omega)|_{(\sigma,v)}$ 
is a rational function in $z$ (Remark~\ref{rem:rational}).  
We expand these rational functions as Laurent series at $z=0$ 
and multiply by $\Delta_\sigma(z)$, 
obtaining an element of $\overline{S}_\T(\!(z)\!)[\![\Laa_+]\!][\![y]\!]$, 
and equate with the left-hand side.  
As discussed just before Lemma~\ref{lem:asymp_diffeq}, 
the left-hand side is also well-defined as an element of  
$\overline{S}_\T(\!(z)\!)[\![\Laa_+]\!][\![y]\!]$.
\end{remark}
\begin{proof}
Recall from Proposition~\ref{pro:Loc_diffeq} and Lemma~\ref{lem:asymp_diffeq} 
that $\Loc$ and $e^{F^\qu_\chi(p)/z}\Asym_p$ 
satisfy similar differential equations. 
From this we deduce that $e^{F^\qu_\chi(p)/z}\Asym_p(e^{F_\chi/z}\Omega)$ 
can be written as a linear combination 
of $\Loc(\Omega)|_{(\sigma',v')}$, with coefficients 
independent of $Q$ and $y_\bl$. 
We regard $\Loc$ as a map taking values in 
$H^*_{\CR,\T}(\frX) \otimes_{R_\T} 
S_\T(\!(z)\!)[\![\Laa_+]\!][\![y]\!]$ as discussed in 
Remark \ref{rem:interpretation_Asymp_Loc}. 
Extending the ground ring, we obtain an isomorphism 
\[
\Loc \colon 
\GM(F_\chi) \otimes_{R_\T[z][\![\Laa_+]\!][\![y]\!]} 
S_\T(\!(z)\!)[\![\Laa_+]\!][\![y]\!] 
\cong H^*_{\CR,\T}(\frX) \otimes_{R_\T} S_\T(\!(z)\!)[\![\Laa_+]\!][\![y]\!]. 
\]
We set $C_p(\alpha) := e^{F^\qu_\chi(p)/z}\Asym_p (e^{F_\chi/z}
\Loc^{-1} \alpha)$ 
for $\alpha \in H^*_{\CR,\T}(\frX)$. The differential equations 
in Proposition \ref{pro:Loc_diffeq} and Lemma \ref{lem:asymp_diffeq} 
show that 
\begin{align*} 
\left(\xi Q\parfrac{}{Q} + \frac{1}{z} \hxi(\sigma)\right) 
C_p(\alpha) & = \frac{1}{z} C_p (\hxi  \alpha), \qquad 
\parfrac{}{y_\bl} C_p(\alpha) = 0.  
\end{align*} 
Note that $C_p$ belongs to $\Hom(H^*_{\CR,\T}(\frX),R_\T) 
\otimes_{R_\T} \overline{S}_\T(\!(z)\!)[\![\Laa_+]\!]
[\![y]\!]$. The second equation shows that $C_p$ is independent 
of $y_\bl$. Expanding $C_p = \sum_{\lambda \in \Laa_+} 
\sum_{k\in \Z} 
C_{p;\lambda,k} Q^\lambda z^k$, we find from the first equation that 
\[
(\xi \cdot \lambda) C_{p;\lambda,k} = 
\hxi(\sigma) C_{p;\lambda,k+1} - C_{p;\lambda,k+1} \circ \hxi. 
\]
For a fixed $\lambda$, $C_{p;\lambda,k}$ vanishes for  sufficiently 
negative $k\in \Z$. Therefore, if $\xi \cdot \lambda \neq 0$, 
repeated applications of 
the operation $C \mapsto \hxi(\sigma) C - C \circ \hxi$ on 
$\Hom(H^*_{\CR,\T}(\frX),R_\T)\otimes_{R_\T} \overline{S}_\T$ 
to the coefficients $C_{p;\lambda,k}$ yield zero. 
On the other hand, it follows easily from the Atiyah-Bott localization 
theorem that this operation is a semisimple endomorphism, 
and thus we must have 
$\hxi(\sigma) C_{p;\lambda,k+1} - C_{p;\lambda,k+1}\circ \xi =0$. 
This implies that $C_{p;\lambda,k} =0$ for non-zero $\lambda$
and that $C_p$ is independent of $Q$. 
Therefore we have shown that 
\[
e^{F^\qu_\chi(p)/z} \Asym_p(e^{F_\chi/z} \Omega)
= \sum_{(\sigma',v')} C_{p,(\sigma',v')} 
\Loc(\Omega)\bigr|_{(\sigma',v')}
\] for some 
$C_{p,(\sigma',v')}\in \overline{S}_\T(\!(z)\!)$ independent of $Q$ and $y_\bl$. 

Let us take $\Omega = w_\bk \omega$ and 
evaluate the coefficients $C_{p,(\sigma',v')}$ 
by taking the limit as $Q \to 0$,~$y \to y^*$. 
In that limit, $\Loc(w_\bk \omega)$ becomes
  \[
  \Loc^{(0)}\big(w_\bk \omega\big)\Bigr|_{(\sigma',v')}= 
  \begin{cases}
   \prod_{i \in \sigma'} \prod_{\substack{0 \leq c < \Psi_i(\bk) 
\\ \fract{c} = v_i}} (u_i(\sigma') - cz)  
& \text{if $v=v' \in \Bx(\sigma')$} \\
    0 & \text{otherwise}
  \end{cases}
  \]
  where $v := \bk - \sum_{i=1}^m \floor{\Psi_i(\bk)} b_i$.
  
It remains to compute the limit as $Q \to 0$,~$y \to y^*$ 
of $e^{F^{\qu}_\chi(p)/z} \Asym_p\big( e^{F_\chi/z} w_\bk \omega\big)$.  
The limit of $e^{F^{\qu}_\chi(p)/z}$ is $1$.  
We have seen that $\Asym_p\big( e^{F_\chi/z} w_\bk \omega\big)$ 
has a well-defined limit, because 
$\Asym_p\big( e^{F_\chi/z} w_\bk \omega\big)$ 
lies in $\overline{S}_\T(\!(z)\!)[\![\Laa_+]\!][\![y]\!]$.  
Also 
\[
\Asym_p\big( e^{F_\chi/z} w_\bk \omega\big) = \Asym_p\big( e^{(F_\chi - F^{\cl}_\chi(p))/z} w_\bk \omega\big),
\]
and we saw in the discussion before Lemma~\ref{lem:limit_on_component} 
that $F_\chi(x;y)- F_\chi^\cl(p)$ has a well-defined limit as $y$ 
approaches $y^*$ and $(x,Q)$ approaches the component $A_{\sigma,\theta}$ 
of the central fiber on which $p$ lies. 
When $v \notin \Bx(\sigma)$, the restriction of $w_\bk$ to $A_{\sigma,\theta}$ 
is zero, and thus $\Asym_p(e^{F_\chi/z} w_\bk \omega)|_{Q=0,y=y^*}=0$. 
We can therefore assume that $v\in \Bx(\sigma)$. 
On the component $A_{\sigma,\theta}$ we have
  \[
  \int_{\Gamma(p)} e^{(F_\chi - F^{\cl}_\chi(p))/z} w_\bk \omega = 
  \int_{\Gamma(p)} \exp\left( \sum_{i \in \sigma} \textstyle \frac{w_i - u_i(\sigma) - u_i(\sigma) \log \big({\textstyle \frac{w_i}{u_i(\sigma)}}\big)}{z} \right) w_\bk \frac{\det (b^{\alpha,\beta})}{|\bN_{\rm tor}|} \bigwedge _{i \in \sigma} \frac{dw_i}{w_i}.
  \]
  Here $\Gamma(p)$ is an appropriate cycle in $A_{\sigma,\theta}$ through $p$, the precise choice of which is irrelevant as we calculate the formal asymptotic expansion of this integral at $p$, and $(b^{\alpha,\beta})$ are the entries of the matrix inverse to $(b_{\alpha,\beta})$.
  Proceeding as in \S\ref{subsec:HRP}, we set $w_i = u_i(\sigma) e^{T_i}$, so that $w_\bk = c_\bk \exp \big( {\sum_{i \in \sigma} \Psi_i(k) T_i}\big)$.  Thus the integral becomes
  \[
  \frac{c_\bk}{|\bN(\sigma)|} \int_{\R^n} \exp\left( \sum_{i \in \sigma} {\textstyle \frac{u_i(\sigma)}{z}} \big(e^{T_i}-1-T_i\big)\right) \exp\left(\sum_{i \in \sigma} \Psi_i(k) T_i \right)\bigwedge _{i \in \sigma} dT_i
  \]
  where we used $\det(b_{\alpha,\beta}) = |\bN(\sigma)|/|\bN_{\rm tor}|$.  This is essentially a $\Gamma$-function.  To see this, assume that $z<0$ and $u_i(\sigma)>0$, and make the change of variables $T_i \mapsto T_i + \log\big(\frac{-z}{u_i(\sigma)}\big)$.  The integral becomes
  \[
  \frac{c_\bk}{|\bN(\sigma)|} \prod_{i \in \sigma} e^{\frac{u_i(\sigma)}{-z}}  \left( {\textstyle \frac{u_i(\sigma)}{-z}} \right)^{\frac{u_i(\sigma)}{z}  - \Psi_i(\bk) } 
  \Gamma\Big( \textstyle \Psi_i(\bk) - \frac{u_i(\sigma)}{z} \Big).
  \]
  Using the functional equation for the $\Gamma$ function (which is also satisfied by its asymptotic expansion) and the fact that $c_\bk = c_v \prod_{i \in \sigma} u_i(\sigma)^{\floor{\Psi_i(\bk)}}$, this is
  \[
  \frac{c_v}{|\bN(\sigma)|} \prod_{i \in \sigma} e^{\frac{u_i(\sigma)}{-z}} \left( {\textstyle \frac{u_i(\sigma)}{-z}} \right)^{\frac{u_i(\sigma)}{z}  - v_i } \left(\prod_{\substack{0 \leq c < \Psi_i(\bk)\\ \fract{c} = v_i}} u_i(\sigma) - cz \right)\Gamma\big(v_i - \textstyle \frac{u_i(\sigma)}{z}\big)
  \]
  Replacing the $\Gamma$ function by its asymptotic expansion, using~\cite[\href{http://dlmf.nist.gov/5.11.E8}{5.11.8}]{NIST:DLMF} 
  \[
  \log \Gamma(z+h) \sim
  (z+h-{\textstyle\frac{1}{2}}) \log z- z + {\textstyle \frac{1}{2} \log(2 \pi)} + \sum_{k=2}^\infty \frac{(-1)^k B_k(h)}{k(k-1) z^{k-1}}
  \]
  where $B_k(\cdot)$ are the Bernoulli polynomials, yields:
  \begin{multline*}
    ({-2 \pi z})^{n/2} \frac{c_v}{|\bN(\sigma)|} \prod_{i \in \sigma} \left( \prod_{\substack{0 \leq c < \Psi_i(\bk) \\ \fract{c} = v_i}} u_i(\sigma) - cz \right) \\
   \times \prod_{i \in \sigma}
    \frac{1}{\sqrt{u_i(\sigma)}}
    \exp \left( - \sum_{k=2}^\infty \frac{B_k(v_i)}{k(k-1)} \left(\frac{z}{u_i(\sigma)}\right)^{k-1}\right)  
  \end{multline*}
  Thus the limit as $Q \to 0$, $y \to y^*$ of $\Asym_p\big( e^{F_\chi/z} w_\bk \omega\big)$ is
  \[
  \frac{c_v}{|\bN(\sigma)|} \prod_{i \in \sigma} \left( \prod_{\substack{0 \leq c < \Psi_i(\bk) \\ \fract{c} = v_i}} u_i(\sigma) - cz \right) 
    \prod_{i \in \sigma}
    \frac{1}{\sqrt{u_i(\sigma)}}
    \exp \left( - \sum_{k=2}^\infty \frac{B_k(v_i)}{k(k-1)} \left(\frac{z}{u_i(\sigma)}\right)^{k-1}\right)  
    \]
 if $v\in \Bx(\sigma)$ and zero otherwise. 
The result follows.
\end{proof}

\begin{theorem} 
\label{thm:pairings_match}
{\rm (1)} The mirror isomorphism $\Theta$ from Theorem~\ref{thm:mirror_isom} intertwines the the higher residue pairing on $\GM(F_\chi)$ with the Poincar\'e pairing on $H^*_{\CR,\T}(\frX) \otimes_{R_\T} R_\T[z][\![\Laa_+]\!][\![y]\!]$:
    \[
    P(\Omega_1,\Omega_2) = \Big(\Theta(\Omega_1)\big|_{z \mapsto -z}, \Theta(\Omega_2) \Big). 
    \]
{\rm (2)} The localization map intertwines the higher residue pairing on $\GM(F_\chi)$ with the Poincar\'e pairing on the Givental space:
  \[
  P(\Omega_1,\Omega_2) = \Big(\Loc(\Omega_1)\big|_{z \mapsto -z}, \Loc(\Omega_2) \Big). 
  \]
\end{theorem}

\begin{proof}
First we prove (2). Using Definition~\ref{def:HRP} 
and Proposition~\ref{pro:Asymp_is_Loc}, 
we find that $P(\Omega_1,\Omega_2)$ equals 
\begin{equation}
\label{eq:sum_over_fixed_points} 
\sum_{\sigma \in \Sigma(n)} 
\sum_{c\in \Crit(\sigma)} 
\sum_{v,w\in \Bx(\sigma)} 
c_{v}c_{w} \Delta_{(\sigma,v)}(-z) 
\Delta_{(\sigma,w)}(z) 
\overline{\Loc(\Omega_1)}\bigr|_{(\sigma,v)} 
\Loc(\Omega_2)\bigr|_{(\sigma,w)} 
\end{equation} 
where $\overline{\Loc(\Omega_1)} = \Loc(\Omega_1)|_{z\to -z}$. 
Recall that $\Crit(\sigma)$ is a torsor over $\hbN(\sigma)$. 
Choose a base point $c^*\in \Crit(\sigma)$ and write 
a general element $c\in\Crit(\sigma)$ as 
$c = \theta \cdot c^*$ for $\theta \in \hbN(\sigma)$. 
Orthogonality of characters implies that
\[
\sum_{c\in \Crit(\sigma)} c_{v} c_{w} 
= \sum_{\theta \in \hbN(\sigma)} c^*_v c^*_w \theta(v) \theta(w) 
= c^*_{v+w}  |\bN(\sigma)| \delta_{v,-w} 
\]
where $\delta_{v,-w}$ equals $1$ if $v \equiv -w$ in $\bN(\sigma)$ 
and zero otherwise. 
Let $\sigma(v)\subset \sigma$ be the minimal cone of $\Sigma$ 
containing $v$. We have 
\[
c^*_{v+w} = \prod_{j\in \sigma(v)} u_j(\sigma) 
\]
whenever $v \equiv - w$ in $\bN(\sigma)$. 
Using the fact~\cite[\href{http://dlmf.nist.gov/24.4.E3}{24.4.3}]{NIST:DLMF} 
that $B_k(1-h) = (-1)^k B_k(h)$, 
we find that the quantity \eqref{eq:sum_over_fixed_points} equals 
\[
\sum_{\sigma\in \Sigma(n)} 
\sum_{v \in \Bx(\sigma)} 
\frac{1}{|\bN(\sigma)|} 
\frac{1}{\prod_{j\in \sigma\setminus \sigma(v)}u_j(\sigma)}
\overline{\Loc(\Omega_1)}
\bigr|_{(\sigma,v)} \Loc(\Omega_2)\bigr|_{(\sigma,-v)}.  
\]
The Atiyah--Bott localization formula now yields part (2) of the Theorem. 
In view of Remark~\ref{rem:Theta_Loc}, and the fact that 
$M(\tau,z)$ in Remark~\ref{rem:Theta_Loc} is pairing-preserving 
(see Proposition~\ref{pro:fundsol}), we have that (2) implies (1).  
\end{proof}

\section{Convergence}  
\label{sec:convergence}
In this section, we discuss convergence of the mirror isomorphism 
and the mirror map from Theorem \ref{thm:mirror_isom}. Recall that $Q$ is a Novikov variable, $y=\{y_\bk : \bk\in S\}$ 
is a deformation parameter, and $\chi$ is an equivariant parameter. 
The main result in this section says that the mirror map is analytic 
in all the parameters $(Q,y,\chi)$ and that the mirror isomorphism 
is a formal power series in $z$ with coefficients in analytic functions in $(Q,y,\chi)$. 
This implies the convergence of the big equivariant 
quantum product, and thus generalizes the convergence result \cite{Iritani:coLef} 
for compact toric varieties to arbitrary semi-projective 
toric Deligne--Mumford stacks in the \emph{big} and \emph{equivariant} 
setting. 

\subsection{Result} 
In order to state the convergence result, we introduce a co-ordinate 
system. We choose a $\Z$-basis of $\Laa$ such that $\Laa_+$ 
is contained in the cone spanned by the basis. 
This basis defines a co-ordinatization $Q=(Q_1,\dots,Q_r)$ 
of the variable $Q$ (where $r = \rank \Laa$). 
Note that any power series in $Q$ whose exponents are supported in 
$\Laa_+$ can be expressed as a nonnegative power series in $Q_1,\dots,Q_r$. 
Choosing a basis of $\bM$, we have co-ordinates 
$\chi=(\chi_1,\dots,\chi_n)$ on $\Lie \T =\bN_\C$ as before 
(so that $R_\T = \C[\chi_1,\dots,\chi_n]$). 
We write 
\[
q := (q_1,\dots,q_s) := (Q_1,\dots,Q_r, \log y_1,\dots,\log y_m, \{y_\bk : \bk\in G\}) 
\]
with $s = r + m + |G|$. 
Note that $q=0$ corresponds to $Q=0$ and $y=y^*$ 
 -- see \eqref{eq:specialization_y} for $y^*$. 
We also choose a homogeneous basis $\{T_i\}_{i=1}^N$ 
of $H_{\CR,\T}(\frX)$ over $R_\T$
and homogeneous algebraic differential forms 
$\Omega_i\in \bigoplus_{\bk \in \bN \cap |\Sigma|} 
\C[z]w_\bk \omega$ on the Landau--Ginzburg model 
(see \S\ref{subsec:GM}) 
such that 
\[
\Loc^{(0)}(\Omega_i)= T_i, \qquad 1\le i\le N.  
\]
Here $\Loc^{(0)}$ is the restriction of the localization map 
(see \S\ref{subsec:sol_free}) 
to the origin $Q=0,y=y^*$; such $\Omega_i$ exist by Proposition \ref{pro:Loc0}. 
By Theorem \ref{thm:freeness} (and its proof), 
$\{\Omega_i\}_{i=1}^N$ form a basis of the equivariant Gauss--Manin system 
$\GM(F_\chi)$ over $R_\T[z][\![\Laa_+]\!][\![y]\!]$. 
Let $\N = \{0,1,2,\dots\}$ denote the set of nonnegative integers 
and write $q^d = q_1^{d_1} \cdots q_s^{d_s}$ for 
$d =(d_1,\dots,d_s) \in \N^s$. 
\begin{definition} 
\label{def:Oz} 
We define $\cO^z$ to be the space of (possibly divergent) 
formal power series in $q=(q_1,\dots,q_s)$ and $z$ of the form 
\[
\sum_{d\in \N^s} \sum_{k=0}^\infty a_{d,k}(\chi) q^d z^k
\] 
where the coefficients $a_{d,k}(\chi)$ are holomorphic functions of 
$\chi =(\chi_1,\dots,\chi_n)$ 
defined in a uniform neighbourhood of $\chi=0$ 
and satisfying the following estimate: 
there exist positive constants $C_1$,~$C_2$,~$\epsilon>0$ such that we have 
\begin{equation} 
\label{eq:basic_estimate} 
|a_{d,k}(\chi)|\le C_1 C_2^{|d|+k} |d|^k  
\quad \text{for}\quad |\chi|\le \epsilon,  
\end{equation} 
where $|\chi| = (\sum_{i=1}^n |\chi_i|^2)^{1/2}$ 
and $|d| =\sum_{i=1}^s |d_i|$. 
We adopt the convention that $|d|^k =1$ if $|d|=k=0$. 
The constants $C_1$,~$C_2$,~$\epsilon$ here are allowed to depend on the element 
of $\cO^z$. 
Note that the condition \eqref{eq:basic_estimate} implies 
$a_{0,k}(\chi) = 0$ if $k>0$. 
Note also that the subseries $\sum_{d\in \N^s} a_{d,k}(\chi) q^d$ 
coverges on $|q_a|<1/C_2$, $|\chi|\le \epsilon$ 
for each $k\in \N$. 
\end{definition} 
It is easy to check that $\cO^z$ is a ring; moreover 
it is a \emph{local ring} 
-- see Lemma \ref{lem:Oz_local_ring}. 

\begin{theorem}
\label{thm:convergence}
Expand the mirror map $\tau(y)$ and 
the mirror isomorphism $\Theta$ in Theorem \ref{thm:mirror_isom} 
with respect to the bases $\{T_i\}$, $\{\Omega_i\}$ as follows: 
\[
\tau(y) = \sum_{i=1}^N \tau^i(Q,y,\chi) T_i, \qquad 
\Theta(\Omega_i) = \sum_{j=1}^N \Theta_i^j(Q,y,\chi,z) T_j.  
\]
Then:
\begin{itemize} 
\item[(1)] the coefficients $\tau^i(Q,y,\chi)\in R_\T[\![\Laa_+]\!][\![y]\!]$ 
are convergent and analytic in a neighbourhood 
of $Q=0$, $y=y^*$ and $\chi =0$; 
\item[(2)] the coefficients $\Theta^i_j(Q,y,\chi,z)\in R_\T[z][\![\Laa_+]\!][\![y]\!]$ 
lie in the ring $\cO^z$; in particular they are formal power series in $z$ with 
coefficients in analytic functions of $(Q,y,\chi)$ defined in a uniform neighbourhood of 
$Q=0$, $y=y^*$ and $\chi=0$. 
\end{itemize} 
\end{theorem} 

\begin{corollary} 
\label{cor:convergence} 
The structure constants $(\alpha\star\beta,\gamma)$ 
of the big and equivariant quantum product \eqref{eq:qprod} 
of a semi-projective toric Deligne--Mumford stack are convergent 
power series in $\tau$, $Q$ and $\chi$. 
\end{corollary} 

\begin{remark} 
To motivate the definition of $\cO^z$, we give the following example. 
Suppose for simplicity that $q$ is one variable. 
If $\sum_{k=0}^\infty a_k t^k$ is a convergent power series, 
the differential operator $\sum_{k=0}^\infty a_k (zq \parfrac{}{q})^k$ 
applied to $\sum_{d=0}^\infty q^d=1/(1-q)$ yields a power series 
$\sum_{d=0}^\infty \sum_{k=0}^\infty a_k d^k q^d z^k$ 
that belongs to $\cO^z$. 
\end{remark} 

\begin{remark} 
The coefficients $\Theta^i_j$ of the mirror isomorphism are 
in general divergent power series in $z$. 
If we restrict ourselves to the ``extended weak Fano'' situation 
\cite[\S 3.1.4]{Iritani:integral}, 
that is, if $\frX$ is weak Fano and the extension $G$ is contained in 
$\Bx^{\le 1} := \Bx \cap \{v \in \bN\cap |\Sigma|: |v| \le 1\}$, 
then the mirror isomorphism $\Theta$ becomes fully 
convergent \cite[Proposition 4.8]{Iritani:integral}. 
On the other hand, if $\frX$ is not weak Fano or $G$ is not contained 
in $\Bx^{\le 1}$, the $I$-function and the mirror isomorphism 
are typically divergent (see \cite[Proposition 5.13]{Iritani:coLef}). 
The convergence issue is also related to ``good asymptotics'' 
of the $I$-function, see \cite[\S 2.6]{CCIT:applications}. 
\end{remark} 

\begin{lemma}
\label{lem:Oz_local_ring} 
The space $\cO^z$ is a local ring. 
\end{lemma} 
\begin{proof} 
We claim that if $x= \sum_{d\in \N^s} \sum_{k\ge 0} a_{d,k}(\chi) q^d z^k 
\in \cO^z$ satisfies $a_{0,0}(0) \neq 0$, then it is invertible in $\cO^z$. 
Without loss of generality we may assume that $a_{0,0}(\chi)=1$. 
There exist constants $C_1$,~$C_2$,~$\epsilon>0$ such that 
$|a_{d,k}(\chi)|\le C_1 C_2^{|d|+k} |d|^k$ for $|\chi|\le \epsilon$. 
We have 
\[
x^{-1} = 1 + \sum_{l=1}^\infty (-1)^l 
\sum_{\substack{d(1)\neq 0,\dots,d(l)\neq 0 \\ k_1,\dots,k(l)\ge 0}}  
a_{d(1),k(1)} \cdots a_{d(l),k(l)} q^{d(1)+\cdots + d(l)} 
z^{k(1)+\cdots + k(l)} 
\]
The coefficient of $q^dz^k$ can be estimated as: 
\begin{align*} 
\left| \sum_{l=1}^{|d|} 
\sideset{}{^{(l)}}{\sum} a_{d(1),k(1)} \cdots a_{d(l),k(l)}\right|
& \le \sum_{l=1}^{|d|} \sideset{}{^{(l)}}{\sum} 
|d(1)|^{k(1)} \cdots |d(l)|^{k(l)} C_1 C_2^{|d|+k} \\ 
& \le |d|^k C_1 C_2^{|d|+k} 
\sum_{l=1}^{|d|} \sum_{\substack{
d(1)+\cdots + d(l)=d\\ d(1)\neq 0,\dots,d(l)\neq 0}} 1 
\end{align*}
where $\sum^{(l)}$ means the sum over all $d(1),\dots,d(l) \in \N^s$ 
and $k(1),\dots,k(l) \in \N$ such that $d(1)\neq 0,\dots,d(l)\neq 0$,  
$d(1)+\cdots + d(l) =d$ and $k(1) + \cdots + k(l) = k$. 
It is easy to check that the sum in the second line is of exponential 
growth in $|d|$. 
\end{proof} 
The rest of this section (\S\ref{sec:convergence}) is devoted to the proof of 
Theorem \ref{thm:convergence} and Corollary \ref{cor:convergence}. 
The proof of Theorem \ref{thm:convergence} 
consists of two steps: first we give an estimate for the 
connection matrices of the Gauss--Manin connection in the 
basis $\{\Omega_i\}$, and then we use a theorem on gauge fixing 
to show that a gauge transformation that transforms the Gauss--Manin 
connection into the quantum connection is defined over $\cO^z$. 

\subsection{Estimates for the Gauss--Manin connection} 
Let $A_a=({{A_a}^j}_i(q,\chi,z))$ denote the connection matrix of the 
equivariant Gauss--Manin system 
with respect to the basis $\{\Omega_i\}_{i=1}^N$ 
defined by 
\[
z \nabla_{q_a \parfrac{}{q_a}} \Omega_i = 
\sum_{j=1}^N {{A_a}^j}_i(q,\chi,z) \Omega_j 
\]
where ${{A_a}^j}_i(q,\chi,z) \in R_\T[z][\![\Laa_+]\!][\![y]\!] 
\subset \C[z,\chi][\![q_1,\dots,q_s]\!]$. 
In this section we prove the following proposition.
\begin{proposition} 
\label{pro:estimate_GM_conn} 
There exist open neighbourhoods $U \subset \C^s$, 
$V \subset \C^n$ of the origin such that the 
following hold. 
Entries of the connection matrix $A_a$ can be 
expanded in power series 
$\sum_{k=0}^\infty a_k(q,\chi) z^k$ that satisfy 
\begin{itemize} 
\item[(1)] for every $k$, $a_k(q,\chi)$ is convergent and holomorphic 
for $(q,\chi) \in U\times V$;  
\item[(2)] there exist contants $C_1,C_2>0$ such that 
\begin{equation} 
\label{eq:Gevrey}
|a_k(q,\chi)|\le C_1 C_2^k k! \qquad \text{for all $(q,\chi) \in U\times V$.}
\end{equation} 
\end{itemize} 
\end{proposition} 

\begin{remark} 
\label{rem:Gevrey} 
Power series $\sum_{k=0}^\infty a_k z^k$ satisfying the estimate 
\eqref{eq:Gevrey} 
are called \emph{Gevrey series of order $1$}. 
The same estimate also appears as a convergence condition 
for microdifferential operators (see \cite[\S 2]{Kashiwara:microlocal}). 
We write 
\[
\C\{\!\{z\}\!\} = \left\{\sum_{k=0}^\infty a_k z^k : \text{$a_k\in \C$,  
$\exists C_1$,~$C_2>0$ s.t.~$|a_k| \le C_1 C_2^k k!$} \right\}. 
\]
It is well-known that $\C\{\!\{z\}\!\}$ is a local ring. 
\end{remark} 

We start by noting that the Landau--Ginzburg potential $F_\chi$ 
is a globally defined (multi-valued) analytic function in the arguments 
$x$,~$y$,~$Q$,~$\chi$. Therefore the (finitely many) critical points of $F_\chi$ over 
$\overline{S}_\T[\![\Laa_+]\!][\![y]\!]$ described 
in Lemma \ref{lem:critical_points} depend 
analytically on the parameters $(Q,y,\chi)$. 
Recall from Lemma \ref{lem:critical_points}(2--3) that 
the co-ordinates $w_\bk(p)$ of a critical point at $Q=0$, $y=y^*$ are 
analytic functions of $\chi$ defined on $(\C^n \setminus \cD)\sptilde$, 
where $\sptilde$ means the universal cover and $\cD 
= \bigcup_{i=1}^m \bigcup_{\sigma \in \Sigma(n)} \{ u_i(\sigma) = 0\}$ 
(see Notation \ref{nota:fixed_points}). 
Thus co-ordinates of each critical point are power series in $q=(Q,y)$ 
with coefficients in analytic functions on $(\C^n\setminus \cD)\sptilde$. 
We fix an arbitrary compact subset 
$K\subset (\C^n\setminus \cD)\sptilde$, and then choose 
 a sufficiently small neighbourhood $U\subset \C^s$ of 
the origin such that co-ordinates of these critical points are 
convergent on $(\chi,q) \in K\times U$. 

Let $p$ be an analytic branch of the critical scheme of $F_\chi$ 
defined over $K\times U$ as above. 
Recall from \S\ref{subsec:HRP} 
that $\Asym_p(e^{F_\chi/z} \Omega_i) \in \overline{S}_\T[z][\![\Laa_+]\!]
[\![y]\!]$ is defined as a formal asymptotic 
expansion of an oscillatory integral at the critical branch $p$:  
\[
\int_{\Gamma(p)} e^{F_\chi/z} \Omega_i \sim e^{F_\chi(p)/z} (-2\pi z)^{n/2} 
\Asym_p(e^{F_\chi/z} \Omega_i). 
\]
By the definition of the formal asymptotic expansion in \S\ref{subsec:HRP} 
and the analyticity of $F_\chi$, it follows that $\Asym_p(e^{F_\chi/z}\Omega_i)$ 
is a formal power series in $z$ with coefficients in analytic functions in 
$(\chi,q)\in K\times U$ (where we also used the fact that $\Omega_i$ is 
an algebraic differential form). 
Moreover, these coefficients satisfy the following: 
\begin{proposition} 
\label{pro:estimate_asymp}
After shrinking $U$ if necessary, we can find constants $C_1$,~$C_2>0$ 
such that the power series expansion $\sum_{k=0}^\infty a_k(q,\chi) z^k$ 
of $\Asym_p(e^{F_\chi/z}\Omega_i)$ satisfies the estimate 
$|a_k(q,\chi)| \le C_1 C_2^k k!$ for all $k\ge 0$ and $(q,\chi) \in U\times K$. 
\end{proposition} 
This proposition follows immediately from:
\begin{lemma} 
Let $f$, $g$ be holomorphic functions on $\C^n$ 
and let $p\in \C^n$ be a non-degenerate critical point of $f$. 
(It suffices that $f$,~$g$ are defined near $p$.) 
The formal asymptotic 
expansion at $p$, as defined in \S\ref{subsec:HRP}:
\[
\int e^{f(t)/z} g(t) dt^1 \cdots dt^n 
\sim e^{f(p)/z} (-2\pi z)^{n/2} (a_0 + a_1 z + a_2 z^2 + \cdots ) 
\]
satisfies the estimate $|a_k|\le C_1 C_2^k k!$ 
for some constants $C_1$,~$C_2>0$. 
If $B_1$,~$B_2>0$ are constants such that the Taylor expansions 
$f(p+t) = \sum_{I} c_I t^I$, $g(p+t) = \sum_I d_I t^I$ 
at $p$ satisfy $\max(|c_I|,|d_I|)\le B_1 B_2^{|I|}$, then the constants $C_1,C_2$ 
here depend only on $B_1$ and $B_2$. 
Here $I\in \N^n$ denotes a multi-index and we write 
$t^I = (t^1)^{i_1} \cdots (t^n)^{i_n}$ and $|I| =i_1+\cdots +i_n$ 
for $I = (i_1,\dots,i_n)$. 
\end{lemma} 
\begin{proof} 
By the holomorphic Morse lemma, we can find local co-ordinates 
$v=(v^1,\dots,v^n)$ centered at $p$ such that $f(t) = f(p) 
+\frac{1}{2}((v^1)^2 + \cdots + (v^n)^2)$. 
Changing co-ordinates, we get 
\[
\int e^{f(t)/z} g(t) dt^1 \cdots dt^n 
= e^{f(p)/z}\int e^{\sum_{i=1}^n (v^i)^2/(2z)} h(v) 
\parfrac{(t^i)}{(v^j)} 
dv^1\cdots dv^n 
\]
where $h(v)$ is a holomorphic function near $v=0$ 
and $\partial(t^i)/\partial(v^j)$ denotes the Jacobian of the co-ordinate change. 
Let $\sum_I e_I v^I$ denote the Taylor expansion 
of $h(v) (\partial (t^i)/\partial (v^j))$; then we have an estimate 
$|e_I| \le C_1 C_2^{|I|}$ with constants $C_1$,~$C_2>0$ depending 
only on $B_1$ and $B_2$. 
The asymptotic expansion of the above integral gives 
\begin{equation} 
\label{eq:after_expansion}
\sum_{\substack{
I=(i_1,\dots,i_n)\in (2\N)^n}}  
e_I  (-z)^{|I|/2} \prod_{a=1}^n (i_a-1)!!  
\end{equation} 
multiplied by $e^{f(p)/z} (-2\pi z)^{n/2}$, 
where $(2k-1)!! = (2k-1) (2k-3) \cdots 3\cdot 1$ 
(we set $(-1)!! =1$). 
The coefficient in front of $z^k$ in \eqref{eq:after_expansion} 
has the estimate 
\begin{align*} 
\left| 
\sum_{I\in \N^n: |I|=k} e_{2I} (-1)^{|I|} \prod_{a=1}^n (2i_a-1)!!
\right| 
& \le C_1 C_2^{2k} 2^k \sum_{I\in \N^n: |I| = k} i_a! \\ 
& \le C_1 C_2^{2k} 2^k \binom{k+n-1}{k} 
k! 
\end{align*} 
which gives the desired estimate since $\binom{a}{b}\le 2^a$. 
\end{proof} 

\begin{proof}[Proof of Proposition \ref{pro:estimate_GM_conn}] 
We use the fact from Lemma \ref{lem:asymp_diffeq} 
that $\Asym_p$ is a solution to the Gauss--Manin system.  
Let $\bR$ be the square matrix with entries 
$\bR_{p,i} = \Asym_p(e^{F_\chi/z} \Omega_i)$ 
where $p$ ranges over all critical points in Lemma \ref{lem:critical_points} 
and $1\le i\le N$. 
Let $\bU$ be the diagonal matrix with entries $F_\chi^\qu(p)$ 
with $p$ ranging over the same set. 
The differential equation in Lemma \ref{lem:asymp_diffeq} shows that 
\[
\left(z q_a\parfrac{}{q_a} +\xi_a\right) 
\left( e^{\bU/z} \bR\right) = e^{\bU/z} \bR A_a 
\]
for some $\xi_a \in H^2_\T(\pt,\C)=\bigoplus_{i=1}^n \C \chi_i$.  
This is equivalent to 
\[
A_a = \xi_a I + \bR^{-1} z q_a\parfrac{\bR}{q_a} 
+ \bR^{-1} q_a \parfrac{\bU}{q_a} \bR  
\]
where $I$ denotes the identity matrix. 
Proposition \ref{pro:Asymp_is_Loc} together with the fact that 
$\Loc^{(0)}(\Omega_i) =T_i$ shows that $\bR|_{q=0}$ 
is invertible for $\chi \in (\C^n\setminus \cD)\sptilde$. 
Recall from Proposition \ref{pro:estimate_asymp} 
that $\bR_{p,i}$ is a Gevrey series of order 1 as a power series in $z$
(see Remark \ref{rem:Gevrey}). 
Since $\C\{\!\{z\}\!\}$ is a local ring, we conclude that 
the entries of $\bR^{-1}$ satisfy estimates similar to 
Proposition \ref{pro:estimate_asymp}, after shrinking $U$ if necessary. 
Also, the entries of $z q_a \parfrac{\bR}{q_a}$ satisfy estimates similar 
to Proposition \ref{pro:estimate_asymp} after shrinking $U$ if necessary; 
this follows from the Cauchy integral formula for derivatives. 
The matrix $\bU$ is convergent and analytic on $(q,\chi) \in U\times K$. 
Therefore each entry of $A_a$ can be expanded in a series 
$\sum_{k=0}^\infty a_k(q,\chi) z^k$ that satisfies 
the estimate $|a_k(q,\chi)|\le C_1 C_2^k k!$ for 
$(q,\chi) \in U \times K$. On the other hand, 
we know that $a_k(q,\chi)$ here lies in $\C[\chi][\![q_1,\dots,q_s]\!]$, 
i.e.~coefficients of the Taylor expansion of $a_k(q,\chi)$ with respect to $q$ 
are polynomials in $\chi$. 
Expand $a_k(q,\chi) = \sum_{d\in \N^s} a_{k,d}(\chi) q^d$. 
Then we have the estimate 
\[
|a_{k,d}(\chi)| \le C_1C_2^{k+|d|} k! \qquad \forall \chi \in K 
\]
for possibly bigger constants $C_1$,~$C_2>0$. 
Recall that $K$ can be taken to be an arbitrary compact subset 
in $(\C^n\setminus \cD)\sptilde$ (the constants 
$C_1$,~$C_2$ depend on the choice of $K$); we can choose $K$ so that 
the holomorphically convex hull (or polynomially convex hull) 
of the image of $K$ in $\C^n \setminus \cD$ 
contains the origin in its interior. Then the above estimate holds 
in a neighbourhood of the origin $\chi=0$. 
This shows that each entry of $A_a$ satisfies the estimate in 
Proposition~\ref{pro:estimate_GM_conn}. 
\end{proof} 

\begin{remark} 
The consideration of critical points in this section together with 
the study of the Jacobi ring in the non-equivariant limit 
in \cite[Proposition 3.10(ii)]{Iritani:integral} shows that 
\emph{the non-equivariant quantum cohomology of a compact toric stack is generically 
semisimple}. See also \cite[\S 5.4]{Iritani:coLef}, 
\cite[Corollary 4.9]{Iritani:integral}. 
\end{remark}

\subsection{Gauge fixing} 
We begin with the following lemma. 

\begin{lemma} 
\label{lem:GM_conn_Oz} 
Matrix entries of $A_a$ belong to $\cO^z$. 
\end{lemma} 
\begin{proof} 
Since the basis $\{\Omega_i\}$ is homogeneous 
with respect to $\Gr^{\rm B}$ -- see \eqref{eq:Gr_B} --
it follows from Proposition \ref{pro:GM_conn_flat} that 
the connection matrices $A_a$ 
have homogeneous entries with respect to the grading 
$\deg Q^d = c_1(\frX) \cdot d$, $\deg y_\bk = 1-|\bk|$, 
$\deg \chi_i  = \deg z = 1$. 
Since there are finitely many matrix entries of $A_a$'s, 
we have a uniform constant $C>0$ such that every 
entry $\sum_{k=0}^\infty a_k(q,\chi) z^k$ 
of $A_1,\dots,A_s$ satisfies 
$\deg a_k(q,\chi) + k\le C$ for all $k \in \N$. 
Expanding $a_k(q,\chi) = \sum_{d\in \N^s} a_{d,k}(\chi) q^d$, 
we obtain $\deg q^d + k \le C$ since $a_{d,k}(\chi) \in \C[\chi]$ 
has non-negative degree. 
This implies $k\le C+ C' |d|$ for a uniform constant $C'>0$. 
On the other hand, the estimate in Proposition \ref{pro:estimate_GM_conn} 
gives $|a_{d,k}(\chi)| \le C_1 C_2^{k+|d|} k!$ for $\chi \in V$. 
Combining the two inequalities we obtain, whenever $d\neq 0$, 
\begin{equation} 
\label{eq:a_dk_estimate}
|a_{d,k}(\chi)| \le C_1 C_2^{k+|d|} (C+C' |d|)^k 
\le C_1 C_3^{k+|d|} |d|^k \qquad 
\text{with $\chi \in V$} 
\end{equation} 
for $C_3 = 2 \max(C,C',1) C_2$. 
We claim that $a_{0,k}(\chi) =0$ for $k>0$. 
Since $\Loc^{(0)}(\Omega_i) = T_i$, Proposition \ref{pro:Loc_diffeq} 
implies that $A_a|_{q=0}$ is conjugate to the multiplication by 
$\hxi_a$ on $H_{\CR,\T}^*(\frX)$ for some $\hxi_a \in H^2_\T(\frX)$ 
($\hxi_a=0$ if $q_a$ corresponds to $\log y_1,\dots,
\log y_m$ or to $y_\bk$ with $\bk\in G$). 
Therefore $A_a|_{q=0}$ is independent of $z$, and the claim follows. 
The claim implies that the estimate \eqref{eq:a_dk_estimate} 
also holds for $d=0$. 
The lemma is proved. 
\end{proof} 

We now give a result on gauge fixing, which says that a logarithmic flat connection 
defined over $\cO^z$ with nilpotent residue at $q=\chi=0$ 
can be made $z$-independent by a unique gauge transformation defined over $\cO^z$. 
This result is a refinement of 
\cite[Proposition 4.8]{Iritani:coLef} which proved a similar result
in the absence of the parameter $\chi$. 
The proof is almost parallel to \cite[Proposition 4.8]{Iritani:coLef}, 
but is different from it in a subtle way. We repeat the argument  
for the convenience of the reader. 

\begin{theorem}[gauge fixing] 
\label{thm:gauge_fixing}
Let $\nabla$ be a logarithmic flat connection of the form 
\[
\nabla = d + \frac{1}{z} \sum_{i=1}^s A_i \frac{dq_i}{q_i} 
\]
where $A_i=A_i(q,\chi,z)\in \Mat_N(\cO^z)$ is a square matrix with entries in $\cO^z$ 
(see Definition \ref{def:Oz}). Note from the definition of $\cO^z$ 
that $A_i(0,\chi,z)$ is independent of $z$. 
Suppose that the residue matrices $A_i(0,0,z)$ are nilpotent. 
Then there exists a unique gauge transformation 
$G=G(q,\chi,z)\in \GL_N(\cO^z)$ with entries in $\cO^z$ such that 
$G(0,\chi,z) = I$ and $\hA_i := G^{-1} z q_i \parfrac{G}{q_i} 
+ G^{-1} A_i G$ is independent of $z$ 
for all $1\le i\le s$. In particular $\hA_i$ is convergent and analytic 
near $q = \chi =0$. 
\end{theorem} 
\begin{proof} 
It is easy to see that such a gauge transformation $G$ exists 
uniquely over the formal power series ring $\C[\![\chi,z,q]\!]$; $G$ is given 
as a positive Birkhoff factor of a fundamental solution\footnote
{We use the nilpotence of $A_i|_{q=\chi=0}$ to construct 
a fundamental solution in the formal setting.} 
of the connection $\nabla$.  (See \cite{Coates-Givental}, 
\cite[Proposition 3.2]{Guest:D-mod} and  
\cite[Theorem 4.6]{Iritani:genmir} for the discussion in 
the context of quantum cohomology.) 
Note that the following argument also gives 
a recursive construction of $G$. 

We expand 
$A_i = \sum_{d\in \N^s} \sum_{k=0}^\infty 
A_{i;d,k}(\chi) q^d z^k$, 
$\hA_i = \sum_{d\in \N^s} \hA_{i;d}(\chi) q^d$ and 
$G=\sum_{d\in\N^s} \sum_{k=0}^\infty G_{d,k}(\chi) q^d z^k$. 
Note that we have $G_{0,k} = \delta_{k,0} I$, $A_{i;0,k} = \delta_{k,0} 
A_{i;0,0}$ and $\hA_{i;0} = \hA_i|_{q=0} = A_{i;0,0}$. 
Expanding the relation $zq_i\parfrac{G}{q_i} + A_i G =G \hA_i$, 
we obtain  
\[
 d_i G_{d,k-1}+  \sum_{d'+d''=d}\sum_{k'+k''=k} A_{i;d',k'} G_{d'',k''} 
=\sum_{d'+d''=d} G_{d',k} \hA_{i;d''}. 
\]  
This can be rewritten as follows: 
\begin{subequations} 
\begin{align} 
\label{eq:G_dk_recursion}
& d_i G_{d,k-1} + \ad(N_i) G_{d,k} = H_{i;d,k}  \qquad \text{for $k\ge 1$}  
\\ 
\label{eq:hA_d_recursion} 
& \hA_{i;d}  = \ad(N_i)G_{d,0}- H_{i;d,0} 
\end{align} 
\end{subequations} 
where we set $N_i = N_i(\chi):=A_{i;0,0}(\chi)$, $\ad(X) Y =XY-YX$, and 
\begin{equation} 
\label{eq:def_Hdk}
H_{i;d,k} := 
\sum_{\substack{d'+d'' = d \\ |d'|>0, |d''|>0}} G_{d',k} \hA_{i;d''} 
- \sum_{\substack{d'+d'' =d \\ |d'|>0}} \sum_{k'+k'' = k}A_{i;d',k'} G_{d'',k''}.  
\end{equation} 
Suppose that we know $G_{d,k}$ and $\hA_{i;d}$ 
for all $(i,d,k)$ with $|d|< e$ (for some $e$). 
This information determines $H_{i;d,k}$ for all $(i,d,k)$ with $|d|= e$. 
Then we can determine $G_{d,k}$ and $\hA_{i;d}$ recursively 
as follows: 
\begin{itemize} 
\item[(1)] we can solve for $G_{d,k}$ with $|d| = e$ 
for all $k$ using \eqref{eq:G_dk_recursion} 
-- see \eqref{eq:solve_G_dk} below; 

\item[(2)] next we can solve for $\hA_{i;d}$ with $|d|=e$ 
using \eqref{eq:hA_d_recursion}.  
\end{itemize} 
We need to give estimates for $G_{d,k}$, $\hA_{i,d}$. We set, 
for $(e,k) \in \N^2$, 
\begin{align*} 
a_{e,k}(\chi) &:= 
\max_{1\le i\le s} \sum_{|d|=e} \|A_{i;d,k}(\chi)\|,  
\quad  
g_{e,k}(\chi):= \sum_{|d| = e} \|G_{d,k}(\chi)\|,  
\\
h_{e,k}(\chi) &:=  \max_{1\le i\le s}
\sum_{|d|=e} \|H_{i;d,k}(\chi)\|, 
\quad 
\ha_{e}(\chi) := \max_{1\le i \le s} \sum_{|d|=e} \|\hA_{i;d}(\chi)\|, 
\end{align*} 
where $\|\cdot \|$ denotes the operator norm. 
By assumption there exist constants $C_1$,~$C_2$,~$\epsilon>0$ such that 
$a_{e,k}(\chi) \le e^k C_1 C_2^{e+k}$ for $|\chi|\le \epsilon$ 
(we set $e^k=1$ for $e=k=0$ as before). 
Let $C>0$ be a constant such that $\|N_i(\chi)\| =\|A_{i;0,0}(\chi)\| 
= \|\hA_{i;0}(\chi)\| \le C$ for all $1 \le i\le s$ and $|\chi| \le \epsilon$. 
Suppose by induction that 
\[
g_{e,k}(\chi) \le \frac{e^k}{(e+1)^M} B_1^e B_2^k \quad \text{and} \quad 
\ha_e(\chi) \le \frac{1}{(e+1)^M} C B_1^e 
\quad \text{whenever $|\chi|\le \delta$}
\]
for all $(e,k)$ with $e<e_0$. 
Here we set $B_2 := 2 C_2$ and the other positive constants 
$B_1$,~$M$,~$\delta>0$ are specified later; they depend only 
on $C$, $C_1$, $C_2$ and $N_i(\chi)$. 
We will choose $\delta$ so that $0< \delta\le \epsilon$. 
Note that this induction hypothesis holds for $e_0=1$ since 
$g_{0,k} = \|G_{0,k}(\chi)\| = \delta_{k,0}$ and 
$\ha_0 = \max_{1\le i\le s} \|N_i(\chi)\|\le C$. 
Under the induction hypothesis, we have 
from \eqref{eq:def_Hdk} that, whenever $|\chi|\le \delta$, 
\begin{align}
\label{eq:estimate_h}  
\begin{split} 
h_{e_0,k} & \le \sum_{0 < e < e_0} g_{e,k}  \ha_{e_0-e} 
+ \sum_{0<e\le e_0} \sum_{l=0}^k 
a_{e,l} g_{e_0-e,k-l} \\
& \le  \sum_{0<e<e_0} e^k \frac{B_1^e B_2^k}{(1+e)^M} 
\frac{CB_1^{e_0-e}}{(1+e_0-e)^M} 
+ \sum_{0<e\le e_0} \sum_{l=0}^k 
e^l C_1C_2^{e+l} 
\frac{(e_0-e)^{k-l}}{(1+e_0-e)^M}B_1^{e_0-e} B_2^{k-l} \\
& \le \left(C \epsilon_1(M)  +2C_1 
\epsilon_2(B_1,M) \right)\frac{e_0^k}{(1+e_0)^M}  
B_1^{e_0} B_2^k 
\end{split} 
\end{align}
where we omit $\chi$ from the notation and set 
\begin{align*} 
\epsilon_1(M) & := \sup_{e_0\ge 1} \sum_{0<e<e_0} 
\frac{(1+e_0)^M}{(1+e)^M(1+e_0-e)^M}, \\ 
\epsilon_2(B_1,M) & := \sup_{e_0\ge 1} 
\sum_{0<e\le e_0} \frac{(1+e_0)^M}{(1+e_0-e)^M} 
\left(\frac{C_2}{B_1}\right)^e.  
\end{align*} 
Next we estimate $g_{e_0,k}$. 
For $d\in \N^s$ with $|d|=e_0$, we choose $1\le i\le s$ 
such that $d_i = \max\{d_1,\dots, d_s\}$. 
Then we have from 
\eqref{eq:G_dk_recursion} that 
\begin{align} 
\label{eq:solve_G_dk}
\begin{split} 
G_{d,k} & = \frac{1}{d_i} H_{i;d,k} - 
\frac{\ad(N_i)}{d_i} G_{d,k+1} \\ 
& = \frac{1}{d_i} H_{i;d,k} - \frac{\ad(N_i)}{d_i^2} H_{i;d,k+1} 
+ \cdots + \frac{(-\ad(N_i))^{l}}{d_i^{l+1}} H_{i;d, k+l} 
+\cdots.  
\end{split} 
\end{align} 
Note that this infinite sum converges in the $\chi$-adic topology 
since $\ad(N_i|_{\chi=0})$ is nilpotent, and hence 
$G_{d,k}$'s are uniquely determined by \eqref{eq:G_dk_recursion}. 
We will see that this sum is convergent 
in the classical topology. 
Using $d_i\ge e_0/s$, we have 
\[
g_{e_0,k} \le \sum_{l=0}^\infty \left( \max_{1\le i\le s} \|\ad(N_i)^l\| \right) 
\left( \frac{s}{e_0} \right)^{l+1} h_{e_0,k+l}.  
\]
Since $N_i(0)=N_i|_{\chi=0}$ is nilpotent, we have $N_i(0)^{l_0}=0$ for some $l_0>0$; 
then we have $\ad(N_i(0))^{2l_0}=0$. 
This implies the following estimate: 
\[
\max_{1\le i\le s} \|\ad(N_i(\chi))^l\| \le 
C_3 (C_4|\chi|)^{\floor{l/(2l_0)}}  \quad 
\text{if $|\chi|\le \epsilon$} 
\] 
for some $C_3$,~$C_4>0$. Thus we get 
\[  
g_{e_0,k}  \le \sum_{l=0}^{2l_0-1} 
\sum_{j=0}^\infty C_3 C_4^j |\chi|^j 
\left(\frac{s}{e_0}\right)^{2l_0j +l +1} h_{e_0,k+2 l_0 j+l}.  
\] 
Using the estimate \eqref{eq:estimate_h} for $h_{e_0,k}$ 
and after some computations, we find that the infinite 
sum converges if $s^{2l_0} B_2^{2l_0} C_4 |\chi|<1$ and 
that 
\begin{equation} 
\label{eq:estimate_g} 
g_{e_0,k} \le \epsilon_3(B_1,M,\chi) \frac{e_0^k}{(1+e_0)^M} 
B_1^{e_0} B_2^k 
\end{equation} 
with 
\[
\epsilon_3(B_1,M,\chi) := 
\left(C \epsilon_1(M)  +2 C_1  
\epsilon_2(B_1,M) \right) 
\frac{C_5}{1- s^{2l_0}B_2^{2l_0} C_4 |\chi|}
\]
and $C_5 := s C_3 \sum_{l=0}^{2l_0-1} (sB_2)^l$. 
Finally we estimate $\ha_{e_0}$ by \eqref{eq:hA_d_recursion}. 
We have for $|\chi|\le\delta$, 
\begin{equation}
\label{eq:estimate_ha} 
\ha_{e_0} \le 2 C g_{e_0,0} + h_{e_0,0} 
\le \frac{ 
2C \epsilon_3(B_1,M,\chi) + C \epsilon_1(M) + 2 C_1 \epsilon_2(B_1,M)}
{(1+e_0)^M} B_1^{e_0} 
\end{equation} 
where we used \eqref{eq:estimate_h} and \eqref{eq:estimate_g} 
in the second inequality. 

We now specify the constants $B_1$,~$M$,~$\delta>0$ to complete the induction 
step. We use the following fact \cite[Lemma 4.9]{Iritani:coLef}: 
\[
\lim_{M\to \infty} \epsilon_1(M) = 0, \qquad 
\lim_{B_1 \to \infty} \epsilon_2(B_1,M) =0.  
\]
We can choose $B_1,M,\delta>0$ in the following way. 
\begin{itemize}
\item[(1)] Choose $\delta>0$ so that $\delta<\epsilon$ and 
$s^{2l_0}B_2^{2l_0} C_4 \delta<1/2$. 

\item[(2)] Choose $M>0$ so that $\epsilon_1(M)\le 1/3$ 
and $2C_5 C \epsilon_1(M)\le 1/12$.  
\item[(3)] Choose $B_1>0$ so that 
$2C_1 \epsilon_2(B_1,M) \le C/3$ 
and $4C_1C_5 \epsilon_2(B_1,M) \le 1/12$. 
\end{itemize} 
Then we have, when $|\chi|\le \delta$, 
\begin{gather*} 
\epsilon_3(B_1,M,\chi)  \le 2 C_5 C \epsilon_1(M) + 4 C_5 C_1 \epsilon_2(B_1,M) 
\le \frac{1}{6},  \\ 
2C \epsilon_3(B_1,M,\chi) + C \epsilon_1(M) + 2 C_1 \epsilon_2(B_1,M) 
\le \frac{1}{3} C + \frac{1}{3}C + \frac{1}{3} C \le C. 
\end{gather*} 
These inequalities together with the estimates \eqref{eq:estimate_g}, 
\eqref{eq:estimate_ha} complete the induction step. 
The theorem is proved. 
\end{proof} 

\begin{remark} 
Theorem \ref{thm:gauge_fixing} could be viewed as an analogue of 
Malgrange's theorem \cite[Theorem 4.3]{Malgrange:deformation}.  There a  
 similar analytification of a flat connection is discussed when the connection 
has no singularities at $q=0$. 
\end{remark} 

\subsection{Proof of Theorem \ref{thm:convergence} and 
Corollary \ref{cor:convergence}} 

First we prove Theorem \ref{thm:convergence}. 
Recall that the connection matrices $A_a$ of the Gauss--Manin connection 
are defined over $\cO^z$ (Lemma \ref{lem:GM_conn_Oz}) and 
the connection matrices of the quantum connection with respect to 
the basis $\{T_i\}$ 
are independent of $z$. 
The matrix $(\Theta_i^j(Q,y,\chi,z))$ gives a gauge transformation 
which transforms the Gauss--Manin connection into the quantum connection; 
it also satisfies $\Theta_i^j|_{q=0} = \delta_i^j$ since 
$\Theta(\Omega_i)|_{q=0} =\Loc^{(0)}(\Omega_i) =  T_i$. 
As remarked in the proof of Theorem \ref{thm:gauge_fixing}, 
the uniqueness of such a gauge transformation holds over 
the formal power series ring $\C[\![z,\chi,q]\!]$, and thus 
Theorem \ref{thm:gauge_fixing} shows that $\Theta_i^j(Q,y,\chi,z)$ 
belongs to $\cO^z$. Theorem \ref{thm:gauge_fixing} also shows 
that the pull-back 
$\tau^*\nabla$ of the quantum connection via $\tau$ is 
analytic in $(Q,y,\chi)$. 

Since $\Theta$ intertwines the Gauss--Manin connection with the 
quantum connection, we have 
\begin{equation} 
\label{eq:mirrormap_derivative} 
(\Theta \circ \nabla_{z q_a \parfrac{}{q_a}} \circ \Theta^{-1}) (1) 
= \nabla_{z q_a \parfrac{\tau}{q_a}} 1 
= q_a \parfrac{\tau}{q_a} 
\end{equation} 
where $1$ is the identity class in $H^*_{\CR,\T}(\frX)$. 
Since $\cO^z$ is a local ring (Lemma \ref{lem:Oz_local_ring}), the inverse matrix of $(\Theta_i^j)$ 
has coefficients in $\cO^z$. Thus the left-hand side can be 
written as an $\cO^z$-linear combination of $\{T_i\}$. 
Since the right-hand side is independent of $z$, this shows that 
$q_a \parfrac{\tau^i}{q_a}$ is analytic in $(q,\chi)$. 
Therefore the mirror map $\tau^i(Q,y,\chi)$ is analytic in 
all the arguments. 

Next we prove Corollary \ref{cor:convergence}. 
We claim that the mirror map $\tau$ is submersive 
when the extension $G$ is sufficiently large. 
Indeed, using the formula 
\eqref{eq:mirrormap_derivative}, we have for $\bk\in S$, 
\begin{align*} 
\left. \parfrac{\tau}{y_\bk} \right|_{Q=0,y=y^*} & = 
\left. (\Theta \circ \nabla_{z \parfrac{}{y_\bk}} \circ \Theta^{-1}) (1) 
\right|_{q=0} 
= \left. \Theta\left( \nabla_{z\parfrac{}{y_\bk}} (\omega + O(q) ) \right) 
\right|_{q=0} \\ 
& = \left. \Theta( w_\bk \omega +O(z) + O(q)) \right|_{q=0} 
=\phi_\bk + O(z) 
\end{align*} 
where we used $\Theta|_{q=0}=\Loc^{(0)}$ and the computation 
in Proposition \ref{pro:Loc0}. Since the left-hand side is independent of 
$z$, it is equal to $\phi_\bk$. 
Since finitely many $\phi_\bk$'s span $H_{\CR,\T}^*(\frX)$ over $R_\T$, 
the claim follows. 
We already showed that the pull-back $\tau^* \nabla$ of the quantum 
connection and the mirror map $\tau$ are analytic 
in a neighbourhood of $Q=0$,~$y=y^*$,~$\chi=0$. 
This immediately implies 
the analyticity of the big equivariant quantum product.

\bibliographystyle{amsplain}
\bibliography{toric_stacks_MS}

\def\cprime{$'$}
\providecommand{\bysame}{\leavevmode\hbox to3em{\hrulefill}\thinspace}
\providecommand{\MR}{\relax\ifhmode\unskip\space\fi MR }
\providecommand{\MRhref}[2]{%
  \href{http://www.ams.org/mathscinet-getitem?mr=#1}{#2}
}
\providecommand{\href}[2]{#2}
\begin{thebibliography}{10}

\bibitem{AGV:GW}
Dan Abramovich, Tom Graber, and Angelo Vistoli, \emph{Gromov-{W}itten theory of
  {D}eligne-{M}umford stacks}, Amer. J. Math. \textbf{130} (2008), no.~5,
  1337--1398. \MR{2450211 (2009k:14108)}

\bibitem{Acosta-Shoemaker:toric}
Pedro Acosta and Mark Shoemaker, \emph{Gromov-{W}itten theory of toric
  birational transformations},  (2016),
  \href{http://arxiv.org/abs/1604.03491}{\texttt{arXiv:1604.03491 [math.AG]}}.

\bibitem{Barannikov:projective}
Serguei Barannikov, \emph{Semi-infinite {H}odge structure and mirror symmetry
  for projective spaces},
  \href{http://arxiv.org/abs/math.AG/0010157}{\texttt{arXiv:math.AG/0010157}},
  2001.

\bibitem{Batyrev:qcoh_toric}
Victor~V. Batyrev, \emph{Quantum cohomology rings of toric manifolds},
  Ast\'erisque (1993), no.~218, 9--34, Journ{\'e}es de G{\'e}om{\'e}trie
  Alg{\'e}brique d'Orsay (Orsay, 1992).

\bibitem{BCS}
Lev~A. Borisov, Linda Chen, and Gregory~G. Smith, \emph{The orbifold {C}how
  ring of toric {D}eligne-{M}umford stacks}, J. Amer. Math. Soc. \textbf{18}
  (2005), no.~1, 193--215 (electronic). \MR{2114820 (2006a:14091)}

\bibitem{Borisov-Horja:bbGKZ}
Lev~A. Borisov and R.~Paul~Horja, \emph{On the better behaved version of the
  {GKZ} hypergeometric system}, Math. Ann. \textbf{357} (2013), no.~2,
  585--603. \MR{3096518}

\bibitem{BMO:Springer}
Alexander Braverman, Davesh Maulik, and Andrei Okounkov, \emph{Quantum
  cohomology of the {S}pringer resolution}, Adv. Math. \textbf{227} (2011),
  no.~1, 421--458. \MR{2782198 (2012h:14133)}

\bibitem{Brown:toric_fibration}
Jeff Brown, \emph{Gromov-{W}itten invariants of toric fibrations}, Int. Math.
  Res. Not. IMRN (2014), no.~19, 5437--5482,
  \href{http://arxiv.org/abs/0901.1290}{\texttt{arXiv:0901.1290 [math.AG]}}.
  \MR{3267376}

\bibitem{CLLT:Seidel}
Kwokwai Chan, Siu-Cheong Lau, Naichung-Conan Leung, and Hsian-Hua Tseng,
  \emph{Open {G}romov-{W}itten invariants, mirror maps, and {S}eidel
  representations for toric manifolds},
  \href{http://arxiv.org/abs/1209.6119}{\texttt{arXiv:1209.6119}}, 2012.

\bibitem{Chen-Ruan:orbGW}
Weimin Chen and Yongbin Ruan, \emph{Orbifold {G}romov-{W}itten theory},
  Orbifolds in mathematics and physics ({M}adison, {WI}, 2001), Contemp. Math.,
  vol. 310, Amer. Math. Soc., Providence, RI, 2002, pp.~25--85. \MR{1950941
  (2004k:53145)}

\bibitem{Chen-Ruan:new_coh}
\bysame, \emph{A new cohomology theory of orbifold}, Comm. Math. Phys.
  \textbf{248} (2004), no.~1, 1--31. \MR{2104605 (2005j:57036)}

\bibitem{Cheong-CF-Kim}
Daewoong Cheong, Ionu{\c{t}} Ciocan-Fontanine, and Bumsig Kim, \emph{Orbifold
  quasimap theory}, Math. Ann. \textbf{363} (2015), no.~3-4, 777--816,
  \href{http://arxiv.org/abs/1405.7160}{\texttt{arXiv:1405.7160 [math.AG]}}.
  \MR{3412343}

\bibitem{CCIT:computing}
Tom Coates, Alessio Corti, Hiroshi Iritani, and Hsian-Hua Tseng,
  \emph{Computing genus-zero twisted {G}romov-{W}itten invariants}, Duke Math.
  J. \textbf{147} (2009), no.~3, 377--438. \MR{2510741 (2010a:14090)}

\bibitem{CCIT:applications}
\bysame, \emph{Some applications of the mirror theorem for toric stacks},
  (2014), \href{http://arxiv.org/abs/1401.2611}{\texttt{arXiv:1401.2611
  [math.AG]}}.

\bibitem{CCIT:mirrorthm}
\bysame, \emph{A mirror theorem for toric stacks}, Compos. Math. \textbf{151}
  (2015), no.~10, 1878--1912,
  \href{http://arxiv.org/abs/1310.4163}{\texttt{arXiv:1310.4163 [math.AG]}}.
  \MR{3414388}

\bibitem{CCLT:wp}
Tom Coates, Alessio Corti, Yuan-Pin Lee, and Hsian-Hua Tseng, \emph{The quantum
  orbifold cohomology of weighted projective spaces}, Acta Math. \textbf{202}
  (2009), no.~2, 139--193. \MR{2506749 (2010f:53155)}

\bibitem{Coates-Givental}
Tom Coates and Alexander Givental, \emph{Quantum {R}iemann-{R}och, {L}efschetz
  and {S}erre}, Ann. of Math. (2) \textbf{165} (2007), no.~1, 15--53.
  \MR{2276766 (2007k:14113)}

\bibitem{Coates-Iritani:Fock}
Tom Coates and Hiroshi Iritani, \emph{A {F}ock sheaf for {G}ivental
  quantization},  (2014),
  \href{http://arxiv.org/abs/1411.7039}{\texttt{arXiv:1411.7039 [math.AG]}}.

\bibitem{CIJ}
Tom Coates, Hiroshi Iritani, and Yunfeng Jiang, \emph{The crepant
  transformation conjecture for toric complete intersections},  (2014),
  \href{http://arxiv.org/abs/1410.0024}{\texttt{arXiv:1410.0024 [math.AG]}}.

\bibitem{CIT:wallcrossing}
Tom Coates, Hiroshi Iritani, and Hsian-Hua Tseng, \emph{Wall-crossings in toric
  {G}romov-{W}itten theory. {I}. {C}repant examples}, Geom. Topol. \textbf{13}
  (2009), no.~5, 2675--2744. \MR{2529944 (2010i:53173)}

\bibitem{CLS}
David~A. Cox, John~B. Little, and Henry~K. Schenck, \emph{Toric varieties},
  Graduate Studies in Mathematics, vol. 124, American Mathematical Society,
  Providence, RI, 2011. \MR{2810322 (2012g:14094)}

\bibitem{deGregorio-Mann}
Ignacio de~Gregorio and {\'E}tienne Mann, \emph{Mirror fibrations and root
  stacks of weighted projective spaces}, Manuscripta Math. \textbf{127} (2008),
  no.~1, 69--80. \MR{2429914}

\bibitem{DKK:symplectomorphism}
Colin Diemer, Ludmil Katzarkov, and Gabriel Kerr, \emph{Symplectomorphism group
  relations and degenerations of {L}andau-{G}inzburg models},
  \href{http://arxiv.org/abs/1204.2233}{\texttt{arXiv:1204.2233 [math.AG]}},
  2012.

\bibitem{NIST:DLMF}
\emph{{NIST Digital Library of Mathematical Functions}}, http://dlmf.nist.gov/,
  Release 1.0.10 of 2015-08-07, Online companion to \cite{Olver:2010:NHMF}.

\bibitem{Douai-Mann:wp}
Antoine Douai and Etienne Mann, \emph{The small quantum cohomology of a
  weighted projective space, a mirror {$D$}-module and their classical limits},
  Geom. Dedicata \textbf{164} (2013), 187--226. \MR{3054624}

\bibitem{Douai-Sabbah:I}
Antoine Douai and Claude Sabbah, \emph{Gauss-{M}anin systems, {B}rieskorn
  lattices and {F}robenius structures. {I}}, Proceedings of the {I}nternational
  {C}onference in {H}onor of {F}r\'ed\'eric {P}ham ({N}ice, 2002), vol.~53,
  2003, pp.~1055--1116. \MR{2033510}

\bibitem{Douai-Sabbah:II}
\bysame, \emph{Gauss-{M}anin systems, {B}rieskorn lattices and {F}robenius
  structures. {II}}, Frobenius manifolds, Aspects Math., E36, Vieweg,
  Wiesbaden, 2004, pp.~1--18. \MR{2115764 (2006e:32037)}

\bibitem{Dubrovin:2DTFT}
Boris Dubrovin, \emph{Geometry of {$2$}{D} topological field theories},
  Integrable systems and quantum groups ({M}ontecatini {T}erme, 1993), Lecture
  Notes in Math., vol. 1620, Springer, Berlin, 1996, pp.~120--348. \MR{1397274
  (97d:58038)}

\bibitem{FMN}
Barbara Fantechi, Etienne Mann, and Fabio Nironi, \emph{Smooth toric
  {D}eligne-{M}umford stacks}, J. Reine Angew. Math. \textbf{648} (2010),
  201--244. \MR{2774310 (2012b:14097)}

\bibitem{FOOO:toricI}
Kenji Fukaya, Yong-Geun Oh, Hiroshi Ohta, and Kaoru Ono, \emph{Lagrangian
  {F}loer theory on compact toric manifolds. {I}}, Duke Math. J. \textbf{151}
  (2010), no.~1, 23--174. \MR{2573826 (2011d:53220)}

\bibitem{GKZ:hypergeom}
I.~M. Gel{\cprime}fand, A.~V. Zelevinski{\u\i}, and M.~M. Kapranov,
  \emph{Hypergeometric functions and toric varieties}, Funktsional. Anal. i
  Prilozhen. \textbf{23} (1989), no.~2, 12--26. \MR{1011353}

\bibitem{Givental:ICM}
Alexander Givental, \emph{Homological geometry and mirror symmetry},
  Proceedings of the {I}nternational {C}ongress of {M}athematicians, {V}ol.\ 1,
  2 ({Z}\"urich, 1994), Birkh\"auser, Basel, 1995, pp.~472--480. \MR{1403947
  (97j:58013)}

\bibitem{Givental:homological}
\bysame, \emph{Homological geometry. {I}. {P}rojective hypersurfaces}, Selecta
  Math. (N.S.) \textbf{1} (1995), no.~2, 325--345. \MR{1354600 (97c:14052)}

\bibitem{Givental:fixedpoint_toric}
\bysame, \emph{A symplectic fixed point theorem for toric manifolds}, The
  {F}loer memorial volume, Progr. Math., vol. 133, Birkh\"auser, Basel, 1995,
  pp.~445--481. \MR{1362837}

\bibitem{Givental:equivariant}
\bysame, \emph{Equivariant {G}romov-{W}itten invariants}, Internat. Math. Res.
  Notices (1996), no.~13, 613--663. \MR{1408320 (97e:14015)}

\bibitem{Givental:elliptic}
\bysame, \emph{Elliptic {G}romov-{W}itten invariants and the generalized mirror
  conjecture}, Integrable systems and algebraic geometry ({K}obe/{K}yoto,
  1997), World Sci. Publ., River Edge, NJ, 1998, pp.~107--155. \MR{1672116
  (2000b:14074)}

\bibitem{Givental:toric_mirrorthm}
\bysame, \emph{A mirror theorem for toric complete intersections}, Topological
  field theory, primitive forms and related topics ({K}yoto, 1996), Progr.
  Math., vol. 160, Birkh\"auser Boston, Boston, MA, 1998, pp.~141--175.
  \MR{1653024 (2000a:14063)}

\bibitem{Givental:quadratic}
\bysame, \emph{Gromov-{W}itten invariants and quantization of quadratic
  {H}amiltonians}, Mosc. Math. J. \textbf{1} (2001), no.~4, 551--568, 645,
  Dedicated to the memory of I. G. Petrovskii on the occasion of his 100th
  anniversary. \MR{1901075 (2003j:53138)}

\bibitem{Givental:symplectic}
\bysame, \emph{Symplectic geometry of {F}robenius structures}, Frobenius
  manifolds, Aspects Math., E36, Friedr. Vieweg, Wiesbaden, 2004, pp.~91--112.
  \MR{2115767 (2005m:53172)}

\bibitem{Gonzalez-Iritani:Selecta}
Eduardo Gonz{\'a}lez and Hiroshi Iritani, \emph{Seidel elements and mirror
  transformations}, Selecta Math. (N.S.) \textbf{18} (2012), no.~3, 557--590.
  \MR{2960027}

\bibitem{Gonzalez-Woodward:tmmp}
Eduardo Gonz{\'a}lez and Chris Woodward, \emph{Quantum cohomology and toric
  minimal model programs},  (2012),
  \href{http://arxiv.org/abs/1207.3253}{\texttt{arXiv:1010.2118}}.

\bibitem{Graber-Pandharipande}
Tom Graber and Rahul Pandharipande, \emph{Localization of virtual classes},
  Invent. Math. \textbf{135} (1999), no.~2, 487--518. \MR{1666787
  (2000h:14005)}

\bibitem{Gross:tropical_P2}
Mark Gross, \emph{Mirror symmetry for {$\Bbb P^2$} and tropical geometry}, Adv.
  Math. \textbf{224} (2010), no.~1, 169--245. \MR{2600995 (2011j:14089)}

\bibitem{Gross:tropical_book}
\bysame, \emph{Tropical geometry and mirror symmetry}, CBMS Regional Conference
  Series in Mathematics, vol. 114, Published for the Conference Board of the
  Mathematical Sciences, Washington, DC; by the American Mathematical Society,
  Providence, RI, 2011. \MR{2722115}

\bibitem{Guest:D-mod}
Martin~A. Guest, \emph{Quantum cohomology via {$D$}-modules}, Topology
  \textbf{44} (2005), no.~2, 263--281. \MR{2114708 (2005j:53105)}

\bibitem{Hori-Vafa}
Kentaro Hori and Cumrum Vafa, \emph{Mirror symmetry},
  \href{http://arxiv.org/abs/hep-th/0002222}{\texttt{arXiv:hep-th/0002222}},
  2000.

\bibitem{Iritani:efc}
Hiroshi Iritani, \emph{Quantum {$D$}-modules and equivariant {F}loer theory for
  free loop spaces}, Math. Z. \textbf{252} (2006), no.~3, 577--622. \MR{2207760
  (2007e:53118)}

\bibitem{Iritani:coLef}
\bysame, \emph{Convergence of quantum cohomology by quantum {L}efschetz}, J.
  Reine Angew. Math. \textbf{610} (2007), 29--69.

\bibitem{Iritani:genmir}
\bysame, \emph{Quantum {$D$}-modules and generalized mirror transformations},
  Topology \textbf{47} (2008), no.~4, 225--276. \MR{2416770 (2009a:53153)}

\bibitem{Iritani:integral}
\bysame, \emph{An integral structure in quantum cohomology and mirror symmetry
  for toric orbifolds}, Adv. Math. \textbf{222} (2009), no.~3, 1016--1079.
  \MR{2553377 (2010j:53182)}

\bibitem{Iritani:periods}
\bysame, \emph{Quantum cohomology and periods}, Ann. Inst. Fourier (Grenoble)
  \textbf{61} (2011), no.~7, 2909--2958. \MR{3112512}

\bibitem{Iritani:shift_mirror}
\bysame, \emph{A mirror construction for the big equivariant quantum cohomology
  of toric manifolds},  (2015),
  \href{http://arxiv.org/abs/1503.02919}{\texttt{arXiv:1503.02919 [math.AG]}}.

\bibitem{Iwanari1}
Isamu Iwanari, \emph{The category of toric stacks}, Compos. Math. \textbf{145}
  (2009), no.~3, 718--746. \MR{2507746 (2011b:14113)}

\bibitem{Iwanari2}
\bysame, \emph{Logarithmic geometry, minimal free resolutions and toric
  algebraic stacks}, Publ. Res. Inst. Math. Sci. \textbf{45} (2009), no.~4,
  1095--1140. \MR{2597130 (2011f:14084)}

\bibitem{Jiang-Tseng}
Yunfeng Jiang and Hsian-Hua Tseng, \emph{Note on orbifold {C}how ring of
  semi-projective toric {D}eligne-{M}umford stacks}, Comm. Anal. Geom.
  \textbf{16} (2008), no.~1, 231--250. \MR{2411474 (2009h:14092)}

\bibitem{Kashiwara:microlocal}
Masaki Kashiwara, \emph{Microlocal analysis}, Proceedings of the
  {I}nternational {C}ongress of {M}athematicians ({H}elsinki, 1978), Acad. Sci.
  Fennica, Helsinki, 1980, pp.~139--150. \MR{562603}

\bibitem{KKP}
Ludmil Katzarkov, Maxim Kontsevich, and Tony Pantev, \emph{Hodge theoretic
  aspects of mirror symmetry}, From {H}odge theory to integrability and {TQFT}
  tt*-geometry, Proc. Sympos. Pure Math., vol.~78, Amer. Math. Soc.,
  Providence, RI, 2008, pp.~87--174. \MR{2483750 (2009j:14052)}

\bibitem{CCLiu:localization}
Chiu-Chu~Melissa Liu, \emph{Localization in {G}romov-{W}itten theory and
  orbifold {G}romov-{W}itten theory}, Handbook of moduli. {V}ol. {II}, Adv.
  Lect. Math. (ALM), vol.~25, Int. Press, Somerville, MA, 2013, pp.~353--425.
  \MR{3184181}

\bibitem{Malgrange:deformation}
B.~Malgrange, \emph{D\'eformations de syst\`emes diff\'erentiels et
  microdiff\'erentiels}, Mathematics and physics ({P}aris, 1979/1982), Progr.
  Math., vol.~37, Birkh\"auser Boston, Boston, MA, 1983, pp.~353--379.
  \MR{728429}

\bibitem{Mann:wP}
Etienne Mann, \emph{Orbifold quantum cohomology of weighted projective spaces},
  J. Algebraic Geom. \textbf{17} (2008), no.~1, 137--166. \MR{2357682}

\bibitem{Mann-Reichelt}
Etienne Mann and Thomas Reichelt, \emph{Logarithmic degenerations of
  landau-ginzburg models for toric orbifolds and global $tt^*$ geometry},
  \href{http://arxiv.org/abs/1605.08937}{\texttt{arXiv:1605.08937 [math.AG]}},
  2016.

\bibitem{McDuff-Tolman}
Dusa McDuff and Susan Tolman, \emph{Topological properties of {H}amiltonian
  circle actions}, IMRP Int. Math. Res. Pap. (2006), 72826, 1--77.

\bibitem{Mochizuki_T:twistor_GKZ}
Takuro Mochizuki, \emph{Twistor property of {GKZ}-hypergeometric systems},
  (2015), \href{http://arxiv.org/abs/1501.04146}{\texttt{arXiv:1501.04146}}.

\bibitem{Olver:2010:NHMF}
F.~W.~J. Olver, D.~W. Lozier, R.~F. Boisvert, and C.~W. Clark (eds.),
  \emph{{NIST Handbook of Mathematical Functions}}, Cambridge University Press,
  New York, NY, 2010, Print companion to \cite{NIST:DLMF}.

\bibitem{Pandharipande:afterGivental}
Rahul Pandharipande, \emph{Rational curves on hypersurfaces (after {A}.
  {G}ivental)}, Ast\'erisque (1998), no.~252, Exp.\ No.\ 848, 5, 307--340,
  S{\'e}minaire Bourbaki. Vol. 1997/98. \MR{1685628 (2000e:14094)}

\bibitem{Pham:Lefschetz}
Fr{\'e}d{\'e}ric Pham, \emph{La descente des cols par les onglets de
  {L}efschetz, avec vues sur {G}auss-{M}anin}, Ast\'erisque (1985), no.~130,
  11--47, Differential systems and singularities (Luminy, 1983). \MR{804048
  (87h:32017)}

\bibitem{Pressley-Segal}
Andrew Pressley and Graeme Segal, \emph{Loop groups}, Oxford Mathematical
  Monographs, The Clarendon Press, Oxford University Press, New York, 1986,
  Oxford Science Publications. \MR{900587 (88i:22049)}

\bibitem{Reichelt-Sevenheck:logFrob}
Thomas Reichelt and Christian Sevenheck, \emph{Logarithmic {F}robenius
  manifolds, hypergeometric systems and quantum d-modules},  (2010),
  \href{http://arxiv.org/abs/1010.2118}{\texttt{arXiv:1010.2118}}.

\bibitem{Reichelt-Sevenheck:nonaffine}
\bysame, \emph{Non-affine {L}andau-{G}inzburg models and intersection
  cohomology}, \href{http://arxiv.org/abs/1210.6527}{\texttt{arXiv:1210.6527
  [math.AG]}}, 2012.

\bibitem{Sabbah:tame}
Claude Sabbah, \emph{Hypergeometric period for a tame polynomial}, C. R. Acad.
  Sci. Paris S\'er. I Math. \textbf{328} (1999), no.~7, 603--608, A longer
  version published in: Port.\ Math.\ (N.S.) 63 (2006), no.2, 173--226.
  \MR{1679978 (2000b:32053)}

\bibitem{SaitoK:higherresidue}
Kyoji Saito, \emph{The higher residue pairings {$K_{F}^{(k)}$} for a family of
  hypersurface singular points},  \textbf{40} (1983), 441--463. \MR{713270
  (85d:32043)}

\bibitem{SaitoK:primitiveform}
\bysame, \emph{Period mapping associated to a primitive form}, Publ. Res. Inst.
  Math. Sci. \textbf{19} (1983), no.~3, 1231--1264. \MR{723468 (85h:32034)}

\bibitem{SaitoM:Brieskorn}
Morihiko Saito, \emph{On the structure of {B}rieskorn lattice}, Ann. Inst.
  Fourier (Grenoble) \textbf{39} (1989), no.~1, 27--72. \MR{1011977
  (91i:32035)}

\bibitem{Seidel:pi1}
Paul Seidel, \emph{{$\pi_1$} of symplectic automorphism groups and invertibles
  in quantum homology rings}, Geom. Funct. Anal. \textbf{7} (1997), no.~6,
  1046--1095,
  \href{http://arxiv.org/abs/dg-ga/9511011}{\texttt{arXiv:dg-ga/9511011}}.
  \MR{1487754 (99b:57068)}

\bibitem{Tseng:QRR}
Hsian-Hua Tseng, \emph{Orbifold quantum {R}iemann-{R}och, {L}efschetz and
  {S}erre}, Geom. Topol. \textbf{14} (2010), no.~1, 1--81. \MR{2578300
  (2011c:14147)}

\bibitem{Zariski-Samuel}
Oscar Zariski and Pierre Samuel, \emph{Commutative algebra. {V}ol. {II}},
  Springer-Verlag, New York-Heidelberg, 1975, Reprint of the 1960 edition,
  Graduate Texts in Mathematics, Vol. 29. \MR{0389876 (52 \#10706)}

\end{thebibliography}

\end{document}